\documentclass[10pt,numbers,sort]{elsarticle} 
\usepackage{amssymb,amsfonts,amsmath,mathrsfs,amsbsy}
\usepackage{graphicx,subfigure}
\usepackage{bm}
\usepackage[margin=1in,papersize={8.5in,11in}]{geometry}
\usepackage[ruled,vlined]{algorithm2e}
\usepackage{amsthm}

\usepackage{xcolor}
\usepackage[colorlinks=true, linkcolor=blue, urlcolor=blue, citecolor = blue]{hyperref}

\newtheorem{lemma}{\bf Lemma}
\newtheorem{theorem}{\bf Theorem}
\newtheorem{remark}{\bf Remark}

\journal{ArXiv}

\renewenvironment{proof}{{\noindent \em Proof.}}{}
\theoremstyle{definition}

\date{}
\begin{document}

\begin{frontmatter}

\title{
Implicit integration of nonlinear evolution equations on tensor manifolds
\thanks{
This research was supported by the U.S. Army Research Office 
grant W911NF1810309, and by the U.S. Air Force 
Office of Scientific Research grant FA9550-20-1-0174.
}}

\author[ucsc]{Abram Rodgers}
\ead{akrodger@ucsc.edu}
\author[ucsc]{Daniele Venturi\corref{correspondingAuthor}}
\address[ucsc]{Department of Applied Mathematics\\
University of California Santa Cruz\\ Santa Cruz, CA 95064}
\cortext[correspondingAuthor]{Corresponding author}
\ead{venturi@ucsc.edu}
\begin{abstract}
Explicit step-truncation tensor methods have recently 
proven successful in integrating initial 
value problems for high-dimensional partial 
differential equations (PDEs). 
However, the combination of non-linearity and 
stiffness may introduce time-step restrictions which 
could make explicit integration computationally 
infeasible.
To overcome this problem, we develop a new class of 
implicit rank-adaptive algorithms for temporal 
integration of nonlinear evolution equations on tensor manifolds.
These algorithms are based on performing 
one time step with a conventional time-stepping 
scheme, followed by an implicit fixed point 
iteration step involving a rank-adaptive truncation 
operation onto a tensor manifold.
Implicit step truncation methods are straightforward 
to implement as they rely only on arithmetic operations 
between tensors, which can be performed by 
efficient and scalable parallel algorithms.
Numerical applications demonstrating the effectiveness 
of implicit step-truncation tensor integrators are 
presented and discussed for the Allen-Cahn equation,
the Fokker-Planck equation, and the nonlinear
Schr\"odinger equation.
\end{abstract}

\end{frontmatter}


\noindent
\section*{Introduction}
High-dimensional nonlinear evolution equations  
of the form 
\begin{align}
\frac{\partial f({\bm x},t) }{\partial t} = 
{\cal N}\left(f({\bm x},t),{\bm x}\right), \qquad 
f({\bm x},0) = f_0({\bm x}),
\label{nonlinear-ibvp} 
\end{align}
arise in many areas of mathematical physics, e.g., 
in statistical mechanics \cite{Beran,risken1989}, 
quantum field theory \cite{Justin}, and 
in the approximation of functional differential equations 
(infinite-dimensional PDEs) \cite{VenturiSpectral,venturi2018} 
such as the Hopf equation of turbulence \cite{Hopf}, 
or functional equations modeling deep learning \cite{Weinan2019}.
In equation \eqref{nonlinear-ibvp}, $f:  \Omega \times [0,T] \to\mathbb{R}$ 
is a $d$-dimensional time-dependent scalar field 
defined on the domain $\Omega\subseteq \mathbb{R}^d$ 
($d\geq 2$), $T$ is the period of integration, 
and $\cal N$ is a nonlinear operator which may depend 
on the variables ${\bm x}=(x_1,\ldots,x_d)\in \Omega$, and may 
incorporate boundary conditions.
For simplicity, we assume that the domain $\Omega$ 
is a Cartesian product of $d$ one-dimensional 
domains $\Omega_i$
\begin{equation}
\Omega=\Omega_1\times\cdots\times \Omega_d,
\end{equation}
and that $f$ is an element of a Hilbert 
space $H(\Omega;[0,T])$. In these hypotheses, we can leverage 
the isomorphism  $H(\Omega;[0,T])\simeq H([0,T])\otimes H(\Omega_1)\otimes \cdots \otimes H(\Omega_d)$ and represent the 
solution of \eqref{nonlinear-ibvp} as 
\begin{equation}
f(\bm x,t) \approx \sum_{i_1=1}^{n_1}\cdots
\sum_{i_d=1}^{n_d}f_{i_1\ldots i_d}(t) 
\phi_{i_1}(x_1)\cdots \phi_{i_1}(x_1),
\label{tt}
\end{equation}
where $\phi_{i_j}(x_j)$ are one-dimensional 
orthonormal basis functions of $H(\Omega_i)$. 
Substituting \eqref{tt} into \eqref{nonlinear-ibvp} and  
projecting onto an appropriate finite-dimensional subspace 
of $H(\Omega)$ yields the semi-discrete form 
\begin{equation}
\label{eqn:ode}
\frac{d{\bm f}}{d t}
=
{\bm G}({\bm f}),\qquad {\bm f}(0) = {\bm f}_0
\end{equation}
where 
${\bm f}:[0,T]\rightarrow
{\mathbb R}^{n_1\times n_2\times \dots \times n_d}$ is 
a multivariate array with coefficients 
$f_{i_1\ldots i_d}(t)$, and $\bm G$ 
is the finite-dimensional representation of the 
nonlinear operator $\cal N$.
The number of degrees of freedom associated 
with the solution to the Cauchy problem \eqref{eqn:ode} 
is $N_{\text{dof}}=n_1 n_2 \cdots n_d$ at each 
time $t\geq 0$, which can 
be extremely large even for 
moderately small dimension $d$. For instance, 
the solution of the Boltzmann-BGK equation
on a six-dimensional ($d=6$) flat torus \cite{Jingmei2022,BoltzmannBGK2020,dimarco2014}  
with $n_i=128$ basis functions in each position and momentum variable  
yields $N_{\text{dof}}=128^6=4398046511104$ degrees
of freedom at each time $t$.
This requires approximately $35.18$ Terabytes 
per temporal snapshot if we store the solution 
tensor $\bm f$ in a double precision 
IEEE 754 floating point format.
Several general-purpose algorithms have been developed 
to mitigate such an exponential growth of degrees of freedom, 
the computational cost, and the memory requirements. 
These algorithms include, e.g., sparse 
collocation methods  \cite{Bungartz,Barthelmann,Foo1,Akil}, 
high-dimensional model representation (HDMR)
\cite{Li1,CaoCG09,Baldeaux}, and techniques based on
deep neural networks \cite{Raissi,Raissi1,Zhu2019,chen2020deephf}.

In a parallel research effort that has its roots in quantum field 
theory and quantum entanglement, researchers have recently developed a new 
generation of algorithms based on tensor networks and low-rank 
tensor techniques to compute the solution of high-dimensional PDEs \cite{khoromskij,Bachmayr,parr_tensor,cho2016,Vandereycken_2019}. 
Tensor networks are essentially factorizations of entangled 
objects such as multivariate functions or operators, into 
networks of simpler objects which are amenable to efficient 
representation and computation. The process of building a 
tensor network relies on a hierarchical decomposition that can be visualized in terms of trees, and has its roots in the spectral theory for linear operators. Such rigorous mathematical foundations can be leveraged to construct high-order methods to compute the numerical solution of high-dimensional Cauchy problems of the form \eqref{eqn:ode} at a cost that scales linearly with respect to the dimension $d$, and polynomially with respect to the tensor rank.

In particular, a new class of algorithms to integrate \eqref{eqn:ode} on 
a low-rank tensor manifold was recently proposed in \cite{rodgers2020stability,Alec2020,rodgers2020adaptive,dektor2020dynamically,Vandereycken_2019,venturi2018}.
These algorithms are known as explicit step-truncation methods 
and they are based on integrating the solution ${\bm f}(t)$ off the 
tensor manifold for a short time using any conventional explicit 
time-stepping scheme, and then mapping it back onto the manifold 
using a tensor truncation operation (see Figure \ref{fig:intro-surf-compare}).  
To briefly describe these methods, let us 
discretize the ODE \eqref{eqn:ode} in time with a 
one-step method on an evenly-spaced temporal grid as 
\begin{equation}
\label{eqn:discrete-flow-intro}
{\bm f}_{k+1} = 
{\bm \Psi}_{\Delta t}({\bm G}, {\bm f}_{k}),
\qquad {\bm f}_{0}={\bm f}(0),
\end{equation}
where ${\bm f}_{k}$ denotes an approximation of 
${\bm f}(k\Delta t)$ for $k=0,1,\ldots$, 
and ${\bm \Psi}_{\Delta t}$ is an increment function.
To obtain a step-truncation integrator, we simply apply 
a truncation operator $\mathfrak{T}_{\bm r}(\cdot)$, i.e., 
a nonlinear projection onto a tensor manifold ${\cal H}_{\bm r}$
with multilinear rank $\bm r$ \cite{uschmajew2013geometry} 
to the scheme \eqref{eqn:discrete-flow-intro}. This yields 
\begin{equation}
\label{eqn:discrete-flow-trunc-intro}
{\bm f}_{k+1} = 
{\mathfrak T}_{\bm r}\left(
{\bm \Psi}_{\Delta t}({\bm G}, {\bm f}_{k})\right).
\end{equation}
The need for tensor rank-reduction when iterating 
\eqref{eqn:discrete-flow-intro} can be easily understood 
by noting that tensor operations such as the application of 
an operator to a tensor and the addition between 
two tensors naturally increase tensor rank
\cite{kressner2014algorithm}. 
Hence, iterating \eqref{eqn:discrete-flow-trunc-intro} with no rank 
reduction can yield a fast increase 
in tensor rank, which, in turn, can tax 
computational resources heavily.

Explicit step-truncation algorithms of the 
form \eqref{eqn:discrete-flow-trunc-intro} 
were studied extensively in 
\cite{rodgers2020adaptive,Vandereycken_2019}. In particular, 
error estimates and convergence results were obtained for 
both fixed-rank and rank-adaptive integrators, i.e., 
integrators in which the tensor rank $\bm r$
is selected at each time step based on accuracy
and stability constraints. Step-truncation methods 
are very simple to implement as they rely only 
on arithmetic operations between tensors, 
which can be performed by scalable parallel algorithms 
\cite{daas2020parallel,townsend,AuBaKo16,grasedyck2018distributed}.

While explicit step-truncation methods have proven 
successful in integrating a wide variety of 
high-dimensional initial value problems, their 
effectiveness for {stiff problems} 
is limited. Indeed, the combination of 
non-linearity and stiffness may introduce 
time-step restrictions which could make explicit 
step-truncation integration computationally infeasible. 
As an example, in Figure \ref{fig:intro-stiff-compare} 
we show that the explicit step-truncation midpoint 
method applied to the Allen-Cahn equation
\begin{equation}
\label{eqn:allen-cahn-intro}
\frac{\partial f}{\partial t}
= \varepsilon \Delta f +f - f^3, 
\end{equation}
undergoes a numerical instability for 
$\Delta t=10^{-3}$. 
\begin{figure}[t!]
\centerline{\hspace{0.9cm}
\footnotesize 
Explicit step-truncation midpoint method \hspace{1.9cm}
Implicit step-truncation midpoint method\hspace{0.7cm}}
\begin{center}
\includegraphics[scale=0.51]{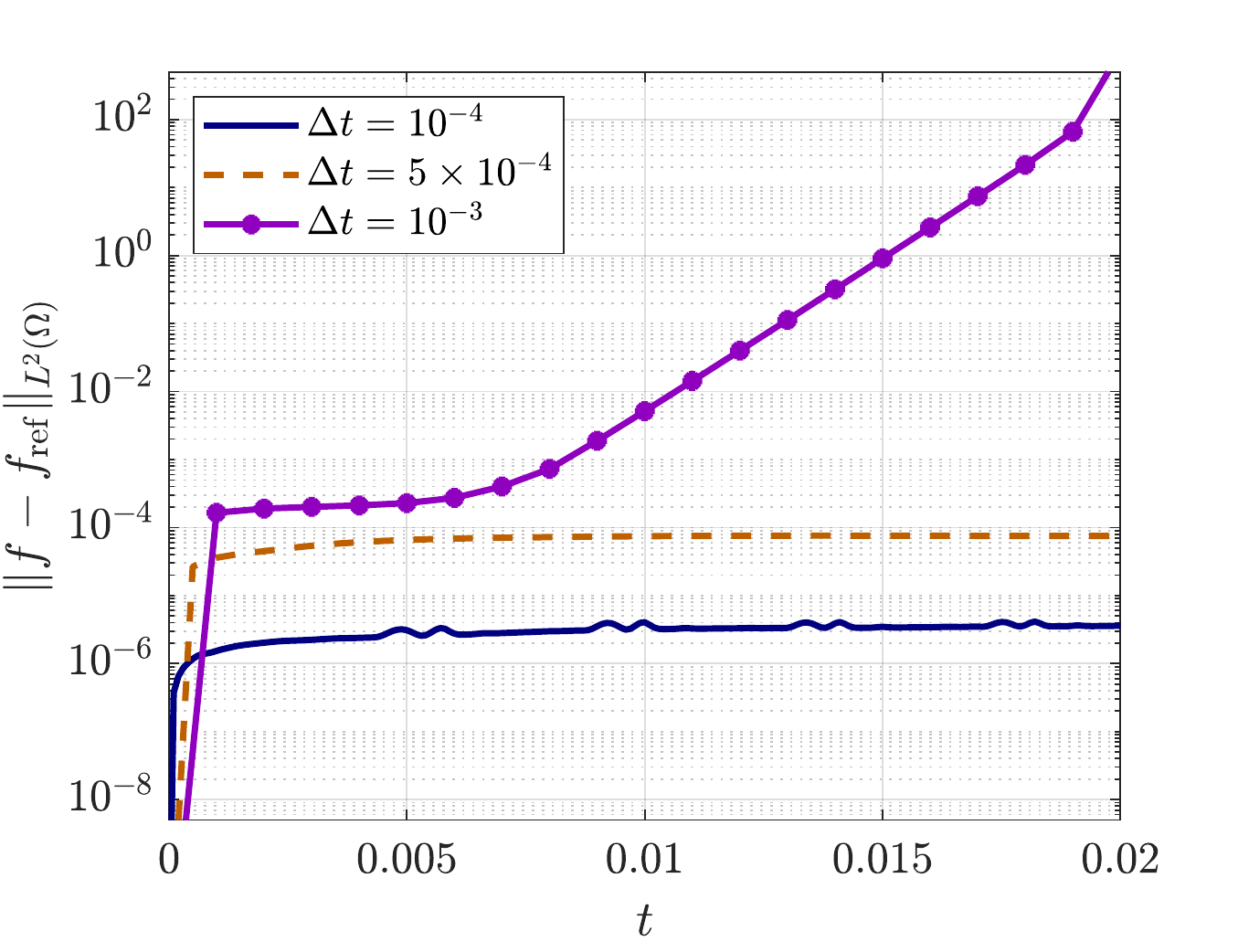}\hspace{0.5cm}
\includegraphics[scale=0.51]{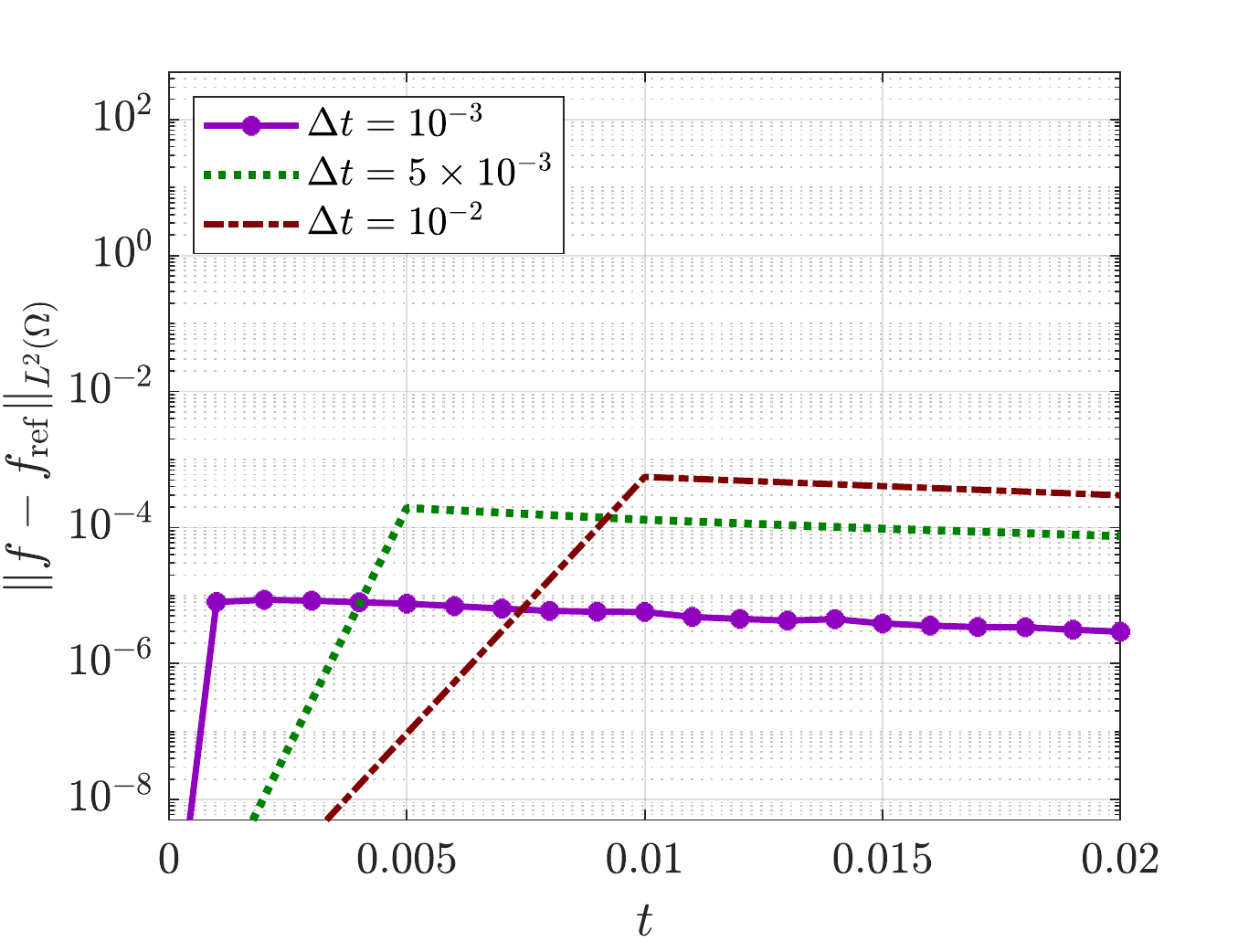}
\end{center}
\caption{
Explicit and implicit step-truncation midpoint methods 
applied the Allen-Cahn equation \eqref{eqn:allen-cahn-intro}
with $\varepsilon = 0.1$.
It is seen that the explicit step-truncation midpoint method undergoes a numerical instability for $\Delta t=10^{-3}$ while the implicit step-truncation midpoint method retains accuracy and stability for $\Delta t=10^{-3}$, and even larger time steps. Stability implicit step-truncation midpoint is studied in section \ref{sec:lin-stability}.}
\label{fig:intro-stiff-compare}
\end{figure}

The main objective of this paper is to develop 
a new class of rank-adaptive {\em implicit} step-truncation 
algorithms to integrate high-dimensional 
initial value problems of the form \eqref{eqn:ode} on 
low-rank tensor manifolds. The main idea of these new 
integrators is illustrated in Figure \ref{fig:intro-surf-compare}.
Roughly speaking, implicit step-truncation method 
take $\bm f_k\in {\cal H}_{\bm r}$ (${\cal H}_{\bm r}$ is a HT 
or TT tensor manifold with multilinear rank $\bm r$) 
and $\bm \Psi_{\Delta t}(\bm G,\bm f_k)$ 
as input and generate a sequence of inexact Newton
iterates ${\bm f}^{[j]}$ converging to a point 
tensor manifold ${\cal H}_{\bm s}$. Once ${\bm f}^{[j]}$ is  
sufficiently close to ${\cal H}_{\bm s}$ we project it 
onto the manifold via a standard truncation operation. 
This operation is also known as ``compression step'' 
in the HT/TT-GMRES algorithm described in \cite{dolgov2013ttgmres}.
Of course the computational cost of 
implicit step-truncation methods is higher 
than that of explicit step truncation methods for 
one single step. However, implicit methods 
allow to integrate stably with larger time-steps 
while retaining accuracy.
Previous research on implicit tensor
integration leveraged the Alternating Least Squares (ALS) 
algorithm \cite{parr_tensor,dolgov2012fast,cho2016},
which essentially attempts to solve an optimization
problem on a low-rank tensor manifold to compute the 
solution of \eqref{eqn:ode} at each time step.
As is well-known, ALS is equivalent to the (linear) 
block Gauss-Seidel iteration applied to the Hessian matrix 
of the residual, and can have convergence 
issues \cite{uschmajew2012}.

\begin{figure}[t]
\begin{center}
\includegraphics[scale=0.5]{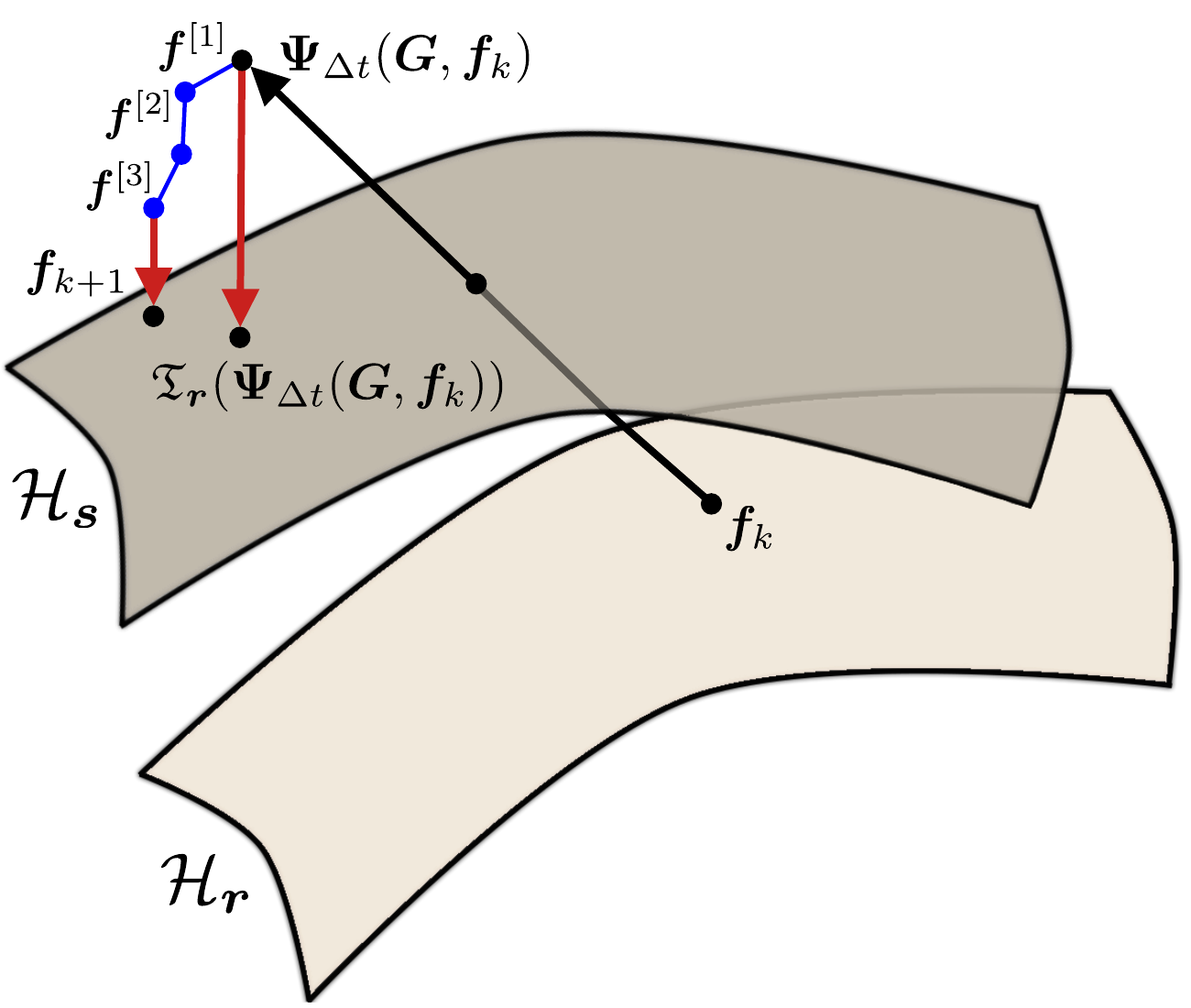}
\end{center}
\caption{
Sketch of implicit and explicit step-truncation 
integration methods. Given a tensor $\bm f_k$ with multilinear 
rank $\bm r$ on the tensor manifold ${\cal H}_{\bm r}$, we first 
perform an explicit time-step, e.g., with the conventional 
time-stepping scheme \eqref{eqn:discrete-flow-intro}.
The explicit step-truncation integrator then projects
$\bm \Psi_{\Delta t}(\bm G,\bm f_k)$
onto a new tensor manifold ${\cal H}_{\bm s}$ 
(solid red line). 
The multilinear rank $\bm s$ is chosen adaptively 
based on desired accuracy and 
stability constraints \cite{rodgers2020adaptive}. 
On the other hand, the implicit step-truncation method 
takes $\bm \Psi_{\Delta t}(\bm G,\bm f_k)$ 
as input and generates a sequence of fixed-point 
iterates ${\bm f}^{[j]}$ shown as dots connected with 
blue lines. The last iterate is then projected 
onto a low rank tensor manifold, illustrated 
here also as a red line landing on
${\cal H}_{\bm s}$. This operation is 
equivalent to the {compression step} 
in the HT/TT-GMRES algorithm described 
in \cite{dolgov2013ttgmres}.}
\label{fig:intro-surf-compare}
\end{figure}

The paper is organized as follows.
In section \ref{sec:explicit-st},
we briefly review rank-adaptive explicit step-truncation 
methods, and present a new convergence proof
for these methods which applies also
to implicit step-truncation methods.
In section \ref{sec:implicit-st},
we discuss the proposed new algorithms for 
implicit step-truncation integration. 
In section \ref{sec:implicit-conv} we study convergence of 
particular implicit step-truncation methods, namely the
step-truncation implicit Euler and midpoint methods.
In section \ref{sec:lin-stability} we prove
that the stability region of an implicit step-truncation
method is identical to that of the implicit
method without tensor truncation.
Finally, in section \ref{sec:numerics}
we present numerical applications
of implicit step-truncation algoritihms to
stiff PDEs. In particular, we
study a two-dimensional Allen-Cahn
equation, a four-dimensional Fokker-Planck
equation, and a six-dimensional 
nonlinear Schr\"odinger equation.
We also include a brief appendix in which 
we discuss numerical algorithms to solve 
linear and nonlinear algebraic equations 
on tensor manifolds via the inexact 
Newton's method with HT/TT-GMRES iterations. 

\section{Explicit step-truncation methods}
\label{sec:explicit-st}
\noindent
In this section we briefly review explicit 
step-truncation methods to integrate the tensor-valued 
ODE \eqref{eqn:ode} on tensor manifolds with variable rank. 
For a complete account of this theory see \cite{rodgers2020adaptive}.
We begin by first discretizing the ODE in time 
with any standard explicit one-step method on an 
evenly-spaced temporal grid
\begin{equation}
\label{eqn:discrete-flow}
{\bm f}_{k+1} = 
{\bm \Psi}_{\Delta t}({\bm G}, {\bm f}_{k}).
\end{equation}
Here, ${\bm f}_{k}$ denotes an approximation of the exact 
solution ${\bm f}(k\Delta t)$ for $k=1,2,...,N$, and 
${\bm \Psi}_{\Delta t}$ is an increment function.
For example, ${\bm \Psi}_{\Delta t}$ can be the 
increment function corresponding to the Euler forward method
\begin{equation}
{\bm \Psi}_{\Delta t}({\bm G}, {\bm f}_k)
={\bm f}_k + \Delta t
{\bm G}({\bm f}_k).
\label{EFm}
\end{equation}
In the interest of saving computational resources 
when iterating \eqref{eqn:discrete-flow} we look for 
an approximation of $\bm f_k$ on a low-rank tensor 
manifold ${\cal H}_{\bm r}$ \cite{uschmajew2013geometry} 
with multilinear rank $\bm r$. ${\cal H}_{\bm r}$ 
is taken to be the manifold of Hierarchical Tucker 
(HT) tensors. The easiest way 
for approximating \eqref{eqn:discrete-flow} 
on ${\cal H}_{\bm r}$ is to apply a nonlinear 
projection operator \cite{grasedyck2010hierarchical}
(truncation operator)
\begin{equation}
\mathfrak{T}_{\bm r}:\mathbb{R}^{n_1\times n_2\times\cdots
\times n_d} \to \overline{\cal H}_{\bm r},
\end{equation}
where $\overline{\cal H}_{\bm r}$ denotes the closure 
of ${\cal H}_{\bm r}$. This yields the 
explicit step-truncation scheme 
\begin{equation}
\label{eqn:discrete-flow-step-truncation}
{\bm f}_{k+1} = 
{\mathfrak T}_{\bm r}\left({\bm \Psi}_{\Delta t}({\bm G}, 
{\bm f}_{k})\right).
\end{equation}
The rank $\bm r$ can vary with 
the time step based on appropriate error 
estimates as time integration proceeds 
\cite{rodgers2020adaptive}.
We can also project $\bm G(\bm f)$ onto 
$\overline{\cal H}_{\bm r}$ before applying 
${\bm \Psi}_{\Delta t}$. With reference to \eqref{EFm} 
this yields
\begin{align}
\label{eqn:st-exp-euler}
\bm f_{k+1}=
{\mathfrak{T}}_{{\bm r}_2}\left({\bm f}_k + \Delta t
{\mathfrak{T}}_{{\bm r}_1}\left({\bm G}({\bm f}_k)\right)\right).
\end{align}
Here $\bm r_1$ and ${\bm r}_2$ are
truncation ranks determined by the inequalities\footnote{Throughout 
the paper, $\left\|\cdot \right\|$ denotes the standard 
tensor 2-norm \cite{grasedyck2010hierarchical,kressner2014algorithm}, 
or a weighted version of it.}      
\begin{equation}
\left\| {\bm G}({\bm f}_k)-\mathfrak{T}_{{\bm r}_1}
\left(
{\bm G}({\bm f}_k)\right)\right\|\leq e_1, \qquad 
\left\| \widetilde{\bm f}_{k+1}-\mathfrak{T}_{{\bm r}_2}
\left(\widetilde{\bm f}_{k+1}\right)\right\|\leq e_2,
\label{et}
\end{equation}
where $e_1$ and $e_2$ are chosen error thresholds. As 
before, $\bm r_1$ and $\bm r_2$ can change with every 
time step. 
In particular, if we choose $e_1 = K_1 \Delta t$ and 
$e_2 = K_2 \Delta t^2$ (with $K_1$ and $K_2$ given constants) 
then the step-tuncation method \eqref{eqn:st-exp-euler} 
is convergent (see \cite{rodgers2020adaptive} for details). 
More generally, let 
\begin{equation}
\bm f_{k+1} =  {\bm \Phi}_{\Delta t}({\bm G},{\bm f_k},{\bm e}) 
\label{gen-explicit-st}
\end{equation}
be an explicit step-truncation method in which all we project 
all $\bm G(\bm f_k)$ appearing in the increment 
function ${\bm \Psi}_{\Delta t}({\bm G}, {\bm f}_k)$  
onto tensor manifolds $\overline{{\cal H}}_{\bm r_i}$ by setting 
suitable error thresholds $\bm e=(e_1,e_2,\ldots)$. For instance, if 
${\bm \Psi}_{\Delta t}$ is defined by the explicit midpoint 
method, i.e., 
\begin{equation}
{\bm \Psi}_{\Delta t}\left({\bm G},{\bm f_k}\right)  = 
\bm f_k +\bm G\left(\bm f_k+\frac{\Delta t}{2}\bm G(\bm f_k)
\right)
\label{ExpMidpoint}
\end{equation}
then 
\begin{equation}
{\bm \Phi}_{\Delta t}({\bm G},{\bm f_k},{\bm e})  = 
\mathfrak{T}_{\bm r_3}\left(\bm f_k +
\mathfrak{T}_{\bm r_2}\left[\bm G\left(\bm f_k+\frac{\Delta t}{2}\mathfrak{T}_{\bm r_1}\left[\bm G(\bm f_k)\right]\right) \right]\right),
\end{equation}
where $\bm e=(e_1,e_2,e_3)$ is a vector collecting the 
truncation error thresholds yielding the multilinear 
ranks $\bm r_1$, $\bm r_2$ and $\bm r_3$. 
By construction, step-truncation methods of the form 
\eqref{gen-explicit-st} satisfy 
\begin{equation}
\left \|
{\bm \Psi}_{\Delta t}({\bm G}, {\bm f})
-
{\bm \Phi}_{\Delta t}({\bm G},{\bm f},{\bm e})
\right \| \leq R({\bm e}),
\end{equation}
where $R({\bm e})$ is the error 
due to tensor truncation.
We close this section with a reformulation of the 
convergence theorem for explicit step-truncation
methods in \cite{rodgers2020adaptive}, which applies 
also to implicit methods.

\begin{theorem}[Convergence of step-truncation methods]
\label{thm:convergence}
Let ${\bm \varphi}_{\Delta t}({\bm G},{\bm f})$ be the 
one-step exact flow map defined by \eqref{eqn:ode}, and 
${\bm \Phi}_{\Delta t}({\bm G},{\bm f},{\bm e})$
be the increment function of a step-truncation method 
with local error of order $p$, i.e.,
\begin{equation}
\label{eqn:st-consistency}
\left \| 
{\bm \varphi}_{\Delta t}({\bm G},{\bm f})
-
{\bm \Phi}_{\Delta t}({\bm G},{\bm f},{\bm e})
\right \|
\leq K\Delta t^{p+1}\quad \text{as} \quad \Delta t\rightarrow 0.
\end{equation}
If there exist truncation errors ${\bm e} = {\bm e}(\Delta t)$ 
(function of $\Delta t$) and constants $C,E>0$ (dependent on
${\bm G}$)
so that the stability condition
\begin{equation}
\label{eqn:st-stability}
\left \|
{\bm \Phi}_{\Delta t}({\bm G},\hat{\bm f},{\bm e})
-
{\bm \Phi}_{\Delta t}({\bm G},\tilde{\bm f},{\bm e})
\right \|
\leq
(1 + C\Delta t)
\left \|
\hat{\bm f}-\tilde{\bm f}
\right \|
+ E\Delta t^{m+1}
\end{equation}
holds as $\Delta t\rightarrow 0$, then
the step-truncation method is convergent with order $z=\text{min}(m,p)$.
\end{theorem}

\begin{proof}
Under the assumption that $\Delta t$ is small enough
for our stability and consistency to hold,
we proceed by induction on the number of steps. The one-step 
case coincides with the consistency condition. Next, we assume that 
\begin{equation}
\label{eqn:inductive-convergence}
\left \| 
{\bm \varphi}_{\Delta t(N-1)}({\bm G},{\bm f}_0)
-
{\bm f}_{N-1}
\right \|
\leq Q_{N-1}\Delta t^{z}.
\end{equation}
Let $T=N\Delta t$ be the final integration time.
By applying triangle inequality and the
semigroup property of the flow map,
\begin{align}
\left \|
{\bm \varphi}_{T}({\bm G},{\bm f}_0)
-
{\bm \Phi}_{\Delta t}({\bm G},{\bm f}_{N-1},{\bm e})
\right \|
&\leq
\left \|
{\bm \varphi}_{\Delta t}({\bm G},
{\bm \varphi}_{\Delta t(N-1)}({\bm G},{\bm f}_{0}))
-
{\bm \Phi}_{\Delta t}({\bm G},
{\bm \varphi}_{\Delta t(N-1)}({\bm G},{\bm f}_{0}),
{\bm e})
\right \|\nonumber\\
&\qquad +
\left \|
{\bm \Phi}_{\Delta t}({\bm G},
{\bm \varphi}_{\Delta t(N-1)}({\bm G},{\bm f}_{0}),
{\bm e})
-
{\bm \Phi}_{\Delta t}({\bm G},{\bm f}_{N-1},{\bm e})
\right \|.
\label{ll}
\end{align}
Using the consistency condition \eqref{eqn:st-consistency}
we can bound the first term at the right hand side of \eqref{ll} 
by $K_{N-1}\Delta t^{p+1}$, where
$K_{N-1}$ represents a local error coefficient.
On the other hand, using the stability condition \eqref{eqn:st-stability}
we can can bound the second term at the right hand side of \eqref{ll} as
\begin{align}
\left \|
{\bm \Phi}_{\Delta t}({\bm G},
{\bm \varphi}_{\Delta t(N-1)}({\bm G},{\bm f}_{0}),
{\bm e})
-
{\bm \Phi}_{\Delta t}({\bm G},{\bm f}_{N-1},{\bm e})
\right \|
\leq
(1+C\Delta t)
&\left \|
{\bm \varphi}_{\Delta t(N-1)}({\bm G},{\bm f}_{0})
- {\bm f}_{N-1}
\right \|\nonumber \\
+ E\Delta t^{m+1}.
\label{ttt}
\end{align}
A substitution of \eqref{ttt} and \eqref{eqn:inductive-convergence}
into \eqref{ll} yields
\begin{align*}
\left \|
{\bm \varphi}_{T}({\bm G},{\bm f}_0)
-
{\bm \Phi}_{\Delta t}({\bm G},{\bm f}_{N-1},{\bm e})
\right \|
&\leq
K_{N-1}\Delta t^{p+1}
+(1+C\Delta t)Q_{N-1}\Delta t^z
+ E\Delta t^{m+1}.
\end{align*}
Recalling that $z = \text{min}(m,p)$ completes the proof.
\begin{flushright}
\qed
\end{flushright}
\end{proof}

\vspace{0.2cm}
We remark that the above proof may be modified to
include explicit step-truncation linear multi-step methods, i.e., 
step-truncation Adams methods. To this end, it is sufficient 
to replace ${\bm f}_k$ with the vector 
$({\bm f}_{k-s},{\bm f}_{k-s+1},\dots,{\bm f}_{k})$
and the stability condition \eqref{eqn:st-stability} 
with
\begin{equation}
\left \|
{\bm \Phi}_{\Delta t}({\bm G},{\bm{\hat f}}_{1},
{\bm{\hat f}}_{2},\dots,{\bm{\hat f}}_{s},{\bm e})
-
{\bm \Phi}_{\Delta t}({\bm G},
{\bm{\tilde f}}_{1},
{\bm{\tilde f}}_{2},\dots,{\bm{\tilde f}}_{s}
,{\bm e})
\right \|
\leq
\left \|
{\bm{\hat f}}_s-{\bm{\tilde f}}_s
\right \|
+
\sum_{j=1}^{s}
C_j\Delta t
\left \|
{\bm{\hat f}}_j-{\bm{\tilde f}}_j
\right \|
+ E\Delta t^{m+1}.
\end{equation}
We then have an analogous theorem
proven using an inductive argument
based on the initial conditions 
${\bm f}(0), {\bm f}(\Delta t),\dots,{\bm f}((s-1)\Delta t)$.

\section{Implicit step-truncation methods}
\label{sec:implicit-st}

\noindent
To introduce implicit step-truncation tensor methods, let us  
begin with the standard Euler backward scheme  
\begin{equation}
\label{eqn:imp-euler}
{\bm f}_{k+1} = {\bm f}_{k} + \Delta t {\bm G}({\bm f}_{k+1}),
\end{equation}
and the associated root-finding problem 
\begin{equation}
\bm H_k(\bm f_{k+1})={\bm f}_{k+1} - {\bm f}_{k} -
\Delta t {\bm G}({\bm f}_{k+1})={\bm 0}.
\label{EB}
\end{equation}
Equation \eqref{EB} allows us to compute $\bm f_{k+1}$ 
as a zero of the function $\bm H_k$. This can be done, e.g., using 
the Netwon's method with initial guess 
$\bm f^{[0]}_{k+1}=\bm f_k$.
As is well-known, if the Jacobian of $\bm H_k$ 
is invertible within a neighborhood of $\bm f_k$, then 
the implicit function theorem guarantees the existence 
of a locally differentiable (in some neighborhood of $\bm f_k$) 
nonlinear map ${\bm \Theta}_{\Delta t}$ depending on $\bm G$ 
such that 
\begin{equation}
{\bm f}_{k+1}={\bm \Theta}_{\Delta t}\left({\bm G},{\bm f}_k\right).
\label{Tmap}
\end{equation} 
In the setting of  Newton's method described above, 
the map ${\bm \Theta}_{\Delta t}$
is computed iteratively. 
An implicit step-truncation scheme can be then
formulated by applying the tensor truncation 
operator $\mathfrak{T}_{\bm r}$ to the right hand side 
of \eqref{Tmap}, i.e., 
\begin{equation}
\label{eqn:ideal-implicit-st}
{\bm f}_{k+1} = {\mathfrak T}_{{\bm r}}\left(
{\bm \Theta}_{\Delta t}\left({\bm G},{\bm f}_k\right)\right).
\end{equation}
The tensor truncation rank $\bm r$ can be selected 
based on the inequality 
\begin{equation}
\label{eqn:ideal-implicit-st_rank}
\|{\mathfrak T}_{{\bm r}}(
{\bm \Theta}_{\Delta t}({\bm G},{\bm f}_k))
-
{\bm \Theta}_{\Delta t}({\bm G},{\bm f}_k)
\|\leq A\Delta t^2.
\end{equation}
It was shown in \cite{rodgers2020adaptive} that 
this yields an order one (in time) integration scheme. 
Of course, if the Jacobian of $\bm H_k$ in equation \eqref{EB} 
can be computed and stored in computer memory, then we can 
approximate $\bm f_{k+1}={\bm \Theta}_{\Delta t}({\bm G},{\bm f}_k)$ 
for any given $\bm f_{k}$ and $\bm G$ using Newton's iterations.  
However, the exact Newton's method is not available to us 
in the high-dimensional tensor setting.

Hence, we look for an approximation of  
${\bm \Theta}_{\Delta t}({\bm G},{\bm f}_k)$
computed using the inexact Newton 
method \cite{dembo1982inexact}.
For the inexact matrix inverse step,
we apply the relaxed TT-GMRES algorithm
described in \cite{dolgov2013ttgmres}. 
This is an iterative method for the solution 
of linear equations which makes use of an inexact 
matrix-vector product defined by low-rank truncation. 
Though the algorithm was developed for TT tensors, 
it may be also applied to HT tensors without 
significant changes. 

In Appendix \ref{apndx:ht-tt-newton} we describe 
the inexact Newton's method with HT/TT-GMRES 
iterations to solve an arbitrary algebraic 
equation of the form $\bm H_k(\bm f)=\bm 0$ 
(e.g., equation \eqref{EB}) on a tensor manifold 
with a given rank. 
The algorithm can be used to approximate the 
mapping ${\bm \Theta}_{\Delta t}({\bm G},{\bm f}_k)$ 
corresponding to any implicit integrator, 
and it returns a tensor with can be then 
truncated further as in \eqref{eqn:ideal-implicit-st}.
This operation is equivalent to the so-called 
``compression step'' in \cite{dolgov2013ttgmres} and 
it is described in detail in the next section.

\subsection{The compression step}
\label{subsec:low-rank-errors}
\noindent
While increasing the tensor rank may be necessary
for convergence of the HT/TT-GMRES iterations, 
it is possible that we raise the rank by more 
than is required for the desired order of accuracy
in a single time step. 
Therefore, it is convenient to apply an additional
tensor truncation after computing, say,
$j$ steps of the HT inexact Newton's
method which returns ${\bm f}_{k+1}^{[j]}$.
This is the same as the ``compression step''
at the end of HT/TT-GMRES algorithm as presented
in \cite{dolgov2013ttgmres}.
This gives us our final estimate of
${\bm f}((k+1)\Delta t)$ as
\begin{equation}
{\bm f}_{k+1} = {\mathfrak T}_{\bm r}
\left ({\bm f}_{k+1}^{[j]}\right ),
\qquad
\left \|
{\mathfrak T}_{\bm r}
\left ({\bm f}_{k+1}^{[j]}\right )
-{\bm f}_{k+1}^{[j]}
\right \|\leq e_{\bm r}.
\end{equation}
Regarding the selection of the truncation error $e_{\bm r}$ 
we proceed as follows. 
Suppose that ${\bm \Psi}_{\Delta t}({\bm G},{\bm f}_k)$ is
an explicit integration scheme of the same order (or higher) 
than the implicit scheme being considered. 
Then we can estimate local error as
\begin{align}
\left \|
{\bm f}_{k+1}^{[j]} - {\bm f}((k+1)\Delta t)
\right \|
&\leq 
\left \|
{\bm f}_{k+1}^{[j]}
-
{\bm \Psi}_{\Delta t}({\bm G},{\bm f}_k)
\right \|
+
\left \|
{\bm \Psi}_{\Delta t}({\bm G},{\bm f}_k)
-{\bm f}((k+1)\Delta t)
\right \|\nonumber
 \\
&=
\left \|
{\bm f}_1^{[j]}
-
{\bm \Psi}_{\Delta t}({\bm G},{\bm f}_k)
\right \|
+O(\Delta t^{p+1}),
\end{align}
Thus, we may roughly
estimate the local
truncation error and
set this as the chosen
error for approximation to
HT or TT rank $\bm r$ using
\begin{equation}
{e}_{\bm r}=
\left \|
{\bm f}_{k+1}^{[j]}
-
{\bm \Psi}_{\Delta t}({\bm G},{\bm f}_k)
\right \|.
\label{er1}
\end{equation}
We may drop more singular values than
needed, especially if the choice of $\Delta t$
is outside the region of stability of the explicit scheme
${\bm \Psi}_{\Delta t}({\bm G},{\bm f}_k)$.
However, this estimate guarantees that
we do not change the overall convergence rate.
Moreover, it cannot impact
stability of the implicit step-truncation 
integrator since the compression step has 
operator norm equal to one, regardless of the 
rank chosen. This statement is supported by the 
analysis presented in \cite{rodgers2020stability} 
for step-truncation linear multi-step methods.
In all our numerical experiments we
use the explicit step-truncation midpoint
method to estimate local error, i.e., 
${\bm \Psi}_{\Delta t}({\bm G},{\bm f}_k)$ 
in \eqref{er1} is set as in \eqref{ExpMidpoint}.

\section{Convergence analysis of implicit step-truncation methods}
\label{sec:implicit-conv}
\noindent
In this section we show that applying the inexact
Newton iterations with HT/TT-GMRES to the implicit Euler and the 
implicit midpoint methods result
in implicit step-truncation schemes 
that fit into the framework of Theorem \ref{thm:convergence}, 
and therefore are convergent. 
\subsection{Implicit step-truncation Euler method}
\label{sec:imp-euler}
\noindent

Consider the implicit Euler scheme \eqref{EB}, and suppose that 
the solution of the nonlinear equation $\bm H_k(\bm f_{k+1})=\bm 0$ is 
computed using the inexact Newton method with HT/TT-GMRES  
iterations as discussed in Appendix \ref{apndx:ht-tt-newton}.
Let ${\bm f}_{k+1}^{[j]}$ ($j=1,2,\ldots,$) be the 
sequence of tensors generated by the algorithm and approximating 
$\bm f_{k+1}$ (exact solution of $\bm H_k(\bm f_{k+1})=\bm 0$). 
We set a stopping criterion for terminating Newton's iterations 
based on the residual, i.e., 
\begin{equation}
\left \|
{\bm H}_k\left ({\bm f}_{k+1}^{[j]}\right )
\right \|\leq \varepsilon_{\text{tol}}.
\label{stopping}
\end{equation}
This allows us to adjust the rank of ${\bm f}_{k+1}^{[j]}$ 
from one time step to the next, depending on the desired 
accuracy $\varepsilon_{\text{tol}}$. 
Our goal is to analyze convergence of such rank-adaptive implicit 
method when a finite number of contraction mapping steps is taken. 
To this end, we will fit the method into the framework of
Theorem \ref{thm:convergence}. We begin by noticing that
\begin{align}
\label{eqn:newton-distance}
\left\|
{\bm f}_{k+1} -
{\bm f}_{k+1}^{[j]}
\right\|
&=
\left \|
{\bm H}_k^{-1}\left (
{\bm H}_k\left (
{\bm f}_{k+1}
\right )
\right )
-
{\bm H}_k^{-1}\left (
{\bm H}_k\left (
{\bm f}_{k+1}^{[j]}
\right )
\right )
\right \|
\leq
L_{{\bm H}_k^{-1}}
\left \|
{\bm H}\left ({\bm f}_{k+1}^{[j]}\right )
\right \|\leq
L_{{\bm H}_k^{-1}}
\varepsilon_{\text{tol}},
\end{align}
where $L_{{\bm H}_k^{-1}}$
is the local Lipschitz constant of the smooth
inverse map $H^{-1}_k$, the existence of which is granted 
by the inverse function theorem.
This allows us to write the local truncation error as
\begin{align}
\label{eqn:consistency-euler}
\left\|
{\bm f}(\Delta t) -
{\bm f}_{1}^{[j]}
\right\|
&\leq
\left \|
{\bm f}(\Delta t)
-
{\bm f}_{1}
\right \|
+
\left \|
{\bm f}_{1}
-
{\bm f}_{1}^{[j]}
\right \|
\leq
K_1\Delta t^{2}
+
L_{{\bm H}^{-1}}
\varepsilon_{\text{tol}},
\end{align}
where $K_1$ is a local error coefficient.
In order to maintain order one convergence,
we require that $\varepsilon_{\text{tol}} \leq K\Delta t^2$
for some constant $K\geq 0$. 

Next, we discuss the stability condition 
\eqref{eqn:st-stability}
in the context of the implicit step-truncation Euler scheme, assuming that 
${\bm f}_{k+1}$ can be found exactly by the inexact Newton's method, 
eventually after an infinite number of iterations.  
Denote by ${\bm{\hat f}}_0$, ${\bm{\tilde f}}_0$
two different initial conditions. Performing one step of the
standard implicit Euler's method yields the following bound 
\begin{align*}
\left \|
{\bm{\hat f}}_1 -{\bm{\tilde f}}_1
\right \|
=
\left \| \left ( {\bm{\hat f}}_0 +
\Delta t {\bm G}({\bm{\hat f}}_1)\right)
-
\left ( {\bm{\tilde f}}_0 + 
\Delta t{\bm G}({\bm{\tilde f}}_1) \right )
\right \|
\leq
\left \| {\bm{\hat f}}_0 - {\bm{\tilde f}}_0 
\right \|
+
\Delta t L_{\bm G} \left \| 
{\bm{\hat f}}_1
-
{\bm{\tilde f}}_1
\right \|.
\end{align*}
where $L_{\bm G}$ is the Lipschitz constant of $\bm G$. 
By collecting like terms, we obtain
\begin{align}
\label{eqn:euler-lipschitz}
\left \|
{\bm{\hat f}}_1 -{\bm{\tilde f}}_1
\right \|
\leq 
\frac{1}{1-\Delta tL_{\bm G}}
\left \| {\bm{\hat f}}_0 - {\bm{\tilde f}}_0 
\right \|
\leq (1+2\Delta tL_{\bm G})
\left \| {\bm{\hat f}}_0 - {\bm{\tilde f}}_0 
\right \|.
\end{align}
The last inequality comes by noting that
$1/({1-\Delta tL_{\bm G}})\leq 1+2\Delta tL_{\bm G}$
when $\Delta t$ is sufficiently small, i.e., 
$\Delta t \leq 1/(2L_{\bm G})$. 
This is zero-stability condition, i.e., 
a stability condition that holds for small $\Delta t$, 
which will be used for convergence analysis. 
Regarding the behavior of the scheme for finite 
$\Delta t$ we will show in section 
\ref{sec:lin-stability} that the implicit 
step-truncation Euler scheme is unconditionally 
stable.
Next, we derive a stability condition of the form \eqref{eqn:euler-lipschitz} 
when the root of $\bm H_k(\bm f_{k+1})=\bm 0$ is 
computed with the inexact Newton's method with HT/TT-GMRES iterations. 
In this case we have 
\begin{align}
\nonumber
\left \|
{\bm{\hat f}}_1^{[j]}
-{\bm{\tilde f}}_1^{[m]}
\right \|
&\leq
\left \|
{\bm{\hat f}}_1^{[j]}
-{\bm{\hat f}}_1^{[\infty]}
\right \|
+
\left \|
{\bm{\hat f}}_1^{[\infty]}
-{\bm{\tilde f}}_1^{[\infty]}
\right \|
+
\left \|
{\bm{\tilde f}}_1^{[m]}
-{\bm{\tilde f}}_1^{[\infty]}
\right \|
\\
&=
\left \|
{\bm{\hat f}}_1^{[j]}
-{\bm{\hat f}_1}
\right \|
+
\left \|
{\bm{\hat f}_1} -{\bm{\tilde f}_1}
\right \|
+
\left \|
{\bm{\tilde f}}_1^{[m]}
-{\bm{\tilde f}_1}
\right \|
\nonumber
\\
\label{eqn:stability-euler-2}
&\leq
(1+2L_{\bm G}\Delta t)
\left \|
{\bm{\hat f}}_0 - {\bm{\tilde f}}_0 \right \|
+
2L_{{\bm H}^{-1}}
 \varepsilon_{\text{tol}}.
\end{align}
Thus, the stability condition is satisfied
by the same condition on $\varepsilon_{\text{tol}}$ 
that satisfies the first-order consistency 
condition \eqref{eqn:consistency-euler}, i.e.,
$\varepsilon_{\text{tol}} \leq K\Delta t^2$.
At this point we apply Theorem \ref{thm:convergence}
with consistency and stability conditions 
\eqref{eqn:st-consistency}-\eqref{eqn:st-stability} replaced by 
\eqref{eqn:consistency-euler} and \eqref{eqn:stability-euler-2}, 
respectively, and conclude that the implicit step-truncation 
Euler scheme is convergent with order one if 
$\varepsilon_{\text{tol}} \leq K\Delta t^2$.

\begin{remark}
\label{rmk:linear-subcase-nonlinear}
When $\bm G$ is linear, i.e., when the 
tensor ODE \eqref{eqn:ode} is linear, 
then we may apply HT/TT-GMRES algorithm in 
Appendix \ref{apndx:ht-tt-newton} without invoking 
Newton's method. In this case, the local error 
coefficient $L_{{\bm H}_k^{-1}}$ can be exchanged 
for the coefficient of $\varepsilon$ at the right 
side of inequality \eqref{eqn:gmres-error-2}.
\end{remark}

\subsection{Implicit step-truncation midpoint method}

\noindent
The implicit midpoint rule \cite{hairer2006geometric},
\begin{equation}
{\bm f}_{k+1} = {\bm f}_{k}+
\Delta t {\bm G}\left (
\frac{1}{2}({\bm f}_{k} + {\bm f}_{k+1})
\right ),
\end{equation}
is a symmetric and
symplectic method of order 2.
By introducing
\begin{equation}
{\bm H}_k
\left (
{\bm f}_{k+1}
\right )
=
{\bm f}_{k+1} - {\bm f}_{k}-
\Delta t {\bm G}\left (
\frac{1}{2}({\bm f}_{k} + {\bm f}_{k+1})
\right )
= {\bm 0}
\label{impmid}
\end{equation}
we again see the implicit method as a root finding problem at each
time step. To prove convergence of the implicit 
step-truncation midpoint method we follow the 
same steps described in the previous section. 
To this end, consider the sequence 
of tensors ${\bm f}^{[j]}_{k+1}$ generated by the inexact 
Newton method with HT/TT-GMRES iterations 
(see Appendix \ref{apndx:ht-tt-newton}) applied 
to \eqref{impmid}. The sequence of tensors ${\bm f}^{[j]}_{k+1}$ 
approximates $\bm f_{k+1}$ satisfying \eqref{impmid}. 
As before, we terminate the inexact Newton's iterations 
as soon as condition \eqref{stopping} is satisfied.
By repeating the same steps that led us to inequality 
\eqref{eqn:newton-distance}, we have
\begin{align}
\label{eqn:consistency-midpoint}
\left\|
{\bm f}(\Delta t) -
{\bm f}_{1}^{[j]}
\right\|
\leq
K_1\Delta t^{3}
+
L_{{\bm H}_k^{-1}}
\varepsilon_{\text{tol}}.
\end{align}
Hence, setting the stopping tolerance
as $\varepsilon_{\text{tol}}\leq K\Delta t^{3}$
we get second-order consistency. We use this to determine
the stability condition \eqref{eqn:st-stability}.
As before, we derive the condition for when the 
zero of \eqref{impmid} is exact. Denote by ${\bm{\hat f}}_0$, 
${\bm{\tilde f}}_0$ two different initial conditions.
Performing one step of the standard implicit midpoint method 
yields the following bound 
\begin{align*}
\left \|
{\bm{\hat f}}_1 - {\bm{\tilde f}}_1
\right \|
\leq
\left (1+ \Delta t\frac{L_{\bm G}}{2}\right )
\left \|
{\bm{\hat f}}_0 - {\bm{\tilde f}}_0
\right \|
+
\Delta t\frac{L_{\bm G}}{2}
\left \|
{\bm{\hat f}}_1 - {\bm{\tilde f}}_1
\right \|.
\end{align*}
By collecting like terms we see that when
$\Delta t\leq 1/L_{\bm G}$,
\begin{align}
\left \|
{\bm{\hat f}}_1 - {\bm{\tilde f}}_1
\right \|
&\leq
\frac{2+ \Delta tL_{\bm G}}
{2- \Delta tL_{\bm G}}
\left \|
{\bm{\hat f}}_0 - {\bm{\tilde f}}_0
\right \|\nonumber\\
&\leq
\left (
1+\Delta t\frac{3L_{\bm G}}{2}
\right )
\left \|
{\bm{\hat f}}_0 - {\bm{\tilde f}}_0
\right \|.
\end{align}
The zero-stability condition for the implicit 
step-truncation midpoint method can now be found by 
repeating the arguments of inequality \eqref{eqn:stability-euler-2}. This gives
\begin{equation}
\left \|
{\bm{\hat f}}_1^{[j]}
-{\bm{\tilde f}}_1^{[m]}
\right \|
\leq
(1+L_{\bm G}\Delta t)
\left \|
{\bm{\hat f}}_0 - {\bm{\tilde f}}_0 \right \|
+
2L_{{\bm H}_k^{-1}}\varepsilon_{\text{tol}}.
\label{eqn:stability-midpoint-2}
\end{equation}
Thus, the stability condition is satisfied
by the same condition on $\varepsilon_{\text{tol}}$ 
that satisfies the second-order consistency 
condition \eqref{eqn:consistency-midpoint}, i.e.,
$\varepsilon_{\text{tol}} \leq K\Delta t^3$.
At this point we apply Theorem \ref{thm:convergence}
with consistency and stability conditions 
\eqref{eqn:st-consistency}-\eqref{eqn:st-stability} replaced by 
\eqref{eqn:consistency-midpoint} and \eqref{eqn:stability-midpoint-2}, 
respectively, and conclude that the implicit step-truncation 
midpoint scheme is convergent with order two if 
$\varepsilon_{\text{tol}} \leq K\Delta t^3$.

\section{Stability analysis}
\label{sec:lin-stability}
\noindent
We now address absolute stability of the proposed 
implicit step-truncation methods. This notion of 
stability is related to the behavior of 
the schemes when applied the initial 
value problem
\begin{equation}
\frac{d{\bm f}}{dt} = {\bm L} {\bm f} ,
\quad {\bm f}(0) = {\bm f}_0
\end{equation}
where $\bm L$ is a linear operator with 
eigenvalues in in the left half complex plane. 
After applying any standard implicit time stepping scheme, 
we end up with a system of linear equations of the form
\begin{equation}
\label{eqn:implicit-linear}
{\bm A}{\bm f}_{k+1} = {\bm W}{\bm f}_k.
\end{equation}
Specifically, for the implicit Euler 
we have ${\bm A} = {\bm I } - \Delta t{\bm L}$,
${\bm W} = {\bm I}$ while for the implicit midpoint method 
we have ${\bm A} = {\bm I } - 0.5\Delta t{\bm L}$,
${\bm W} = {\bm I } + 0.5\Delta t{\bm L}$.
As is well known, both implicit Euler and Implicit midpoint 
are unconditionally stable, in the sense that for any 
$\Delta t>0,$ $\| {\bm f}_k\|\rightarrow 0$ as $k\rightarrow \infty$.
One way of proving this is by noting that whenever 
the eigenvalues of ${\bm L}$ have negative real
part, we get 
\begin{equation}
\|{\bm A}^{-1}{\bm W}\| < 1,
\end{equation} 
and therefore the mapping ${\bm A}^{-1}{\bm W}$ is contractive. 
This implies the sequence $\bm f_k$ defined by 
\eqref{eqn:implicit-linear} converges to zero.
The following theorem characterizes 
what happens when we exchange exact matrix 
inverse ${\bm A}^{-1}$ with an inexact inverse computed 
by tensor HT/TT-GMRES iterations.

\begin{theorem}[Absolute stability of implicit step-truncation methods]
\label{thm:imp-st-stable}
Consider an implicit time stepping scheme of the form
\eqref{eqn:implicit-linear}, and suppose that
$\|{\bm A}^{-1}{\bm W}\| < 1$. Denote by 
${\hat {\bm  f}}_k$ the solution of 
$\bm A {\hat {\bm f}}_{k} = {\bm W}{\hat {\bm f}}_{k-1}$ ($k=1,2,\ldots$)
obtained with the HT/TT-GMRES tensor solver described 
in Appendix \ref{apndx:ht-tt-newton}, with $m$ Krylov iterations 
and stopping tolerance $\eta$. If 
\begin{equation}
\left\|{\bm A}{\hat {\bm  f}}_k - {\bm W}{\hat {\bm  f}}_{k-1}\right\|
\leq m\|{\bm A}\|\|{\bm A}^{-1}\|\eta,
\label{hyp}
\end{equation}
then the distance between ${\hat {\bm  f}}_k$ and the exact solution 
${\bm f}_k=\bm A^{-1}\bm W \bm f_{k-1}$ can be bounded as
\begin{equation}
\left\|{\hat {\bm  f}}_k - {\bm  { f}}_{k}\right\|
\leq \frac{m\|{\bm A}\|\|{\bm A}^{-1}\|^2}{
		1 - \|{\bm A}^{-1}{\bm W}\|}\eta.
\label{error}
\end{equation}
\end{theorem}
Note  that \eqref{error} implies that 
$\left\|{\hat {\bm  f}}_k \right\| = O(\eta)$ as 
$k\rightarrow \infty$. In the context of HT/TT-GMRES iterations, 
the number $\eta$ can be controlled by setting
the stopping tolerance in Lemma
\ref{lemma:ht-tt-gmres} (Appendix \ref{apndx:ht-tt-newton})
as 
\begin{equation}
\varepsilon_k =\frac{\eta}{\left\|{\bm W}{\hat {\bm  f}}_k \right\|}
\end{equation}
at each time step $k$. 

\begin{proof}
The proof follows from a straightforward 
inductive argument. For $k=1$ we have
\begin{align}
{\hat {\bm {f}}}_1 - {\bm f}_1 &= 
\bm A^{-1}\bm A {\hat{\bm {f}}}_1 - \bm A^{-1}\bm W {\bm f}_0\nonumber\\
&= 
\bm A^{-1}\left(\bm A{\hat{\bm {f}}}_1 - \bm W {\bm f}_0\right).
\end{align}
By using \eqref{hyp}, we can bound $\left \|
{\hat{\bm {f}}}_1 - {\bm f}_1
\right \|$ as   
\begin{equation}
\left \|
{\hat{\bm {f}}}_1 - {\bm f}_1
\right \|
\leq m
\| {\bm A}\|
\| {\bm A}^{-1}\|^{2}\eta.
\label{49}
\end{equation}
For $k=2$ we have
\begin{align*}
\left \|
{\bm {\hat f}}_2 - {\bm f}_2
\right \|
&\leq 
\left \|
{\bm {\hat f}}_2 - {\bm A}^{-1}{\bm W}{\bm {\hat f}}_1
\right \|
+
\left \|
{\bm A}^{-1}{\bm W}{\bm {\hat f}}_1 - {\bm f}_2
\right \|\\
&=
\left \|
{\bm {\hat f}}_2 - {\bm A}^{-1}{\bm W}{\bm {\hat f}}_1
\right \|
+
\left \|
{\bm A}^{-1}{\bm W}{\bm {\hat f}}_1 - {\bm A}^{-1}{\bm W}{\bm f}_1
\right \|\\
&\leq
m\| {\bm A}\|
\| {\bm A}^{-1}\|^{2}\eta
+\|{\bm A}^{-1}{\bm W}\| 
\left \|
{\hat{\bm {f}}}_1 - {\bm f}_1
\right \|\\
&\leq\left (1+\|{\bm A}^{-1}{\bm W}\|
\right )m\| {\bm A}\|
\| {\bm A}^{-1}\|^{2}\eta,
\end{align*}
where the last inequality follows from \eqref{49}.
More generally,
\begin{equation}
\left \|
{\bm {\hat f}}_{k-1} - {\bm f}_{k-1}
\right \|\leq m\| {\bm A}\|
\| {\bm A}^{-1}\|^{2}\eta \sum_{j=0}^{k-2}\|{\bm A}^{-1}{\bm W}\|^j.
\label{gen}
\end{equation}
Repeating the string of inequalities above and
replacing the right sum with the inductive hypothesis,
we obtain
\begin{align}
\left \|
{\bm {\hat f}}_k - {\bm f}_k
\right \|&\leq 
m\| {\bm A}\|
\| {\bm A}^{-1}\|^{2}\eta \sum_{j=0}^{k-1}\|{\bm A}^{-1}{\bm W}\|^j\nonumber\\
&\leq \frac{m\|{\bm A}\|\|{\bm A}^{-1}\|^2}{
		1 - \|{\bm A}^{-1}{\bm W}\|}\eta,
\end{align}
which completes the proof.
\begin{flushright}
\qed
\end{flushright}
\end{proof}

Recall that the stability region of an explicit 
step-truncation method is the same as the 
corresponding method without
truncation \cite{rodgers2020stability,rodgers2020adaptive}.
%
Similarly, Theorem \ref{thm:imp-st-stable} shows that 
the stability region of an implicit 
step-truncation method is identical to the corresponding 
method without truncation, though by relaxing accuracy 
we see that instead our iterates decay to within the 
solver tolerance of zero rather than converging to zero 
in an infinite time horizon.
In other words, both implicit step-truncation Euler 
and implicit step-truncation midpoint methods 
are {\em unconditionally stable}. 
However, if the tolerance of HT/TT-GMRES is 
set too large, then we could see poor stability behavior 
akin to an explicit method.

\section{Numerical results}
\label{sec:numerics}
\noindent
In this section we study the performance of the proposed implicit 
step-truncation methods in three numerical applications 
involving time-dependent PDEs. Specifically, we study the Allen-Cahn
equation \cite{montanelli2018fourth} in two spatial dimensions,
the Fokker-Planck equation \cite{risken1989} in 
four dimensions, and the nonlinear Schr\"odinger
equation \cite{PhysRevLett.86.2353} in six dimensions.

\subsection{Allen-Cahn equation}
\label{sec:ac-eqn-example}

\noindent
The Allen-Cahn equation is a reaction-diffusion equation 
that describes the process of phase separation in 
multi-component alloy systems \cite{AllenCahn1972,AllenCahn1973}.
In its simplest form, the equation has a 
cubic polynomial non-linearity (reaction term) and a diffusion term \cite{Trefethen2005}, i.e.,
\begin{equation}
\label{eqn:allen-cahn}
\frac{\partial f}{\partial t}
=
\varepsilon \Delta f
+f - f^3.
\end{equation}
\begin{figure}[t!]
\centerline{\hspace{1cm}
\footnotesize 
ST Implicit Euler \hspace{4.6cm}
ST Implicit Midpoint\hspace{0.5cm}}
\begin{center}
\includegraphics[scale=0.5]{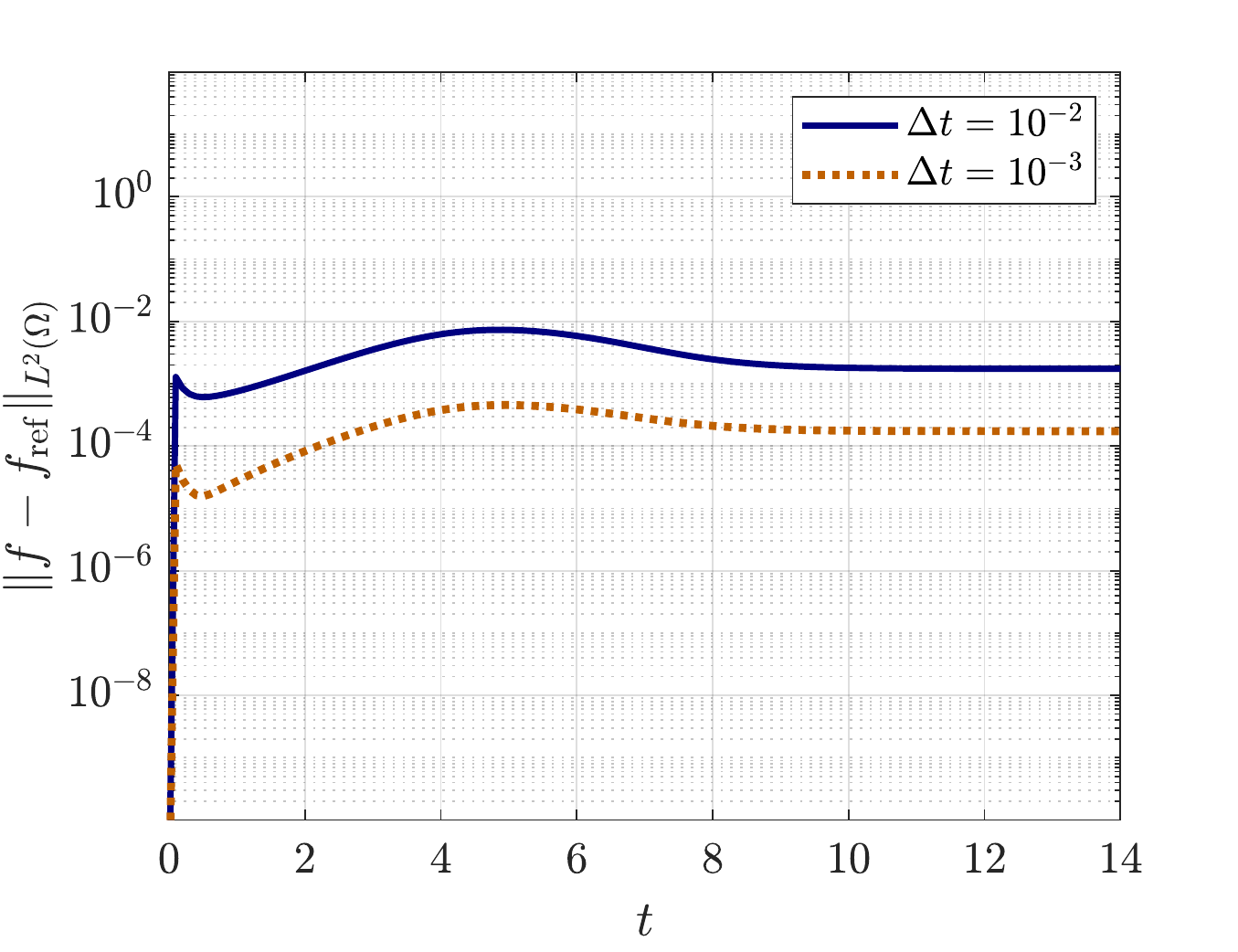}
\includegraphics[scale=0.5]{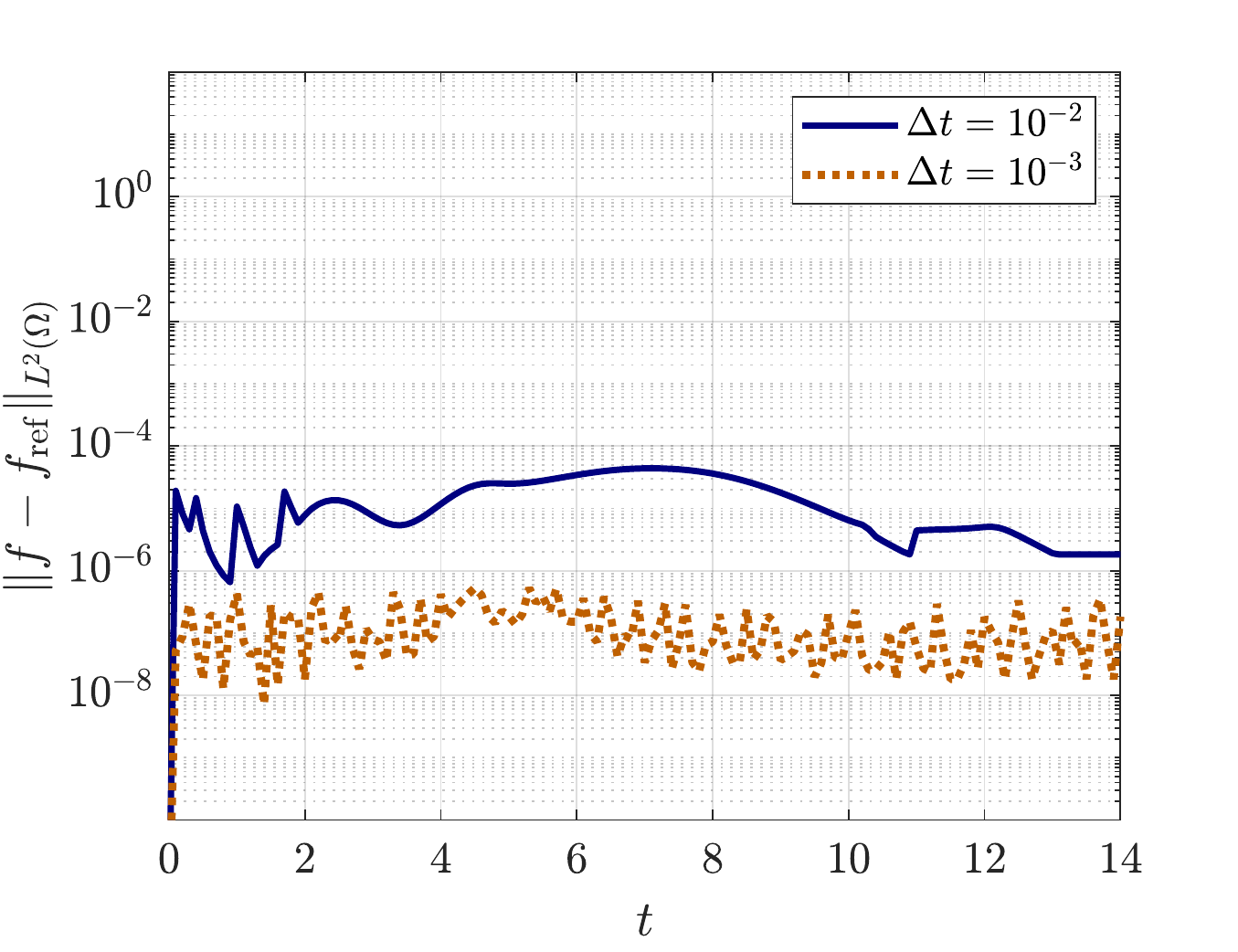}
\end{center}
\caption{ 
Error versus time for step-truncation numerical solutions
of Allen-Cahn equation \eqref{fp-lorenz-96} 
in dimension $d=2$ with initial condition  \eqref{eqn:ac-ic}. 
}
\label{fig:err-compare-ac}
\end{figure}
\begin{figure}[t]
\centerline{\hspace{1cm}
\footnotesize 
ST Implicit Euler \hspace{4.6cm}
ST Implicit Midpoint\hspace{0.5cm}}
\begin{center}
\includegraphics[scale=0.5]{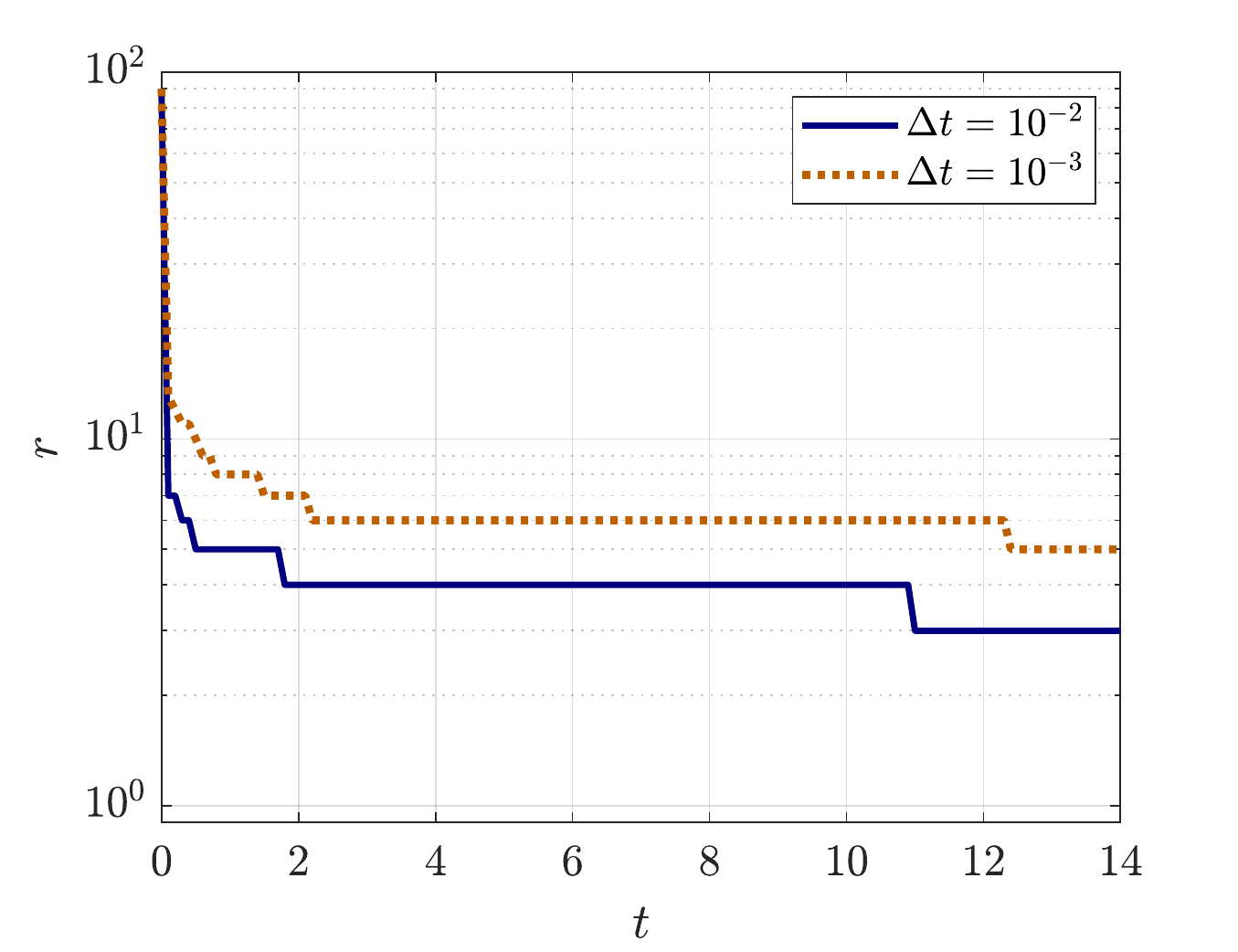}
\includegraphics[scale=0.5]{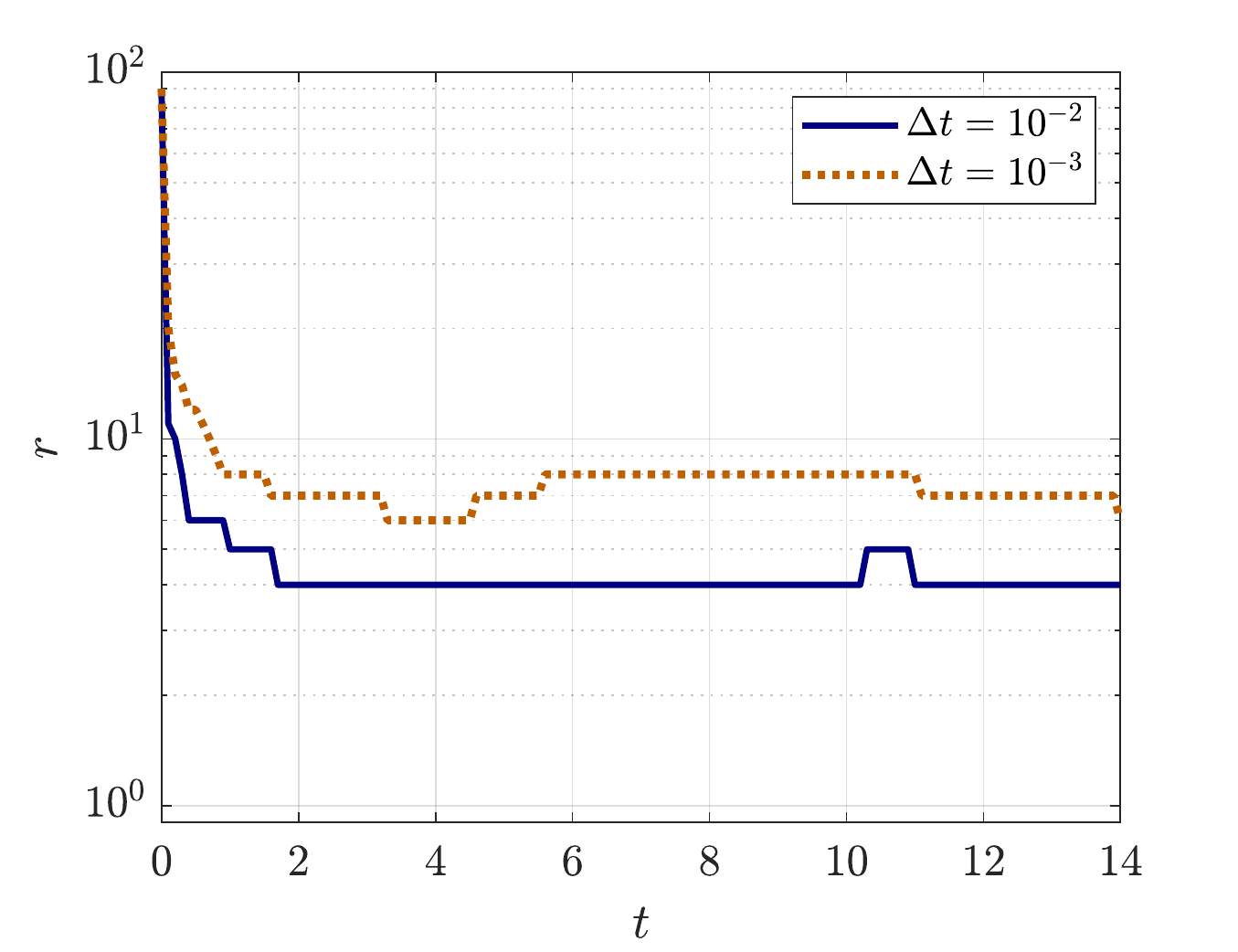}
\end{center}
\caption{ 
Rank versus time for step-truncation numerical solutions
of Allen-Cahn equation \eqref{fp-lorenz-96} 
in dimension $d=2$ with initial condition \eqref{eqn:ac-ic}.
}
\label{fig:rank-compare-ac}
\end{figure}
In our application, we set $\varepsilon = 0.1$, 
and solve \eqref{eqn:allen-cahn} on the two-dimensional
flat torus $\Omega=[0,2\pi]^2$. 
We employ a second order splitting 
\cite{farago2009operator}
method to solve the
Laplacian $\Delta f$ as a fixed rank
temporal integration and and cubic $f-f^3$ term
using our rank adaptive integration.
The initial condition is set as 
\begin{equation}
\label{eqn:ac-ic}
f_0(x,y) = 
u(x,y)- u(x,2y)+ u(3x+\pi,3y+\pi)
             -2u(4x,4y)+
             2u(5x,5y), 
\end{equation}
where
\begin{equation}
u(x,y) = \frac{\left[e^{-\tan^2(x)}+e^{-\tan^2(y)}\right]
				\sin(x)\sin(y)}{
                    1+
               e^{|\csc(-x/2)|}+e^{|\csc(-y/2)|}}.
\end{equation}
We discretize \eqref{eqn:allen-cahn} in space using the two-dimensional 
Fourier pseudospectral collocation method \cite{spectral} with $257\times 257$ points in $\Omega=[0,2\pi]^2$. 
This results in a matrix ODE in the form of \eqref{eqn:ode}.
We truncate the initial condition
to absolute and relative SVD
tolerances of $10^{-9}$, which yields an initial 
condition represented by a $257\times 257$ matrix 
of rank $90$. We also computed a benchmark solution of 
the matrix ODE using a variable step RK4 method with 
absolute tolerance set to $10^{-14}$. We denote 
the benchmark solution as $f_{\rm ref}$.
In Figure \ref{fig:err-compare-ac}
we observe the transient accuracy of our order one
and order two implicit methods. The stopping 
tolerance for inexact Newton's iterations is
set to $\varepsilon_{\rm tol}=2.2\times 10^{-8}$, while 
and HT/TT-GMRES relative
error is chosen as $\eta = 10^{-3}$. Time integration 
was halted at $t=14$. After this time, the system is close 
to steady state and the errors stay bounded near 
the final values plotted Figure \ref{fig:err-compare-ac}. 
Similarly, the rank also levels out around $t=14$.
In Figure \ref{fig:rank-compare-ac}, we plot temporal 
evolution of the rank for both the implicit 
step-truncation Euler and midpoint methods.

Due to the smoothing properties of the Laplacian, 
the high frequencies in the initial condition quickly 
decay and, correspondingly, the rank drops significantly 
within the first few time steps. Due to the
rapidly decaying rank for this problem,
we have plotted it in log scale.
In Figure \ref{fig:stiff-compare-ac}, we provide
a comparison between the rank-adaptive implicit step-truncation 
midpoint method we propose here and the rank-adaptive 
explicit step-truncation midpoint method
\begin{equation}
\label{adaptive-midpoint-method_2}
    {\bm f}_{k+1} =
    {\mathfrak{T}}_{{\bm r_3}} \left (
    {\bm f}_{k} + \Delta t
    {\mathfrak{T}}_{{\bm r_2}}\left (
    {\bm G}\left (
    {\bm f}_{k} + \frac{\Delta t}{2}
    {\mathfrak{T}}_{{\bm r_1}}(
    {\bf G}({\bm f}_{k}))
    \right )\right ) \right ), 
\end{equation}
which was recently studied in \cite{rodgers2020adaptive}.
The truncation ranks
$\bm r_1$, $\bm r_2$, and $\bm r_3$ time-dependent
and satisfy the order conditions 
\begin{equation}
\label{mid_point_trun_requirements}
\varepsilon_{{\bm r_3}} \leq A\Delta t^3, \qquad
\varepsilon_{{\bm r_2}} \leq B\Delta t^2, \qquad
\varepsilon_{{\bm r_1}}\leq G\Delta t.
\end{equation}
Such conditions guarantee that the scheme \eqref{adaptive-midpoint-method_2}
is second-order convergent (see \cite{rodgers2020adaptive}). 
Figure \ref{fig:stiff-compare-ac} shows that the explicit 
step-truncation midpoint method undergoes a numerical instability 
for $\Delta t = 10^{-3}$. Indeed it is a conditionally stable method. 
The explicit step-truncation midpoint method also has other issues. 
In particular, in the rank-adaptive setting we consider here, 
we have that in the limit $\Delta t \rightarrow 0$ the 
parameters $\varepsilon_{{\bm r_1}}$, $\varepsilon_{{\bm r_2}}$ and 
$\varepsilon_{{\bm r_3}}$ all go to zero (see equation \eqref{mid_point_trun_requirements}). This implies that 
the truncation operators may retain all singular values, 
henceforth maxing out the rank and thereby giving
up all computational gains of low-rank tensor compression. 
On the other hand, if $\Delta t$ is too large, then one we 
have stability issues as discussed above.  Indeed, we 
see both these problems with the 
explicit step-truncation  midpoint method, giving only a 
relatively narrow region of acceptable time step sizes in 
which the method is effective.

In Table \ref{tab:comparison} we provide a
comparison between explicit and implicit step-truncation 
midpoint methods in terms of computational cost 
(CPU-time on an Intel Core I9-7980XE workstation) 
and accuracy at time $t=14$. It is seen that the 
implicit step-truncation midpoint method is roughly 20 to 30 
times faster than the explicit step-truncation midpoint method 
for a comparable error\footnote{Our code was built on the backbone of 
the HTucker Matlab package \cite{kressner2014algorithm}, 
and was not optimized for speed. Faster run times for both 
explicit and implicit step-truncation methods
are possible by utilizing scalable high-performance 
algorithms such as those described in \cite{daas2020parallel}.}.
Moreover, solutions with a large time step ($>10^{-3}$) 
are impossible to achieve with the explicit step-truncation 
method due to time step restrictions associated with 
conditional stability.

\begin{figure}[t!]
\centerline{\hspace{1cm}
\footnotesize 
ST Explicit Midpoint \hspace{4.0cm}
ST Implicit Midpoint\hspace{1cm}}
\centerline{\line(1,0){420}}
\begin{center}
\includegraphics[scale=0.5]{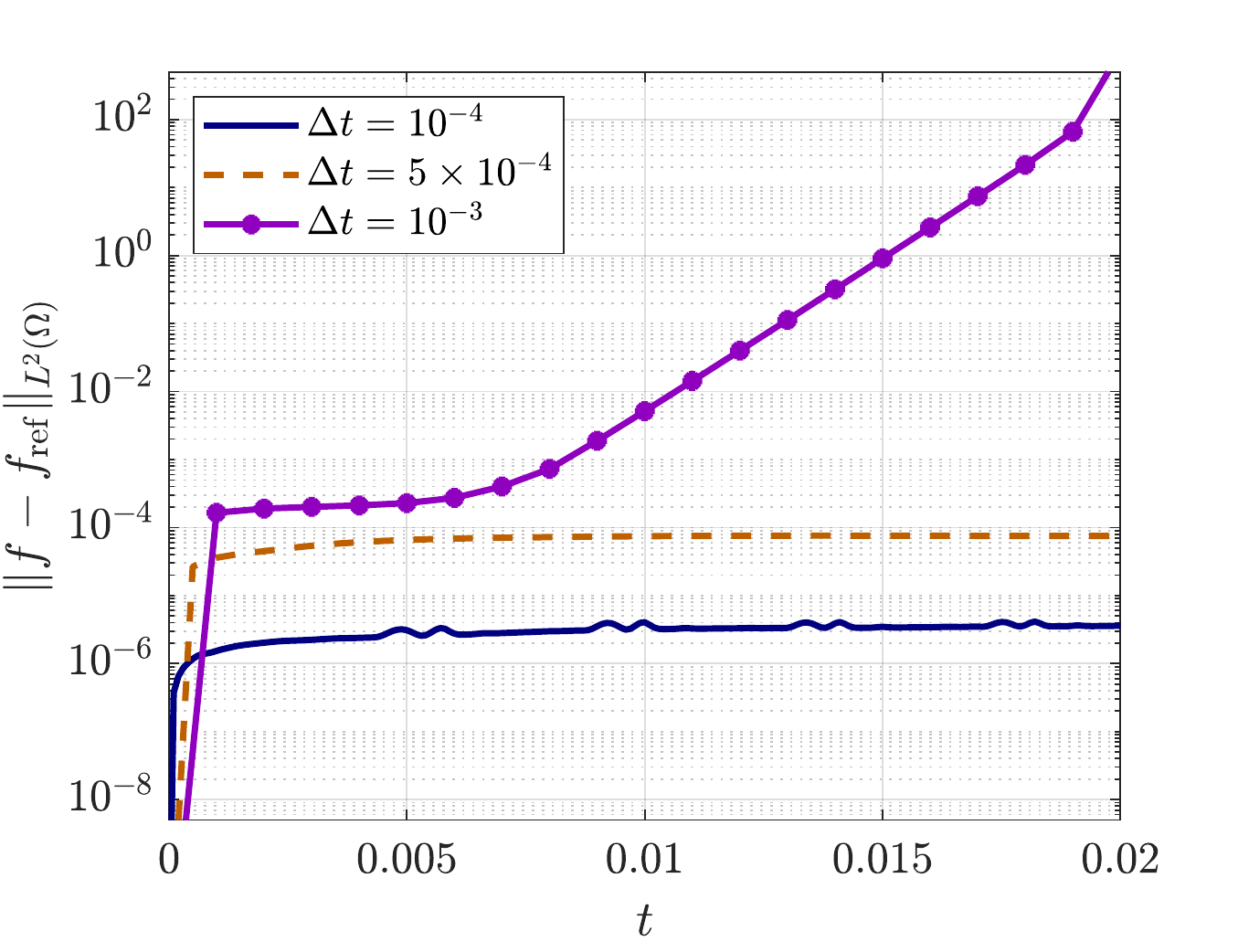}
\includegraphics[scale=0.5]{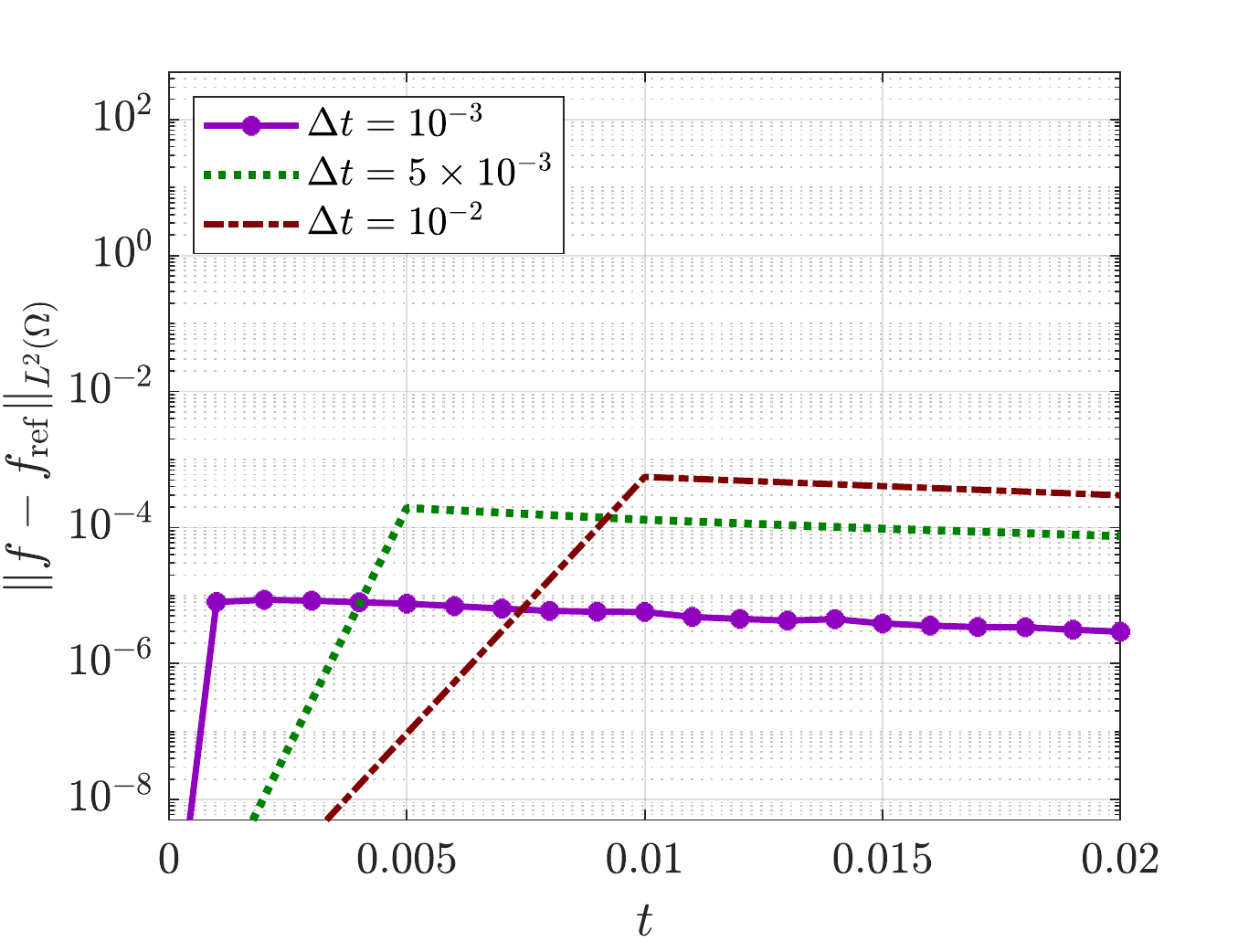}
\end{center}
\caption{ 
Allen-Cahn equation \eqref{eqn:allen-cahn}. Comparison 
between the $L^2(\Omega)$ errors of explicit and implicit step-truncation 
midpoint methods for different $\Delta t$. 
}
\label{fig:stiff-compare-ac}
\end{figure}

\begin{table}
\begin{tabular}{|c|c|c|}
 \hline
 \multicolumn{3}{|c|}{ST Explicit Midpoint} \\
 \hline
 $\Delta t$
 &Runtime (seconds) 
 &$\|{\bm f} - {\bm f}_{\rm {ref}}\|$\\
 \hline
 $1.0\times 10^{-3}$ & Did not finish & Unstable\\
 $5.0\times 10^{-4}$ & $2.2946\times 10^{2}$ & $6.0713\times10^{-3}$\\
 $2.5\times 10^{-4}$ & $4.7828\times 10^{2}$ & $5.2628\times10^{-4}$\\
 $1.0\times 10^{-4}$ & $1.2619\times 10^{3}$ & $5.3648\times10^{-5}$\\
 $5.0\times 10^{-5}$ & $2.7354\times 10^{3}$ & $1.1723\times10^{-5}$\\
 \hline
\end{tabular}
\begin{tabular}{|c|c|c|}
 \hline
 \multicolumn{3}{|c|}{ST Implicit Midpoint} \\
 \hline
 $\Delta t$
 & Runtime (seconds)  
 &$\|{\bm f} - {\bm f}_{\rm {ref}}\|$\\
 \hline
 $1.0\times 10^{-1}$ & $5.3097$ & $2.5652\times10^{-2}$\\
 $5.0\times 10^{-2}$ & $1.0495\times 10^{1}$ & $7.7248\times10^{-3}$\\
 $2.5\times 10^{-2}$ & $1.8987\times 10^{1}$ & $2.0977\times10^{-3}$\\
 $1.0\times 10^{-2}$ & $3.8025\times 10^{1}$ & $6.7477\times10^{-6}$\\
 $5.0\times 10^{-3}$ & $7.6183\times 10^{1}$ & $3.6012\times10^{-6}$\\
 \hline
\end{tabular}
\caption{Allen-Cahn equation \eqref{eqn:allen-cahn}. Comparison between 
explicit and implicit step-truncation midpoint methods in terms of 
computational cost (CPU-time on an Intel Core I9-7980XE workstation) 
and accuracy at final time ($t=14$). It is seen that the 
implicit step-truncation midpoint method is roughly 20 to 30 
times faster than the explicit step-truncation midpoint method 
for a comparable error.}
\label{tab:comparison}
\end{table}

\subsection{Fokker-Planck equation}

\begin{figure}[!htb]
\centering
\centerline{\hspace{1.5cm}
\footnotesize 
ST implicit Euler \hspace{1.8cm}
ST implicit midpoint    \hspace{2.2cm}
Reference}
\centerline{\line(1,0){420}}
\vspace{0.1cm}
{\rotatebox{90}{\hspace{1.6cm}\rotatebox{-90}{
\hspace{0.1cm}
\footnotesize$t=0$\hspace{0.9cm}}}}
\includegraphics[scale=0.27]{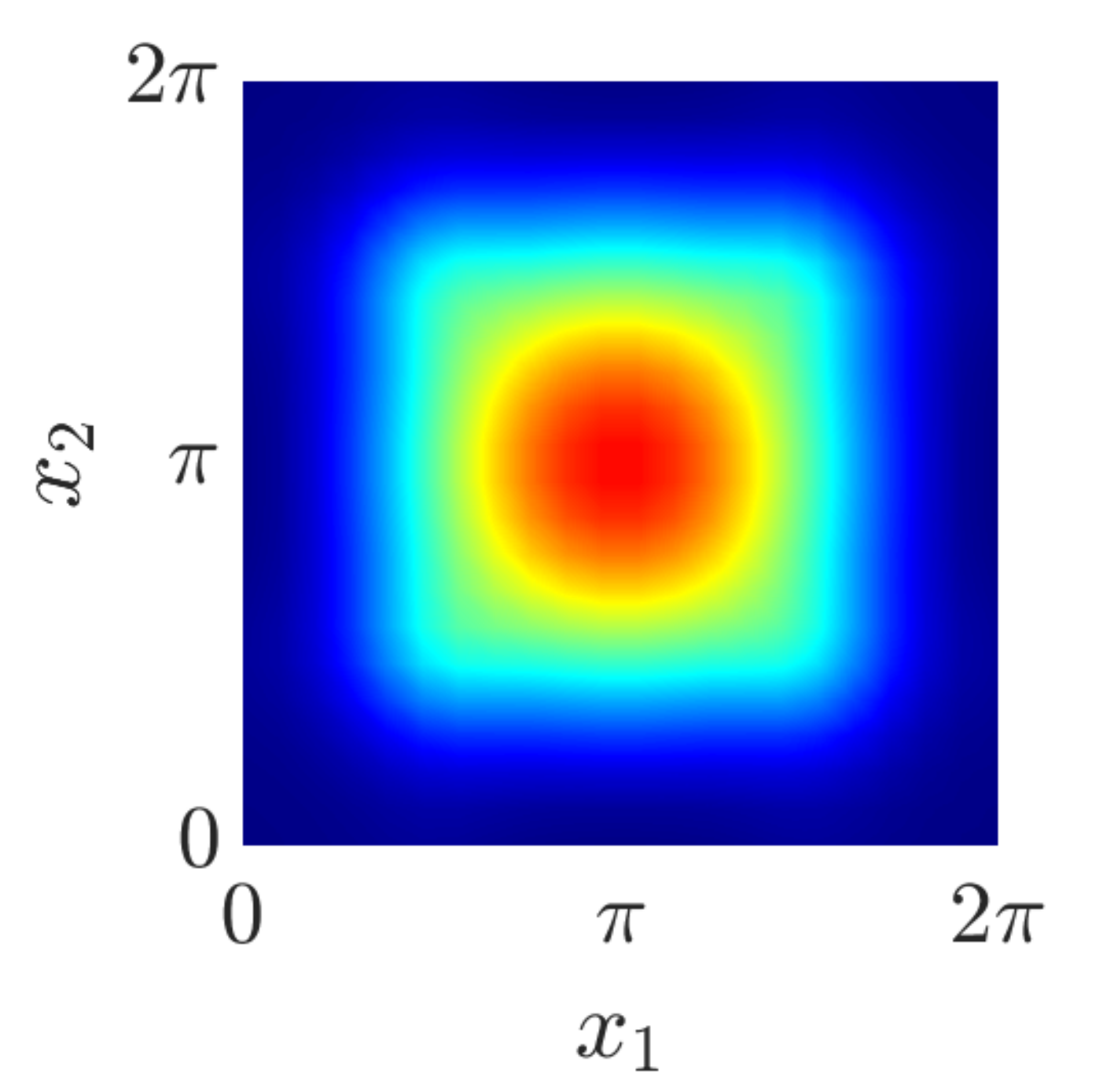}
\includegraphics[scale=0.27]{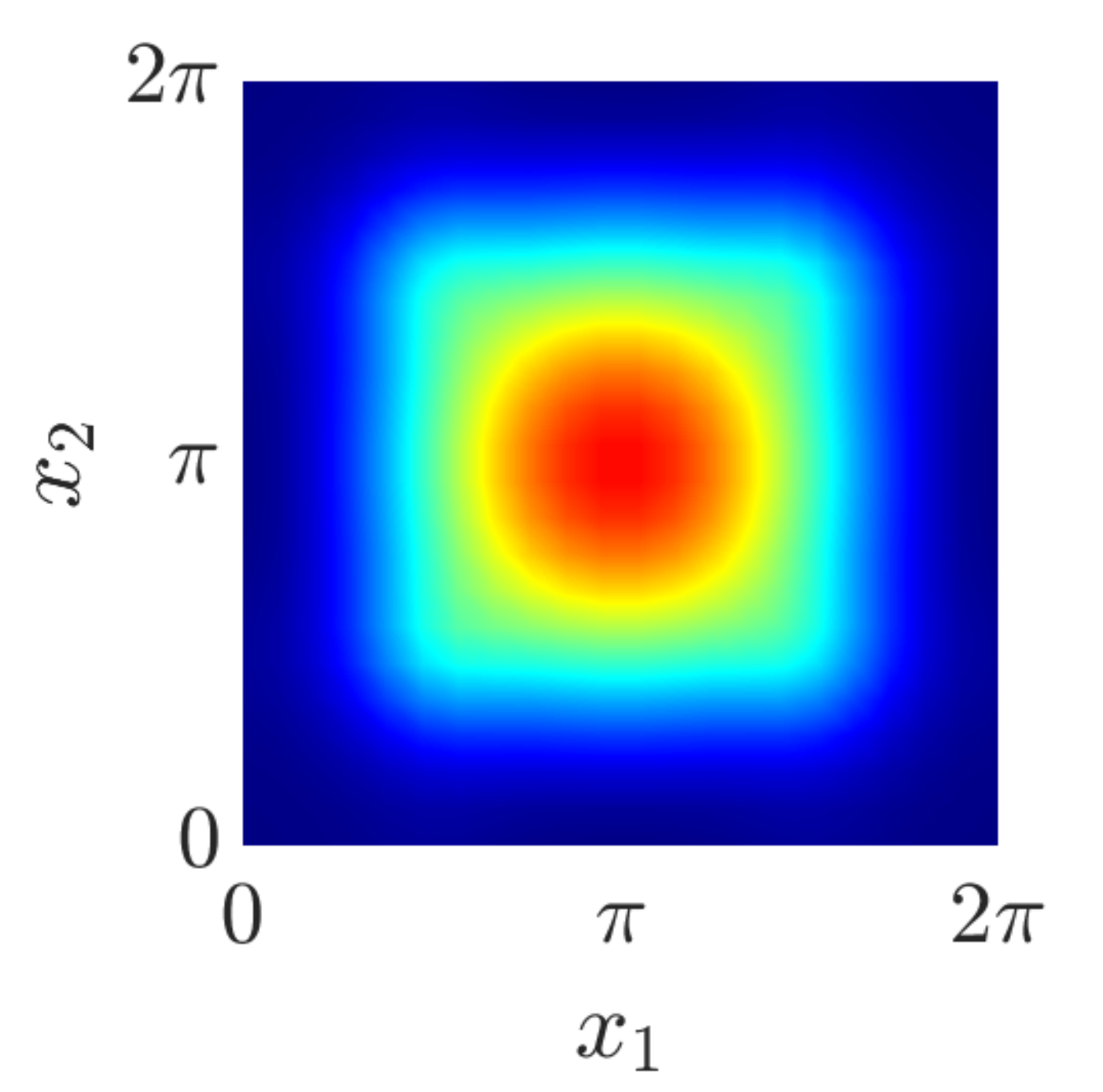}
\includegraphics[scale=0.27]{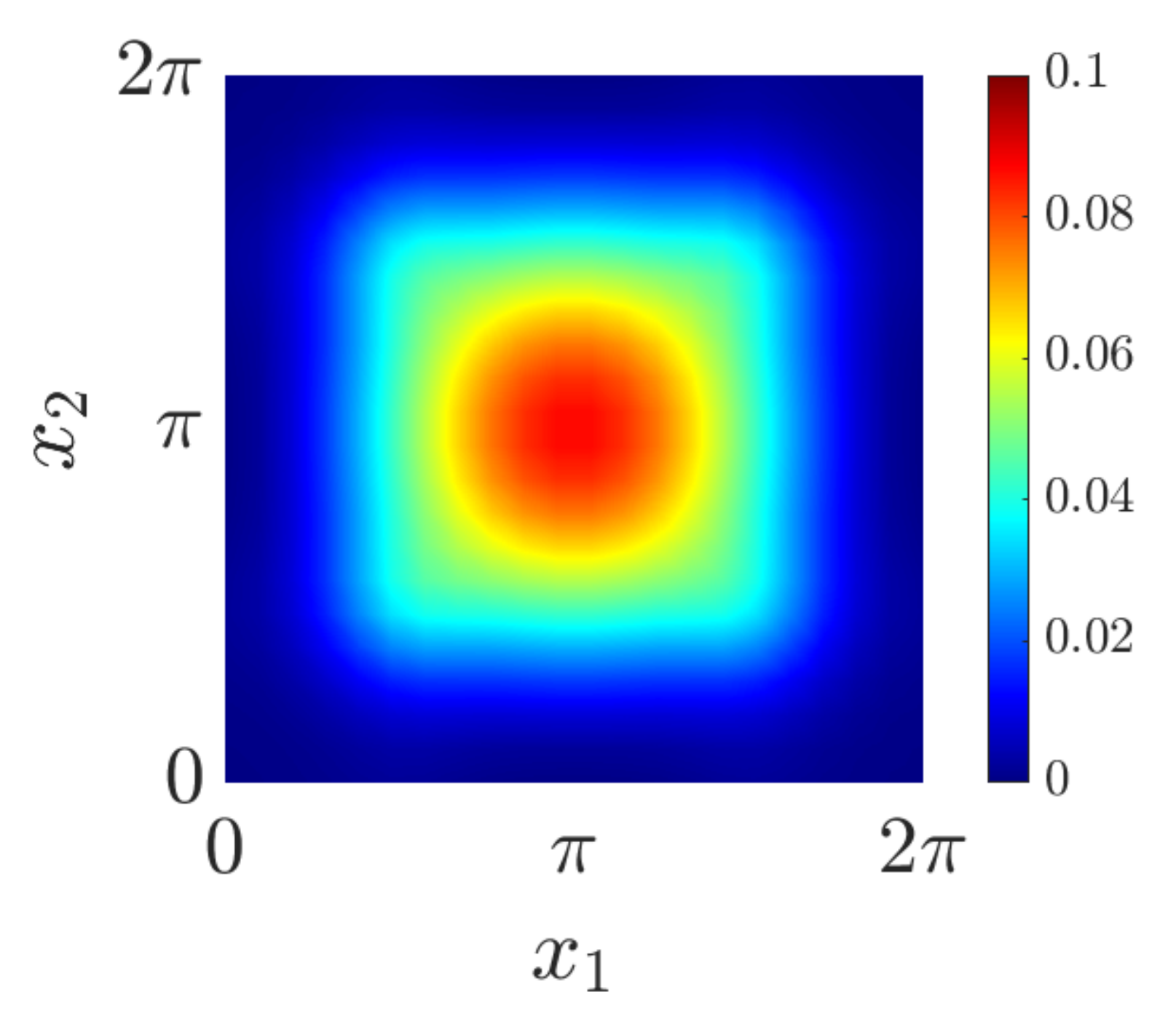}\\
{\rotatebox{90}{\hspace{1.6cm}\rotatebox{-90}{
\footnotesize$t=0.1$\hspace{0.7cm}}}}
\includegraphics[scale=0.27]{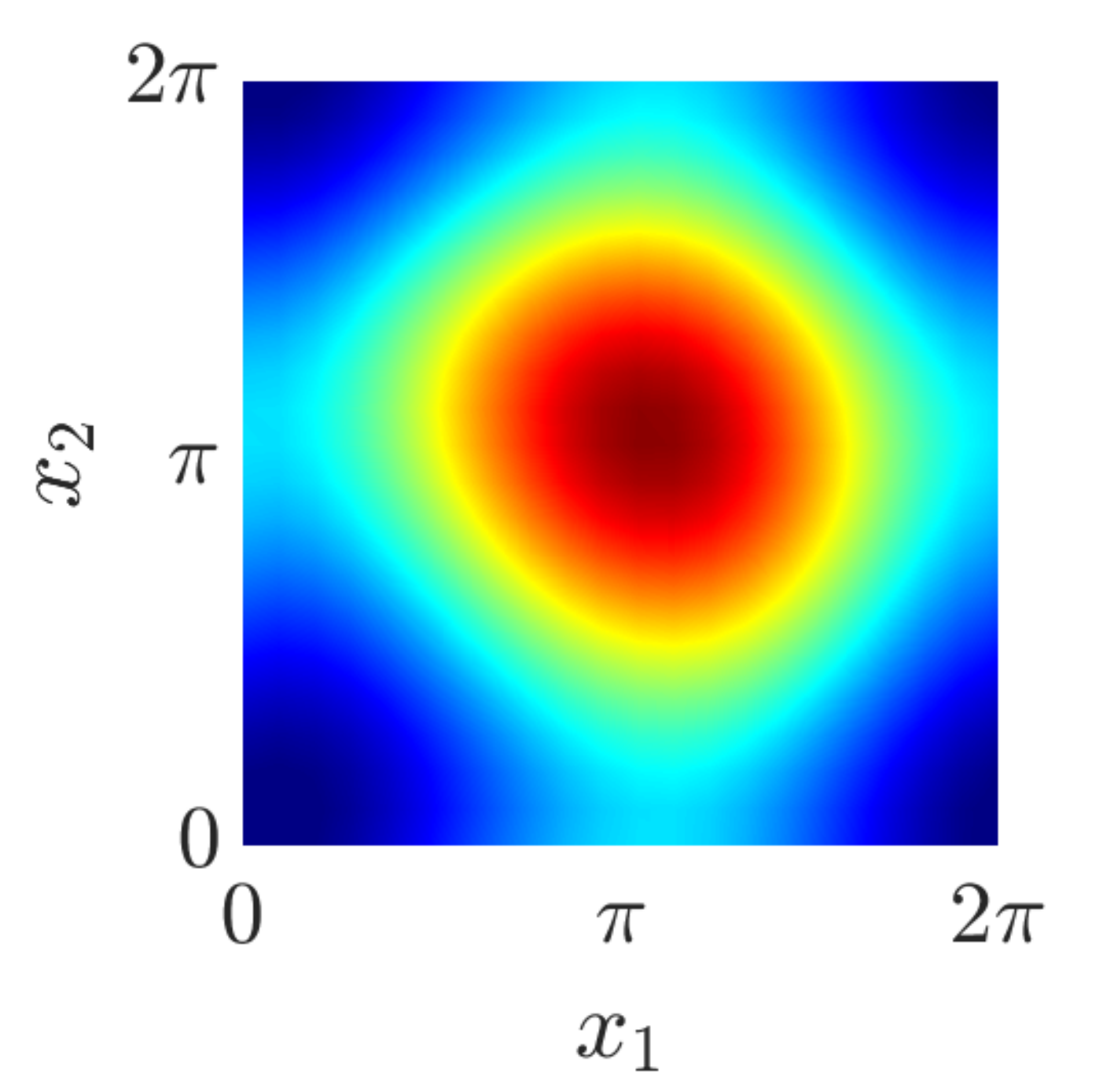}
\includegraphics[scale=0.27]{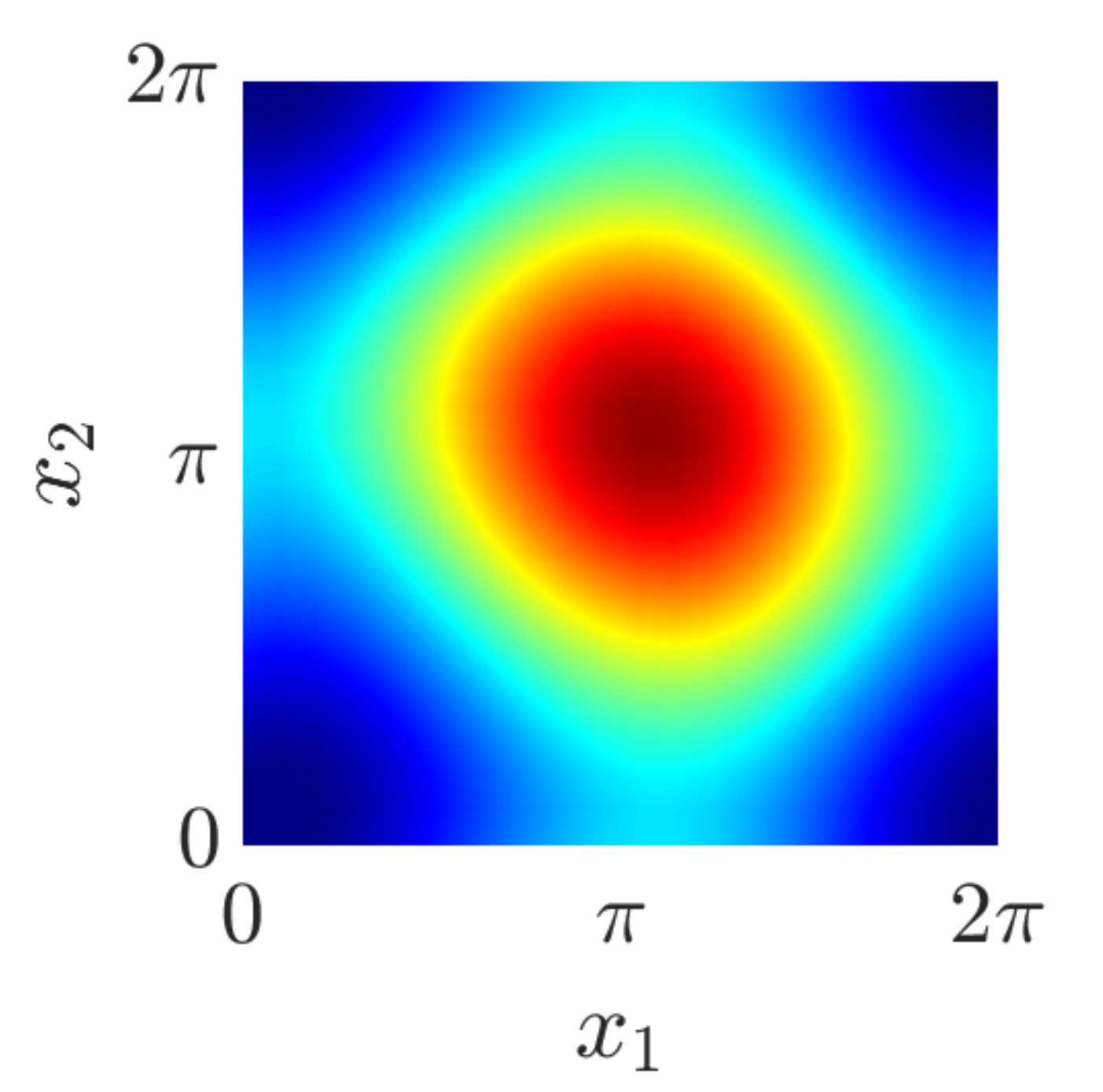}
\includegraphics[scale=0.27]{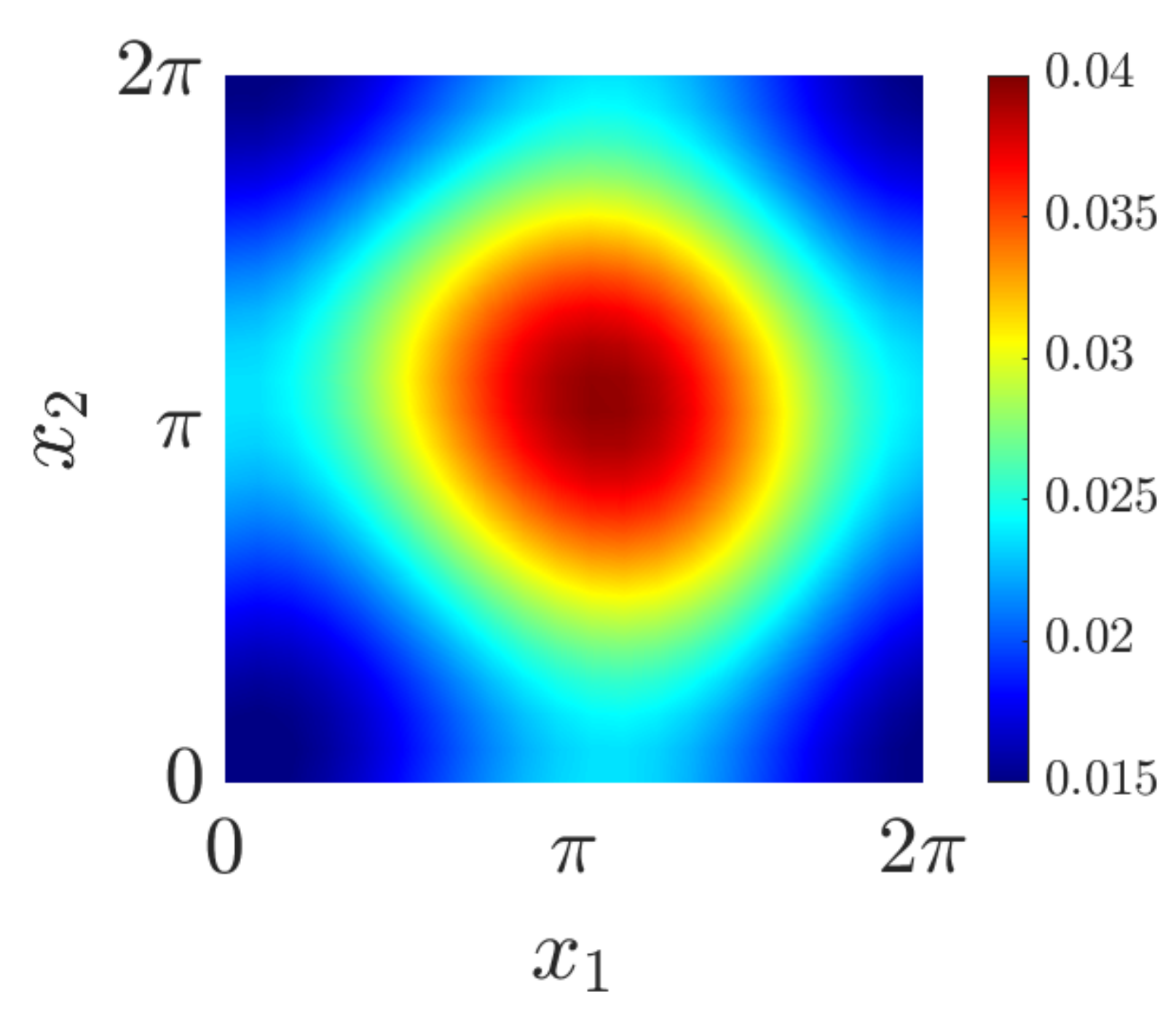}\\
{\rotatebox{90}{\hspace{1.6cm}\rotatebox{-90}{
\footnotesize$t=10$\hspace{0.7cm}}}}
\includegraphics[scale=0.27]{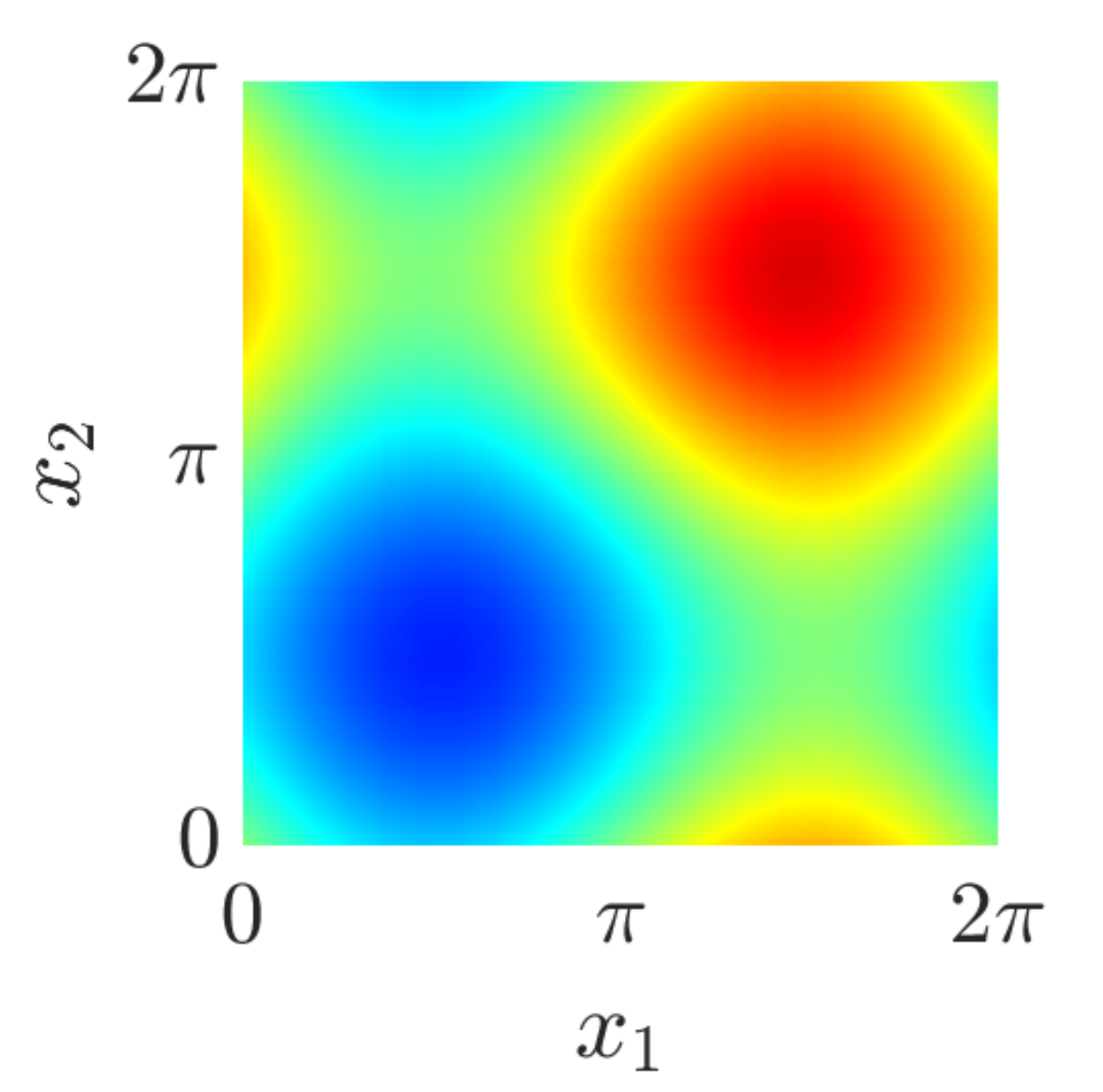}
\includegraphics[scale=0.27]{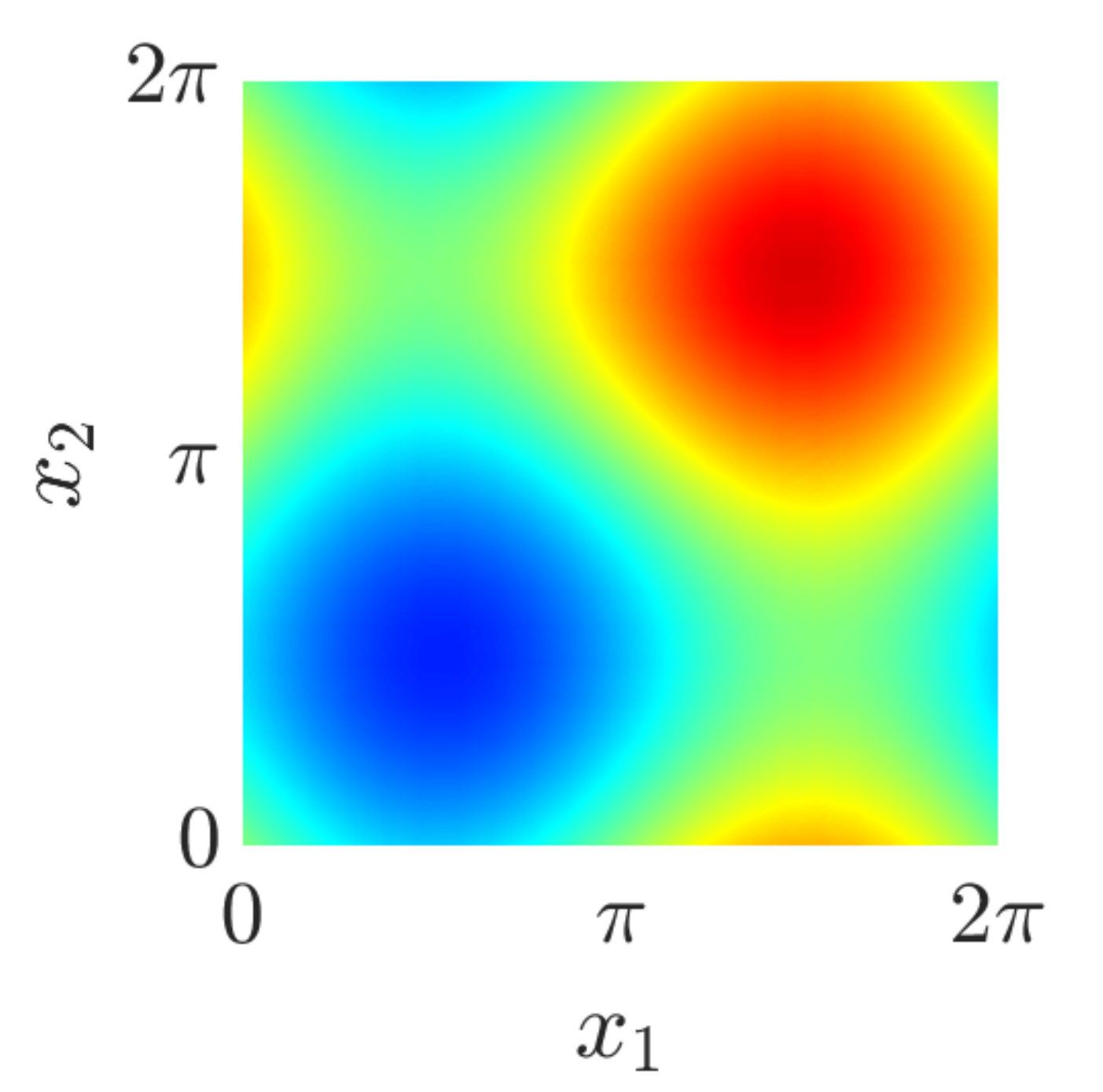}
\includegraphics[scale=0.27]{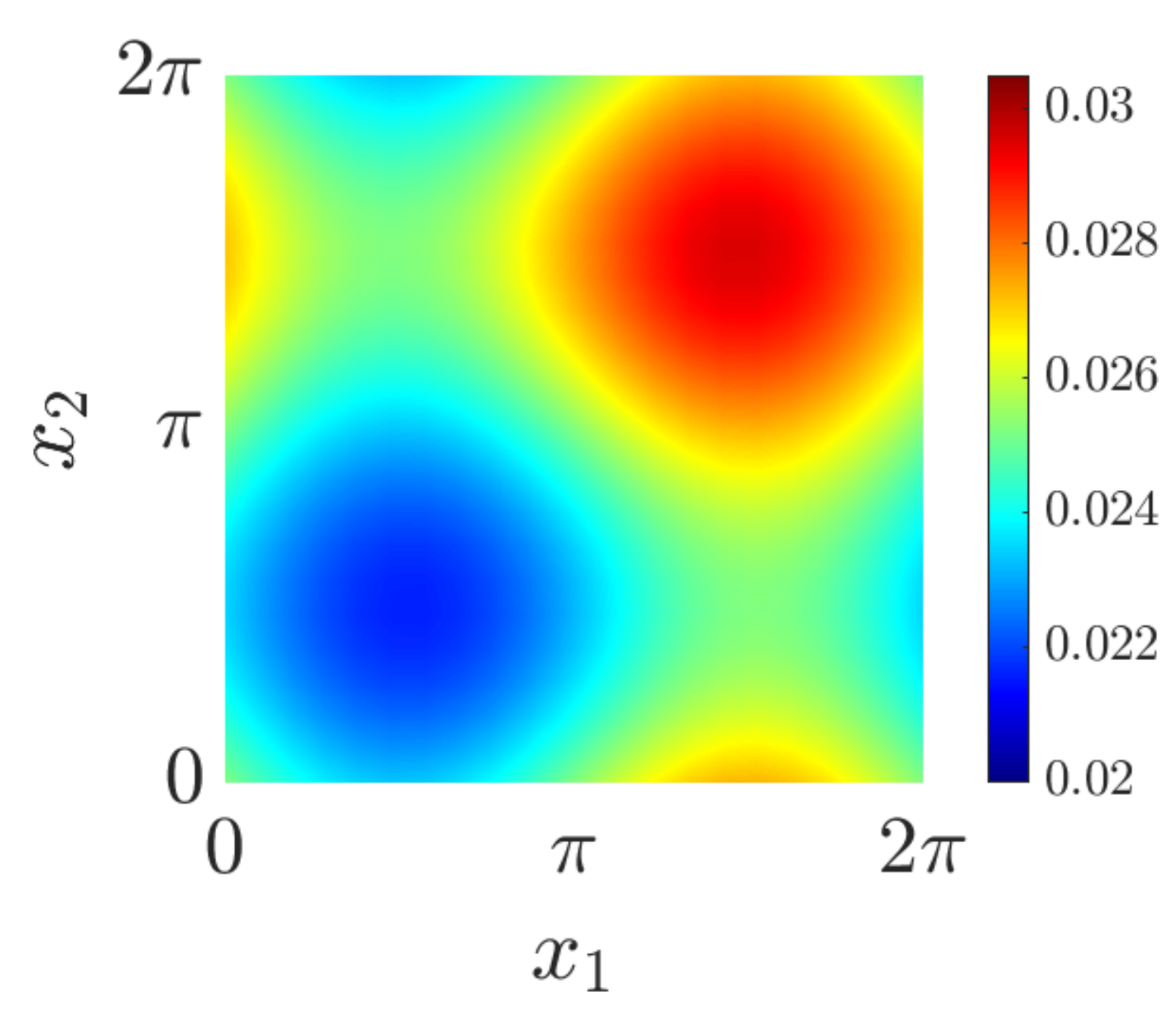}\\
\caption{
Marginal probability density function \eqref{marginal}
obtained by integrating numerically the 
Fokker--Planck equation \eqref{fp-lorenz-96} 
in dimension $d=4$ with $\sigma = 5$ and
initial condition \eqref{ic4d} with two methods:   
i) rank-adaptive implicit step-truncation Euler and 
ii) rank-adaptive implicit step-truncation midpoint.
The reference solution is a variable time step
RK4 method with absolute tolerance of $10^{-14}$.
These solutions are computed  on a grid with 
$20\times 20\times 20\times 20$ interior points 
(evenly spaced).
The steady state is determined for this 
computation by halting execution 
when $\left\|\partial f_{\text{ref}}/\partial t\right\|_2$ 
is below a numerical threshold of $10^{-8}$. This
happens at approximately $t \approx 10$ for the 
initial condition \eqref{ic4d}.}
\label{fig:time-snapshot-plot-fp}
\end{figure}

\begin{figure}[!htb]
\centerline{\hspace{1cm}
\footnotesize 
Transient Error \hspace{5.3cm}
Maximal Rank\hspace{1.0cm}}
\centerline{\line(1,0){420}}
\begin{center}
\includegraphics[scale=0.5]{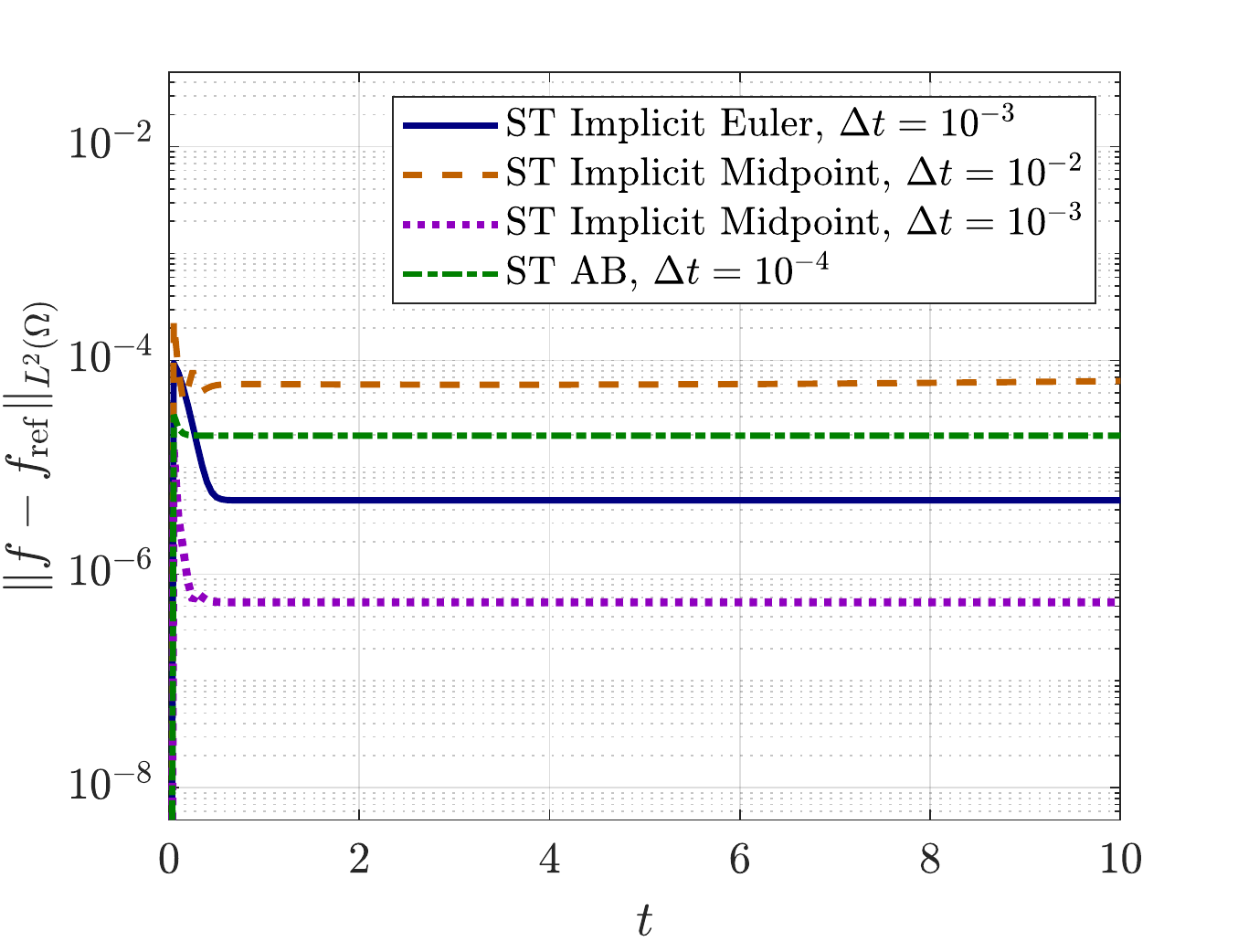}
\includegraphics[scale=0.5]{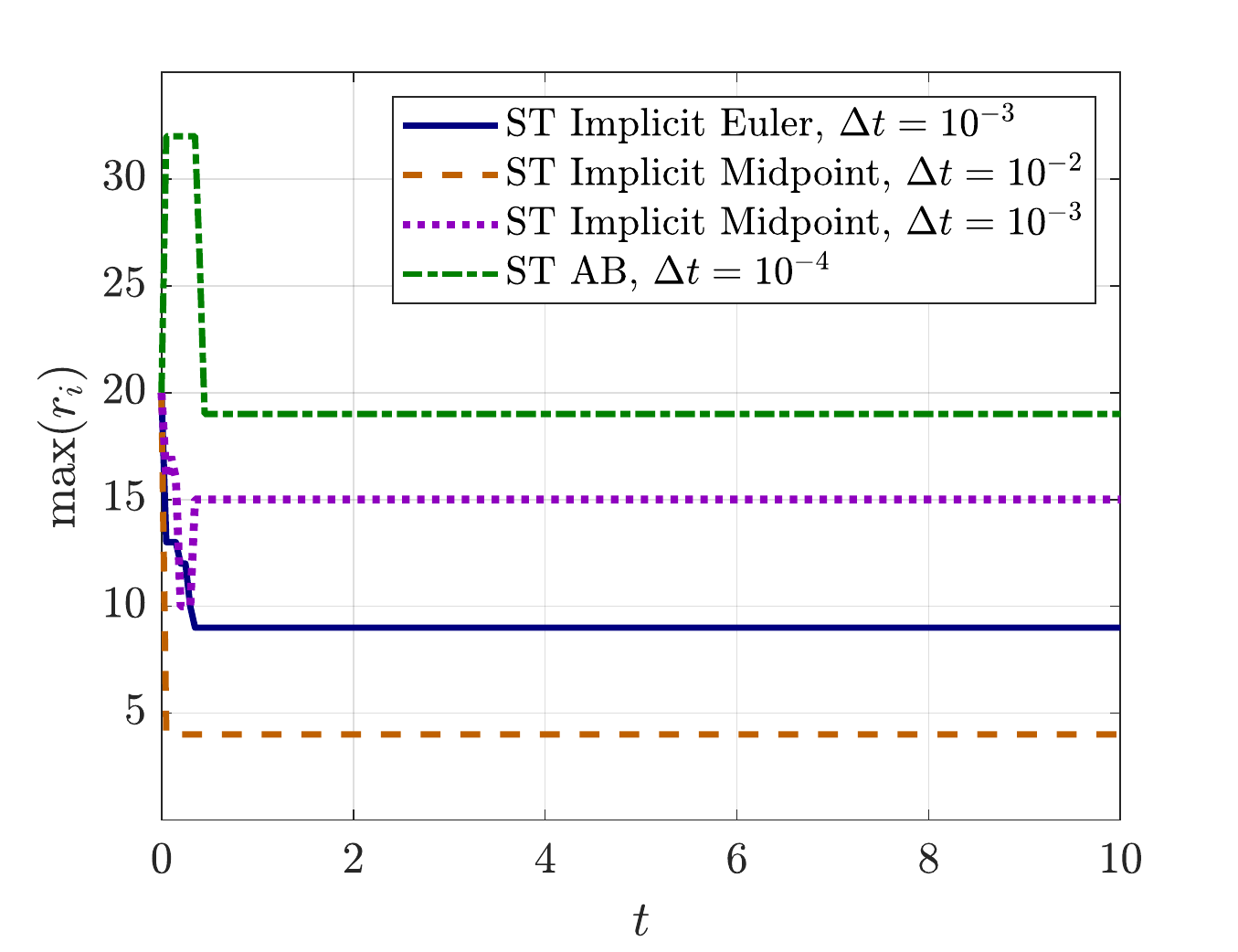}
\end{center}
\caption{$L^2(\Omega)$ error 
and rank versus time  for
numerical solutions of 
Fokker--Planck equation \eqref{fp-lorenz-96} 
in dimension $d=4$ with
initial condition \eqref{ic4d}.
The rank plotted here is the largest rank for all tensors
being used to represent the solution in HT format. 
Rank of the reference solution is in HT format.
}
\label{fig:error-compare-fp}
\end{figure}

\begin{figure}[!htb]
\centerline{ 
\includegraphics[scale=0.5]{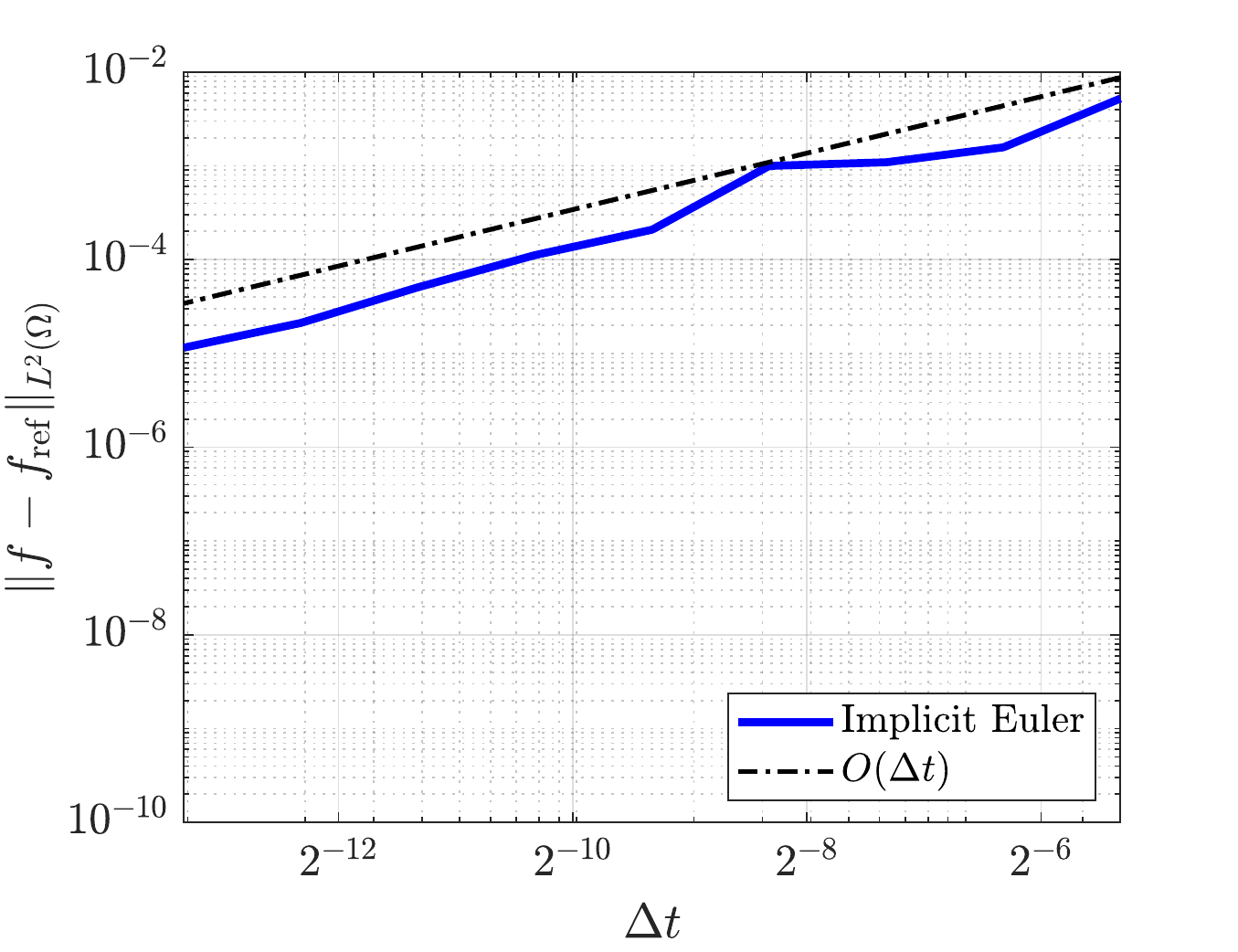}
\includegraphics[scale=0.5]{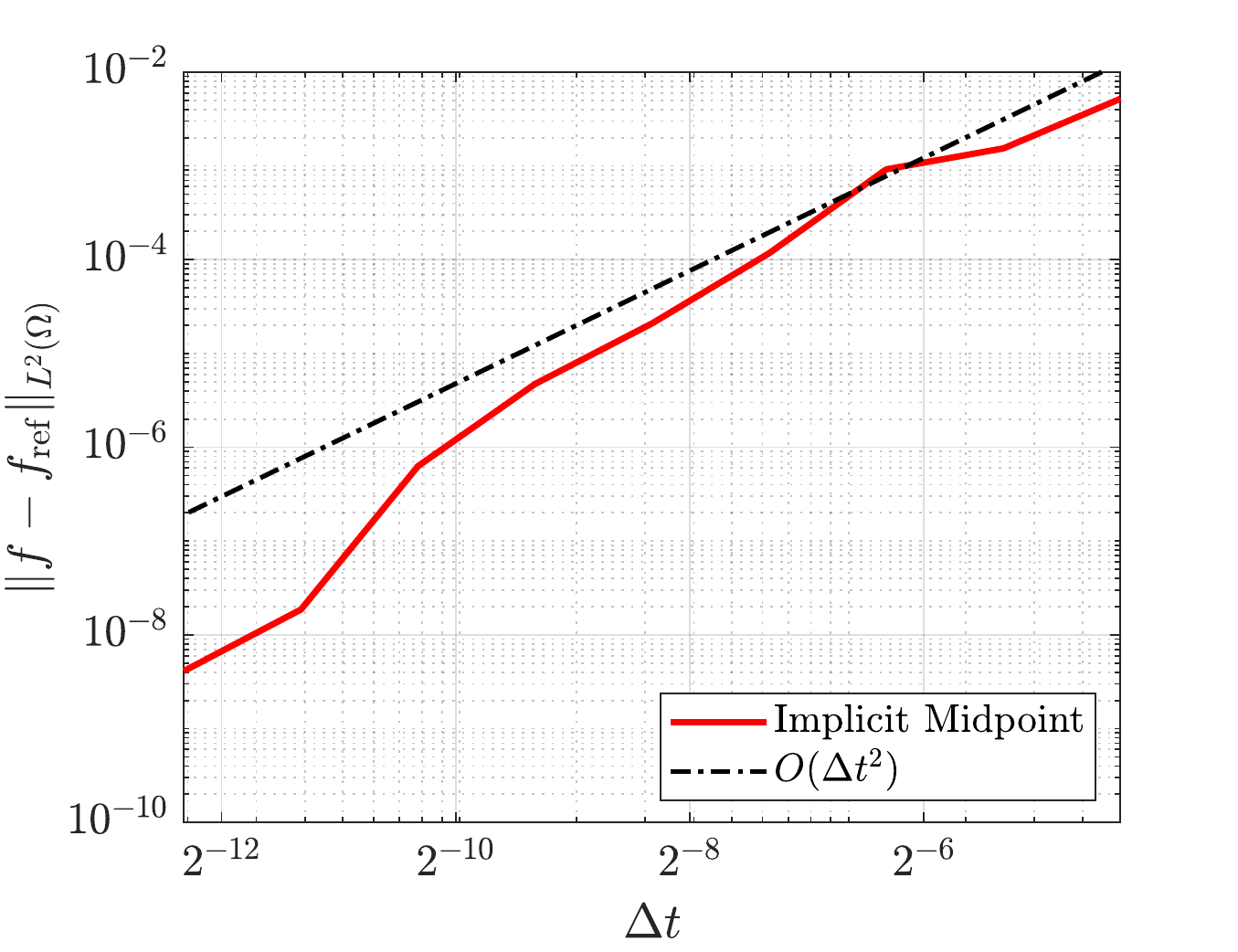}
}
\caption{$L^2(\Omega)$ errors at $t=0.1$ for the implicit rank-adaptive 
step-truncation implicit Euler and midpoint methods versus $\Delta t$. 
The reference solution of \eqref{fp-lorenz-96} if computed 
using a variable time step RK4 method with absolute tolerance of $10^{-14}$.
}
\label{fig:global-err-fp}
\end{figure}

\noindent
Consider the Fokker-Planck equation
\begin{equation}
\label{fp-lorenz-96}
\frac{{\partial} f({\bm x},t) }{\partial t}=
-\sum_{i=1}^{d}
\frac{\partial}{\partial x_i}\left(
\mu_i({\bm x}) f({\bm x},t) \right)
+ \frac{\sigma^2}{2} 
\sum_{i=1}^{d}\frac{\partial^2 f({\bm x},t)}{\partial x_i^2}
\end{equation}
on a four-dimensional ($d=4$) flat torus $\Omega=[0,2\pi]^4$. 
The components of the drift are chosen as 
\begin{equation}
\label{drift-field}
\mu_i({\bm x}) = 
(\gamma(x_{i+1}) - \gamma(x_{i-2}))\xi(x_{i-1}) -
\phi(x_{i}), \qquad i =1,\ldots, d,
\end{equation}
where
$\gamma(x)=\sin(x)$, $\xi(x)=\exp(\sin(x))+1$,
and $\phi(x)=\cos(x)$ are $2\pi$-periodic functions. 
In \eqref{drift-field} we set $x_{i+d}=x_i$. For this particular
drift field, the right side of \eqref{fp-lorenz-96} can be split
into a component tangential to the tensor manifold ${\cal H}_{\bm r}$ and 
a component that is non-tangential as
\begin{align}
\frac{{\partial} f({\bm x},t) }{\partial t}=
\sum_{i=1}^{d}
\underbrace{
\left (
-\gamma(x_{i+1})\xi(x_{i-1})
\frac{\partial f({\bm x},t)}{\partial x_i}
+
\gamma(x_{i-2})\xi(x_{i-1})
\frac{\partial f({\bm x},t)}{\partial x_i}
\right )}_{\text{Not tangential}} 
+ \nonumber\\
\underbrace{
\left (
\frac{\partial}{\partial x_i}
\phi(x_{i})f({\bm x},t)+
\frac{\sigma^2}{2} 
\frac{\partial^2 f({\bm x},t)}{\partial x_i^2}
\right )
}_{\text{tangential}}.
\label{fp-lorenz-96-split}
\end{align}
We solve \eqref{fp-lorenz-96-split} using an operator splitting 
method. To this end, we notice that there are $3d$ many terms 
in the summation above, and therefore we first solve the first $d$ 
time dependent PDEs which are tangential to 
the tensor manifold ${\cal H}_{\bm r}$, i.e.,
\begin{align}
\frac{\partial g_{i}}{\partial t} =
\frac{\partial }{\partial x_i}\phi(x_i)g_i
+ \frac{\sigma^2}{2}
\frac{\partial^2 g_i}{\partial x_i^2}
,\quad 
\ i=1,\dots,d.
\end{align}
Then we solve the non-tangential equations
in two batches,
\begin{align}
\frac{\partial u_{j}}{\partial t} &=
\gamma(x_{i+1})\xi(x_{i-1})
\frac{\partial u_{j}}{\partial x_i},
\quad\ 
\ j =1,\dots,d\\
\frac{\partial u_{k}}{\partial t} &=
\gamma(x_{i-2})\xi(x_{i-1})
\frac{\partial u_{k}}{\partial x_i},
\quad
\ k = 2,\dots,2d.
\end{align}
This yields the first-order (Lie-Trotter) approximation
$f({\bm x},\Delta t) = u_{2d}({\bm x},\Delta t)
+O(\Delta t^2)$. We also use these same list of 
PDEs for the second-order (Strang) splitting integrator.
For each time step in both first- and second-order splitting 
methods, we terminate the HT/TT-GMRES iterations by setting 
the stopping tolerance $\varepsilon_{\text{tol}}=10^{-9}$.
We set the initial probability density function (PDF) as
\begin{equation}
 f_0(x_1,x_2,x_3,x_4) = \frac{1}{F_0}\sum_{j=1}^M
 \left ( \ \prod_{i=1}^4
                    \frac{\sin((2j-1)x_i-\pi/2)+1}{2^{2(j-1)}} +
                    \prod_{i=1}^4
                    \frac{\exp(\cos(2jx_i+\pi))}{2^{2j-1}}
                    \right ),
                    \label{ic4d}
\end{equation}
where $F_0$ is a normalization constant. This
gives an HTucker tensor with rank bounded by $2M$.
We set $M=10$ to give ranks bounded by $20$.
We discretize \eqref{fp-lorenz-96}-\eqref{ic4d} in $\Omega$ 
with the Fourier pseudospectral collocation method \cite{spectral}
on a tensor product grid with $N=20$ evenly-spaced points 
along each coordinate $x_i$, giving the total 
number of points $(N+1)^4 = 194481$. This number corresponds 
to the number of entries in the tensor ${\bm f}(t)$ appearing in
equation \eqref{eqn:ode}. Also, we set $\sigma = 5$ 
in \eqref{fp-lorenz-96}.

In Figure \eqref{fig:time-snapshot-plot-fp}
we compare a few time snapshots of the
marginal PDF
\begin{equation}
\label{marginal}
f_{12}(x_1,x_2) = \int_0^{2\pi}\int_0^{2\pi}
	f(x_1,x_2,x_3,x_4)dx_3dx_4,
\end{equation}
we obtained with the rank-adaptive implicit 
step-truncation Euler and midpoint methods, 
as well as the reference marginal PDF. 
The solution very quickly relaxes to nearly
uniform by $t=0.1$, then slowly
rises to its steady state distribution
by $t=10$.

In Figure \ref{fig:error-compare-fp} we study 
accuracy and rank of the proposed implicit step-truncation 
methods in comparison with the rank-adaptive explicit 
Adams-Bashforth (AB) method of order 2 (see \cite{rodgers2020adaptive}). 
The implicit step-truncation methods use a time 
step size of $\Delta t = 10^{-3}$, while step-truncation AB2 
uses a step size of $\Delta t = 10^{-4}$.
The highest error and lowest rank come from the
implicit step-truncation midpoint method with $\Delta t=10^{-2}$.
The highest rank and second highest error
go to the step-truncation AB2 method, which 
runs with time step size $10^{-4}$ for stability. 
This causes a penalty in the rank,
since the as the time step is made small, the
rank must increase to maintain convergence order 
(see, e.g., \eqref{mid_point_trun_requirements} for similar 
conditions on explicit step-truncation midpoint method)
The implicit step-truncation midpoint method performs the best,
with error of approximately $10^{-6}$ and rank
lower than the step-truncation AB2 method at steady state.
Overall, the proposed implicit step-truncation methods perform
extremely well on linear problems of this form,
especially when the right hand side is explicitly
written as a sum of tensor products of
one dimensional operators.
In Figure \eqref{fig:global-err-fp} we show
a plot of the convergence rate of implicit step-truncation 
Euler and midpoint methods. For this figure, we set $\sigma=2$.
The convergence rates are order one and order two
respectively, verifying 
Theorem \ref{thm:convergence}.

\subsection{Nonlinear Schr\"odinger equation}
\begin{figure}[t]
\centerline{\hspace{1cm}
\footnotesize 
$\theta=0.1$
\hspace{6cm}
$\theta=0.01$\hspace{1.0cm}}
\centering
\includegraphics[scale=0.5]{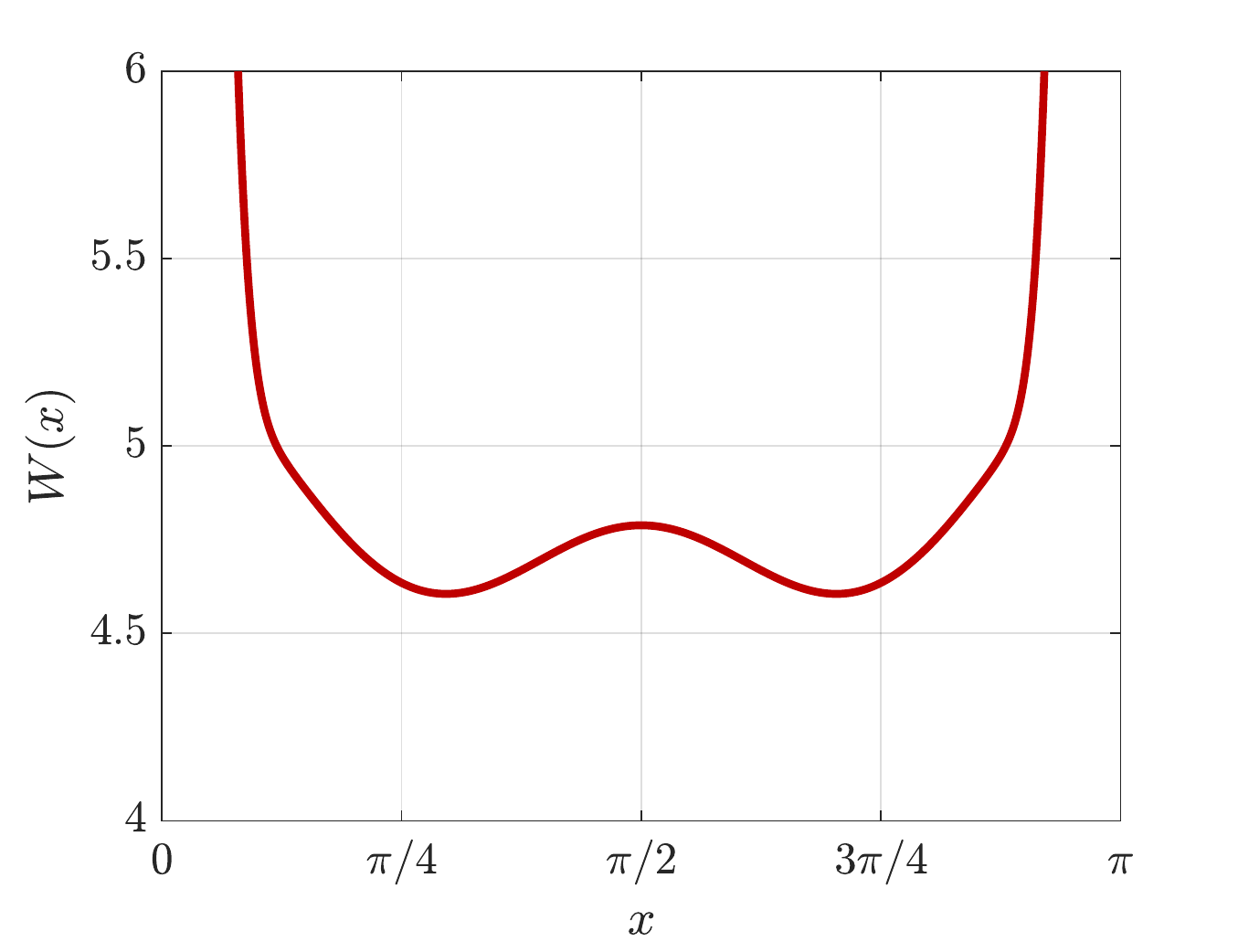}
\includegraphics[scale=0.5]{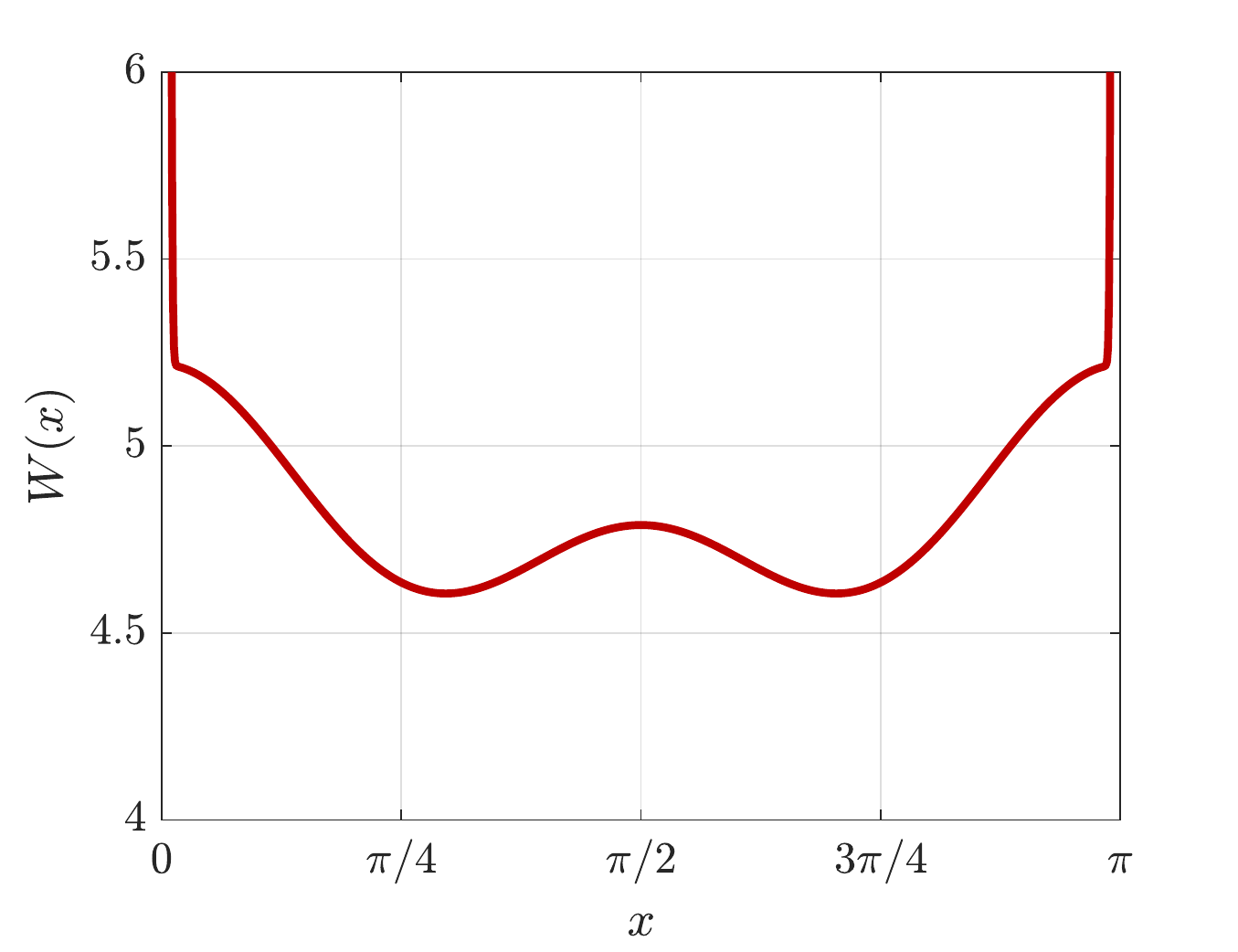}
\caption{Double-well potential \eqref{eqn:pot-well}-\eqref{eqn:mollifier} 
for different values of $\theta$.
It is seen that as $\theta\rightarrow 0$, the potential
barrier at $x=0$ and $x=\pi$ becomes infinitely high. 
This is identical to the well-known homogeneous boundary conditions 
for particles trapped in a box.
}
\label{fig:wells-schrodinger}
\end{figure}

\begin{figure}[t]
\centering
\centerline{\hspace{2.3cm}
\footnotesize 
$t=0$ \hspace{2.75cm}
$t=2.5$    \hspace{2.75cm}
$t=5$}
\centerline{\line(1,0){420}}
\vspace{0.1cm}
{\rotatebox{90}{\hspace{2.2cm}\rotatebox{-90}{
\hspace{0.1cm}
\footnotesize$p(x_1,x_2,t)$\hspace{0.6cm}}}}
\includegraphics[scale=0.27]{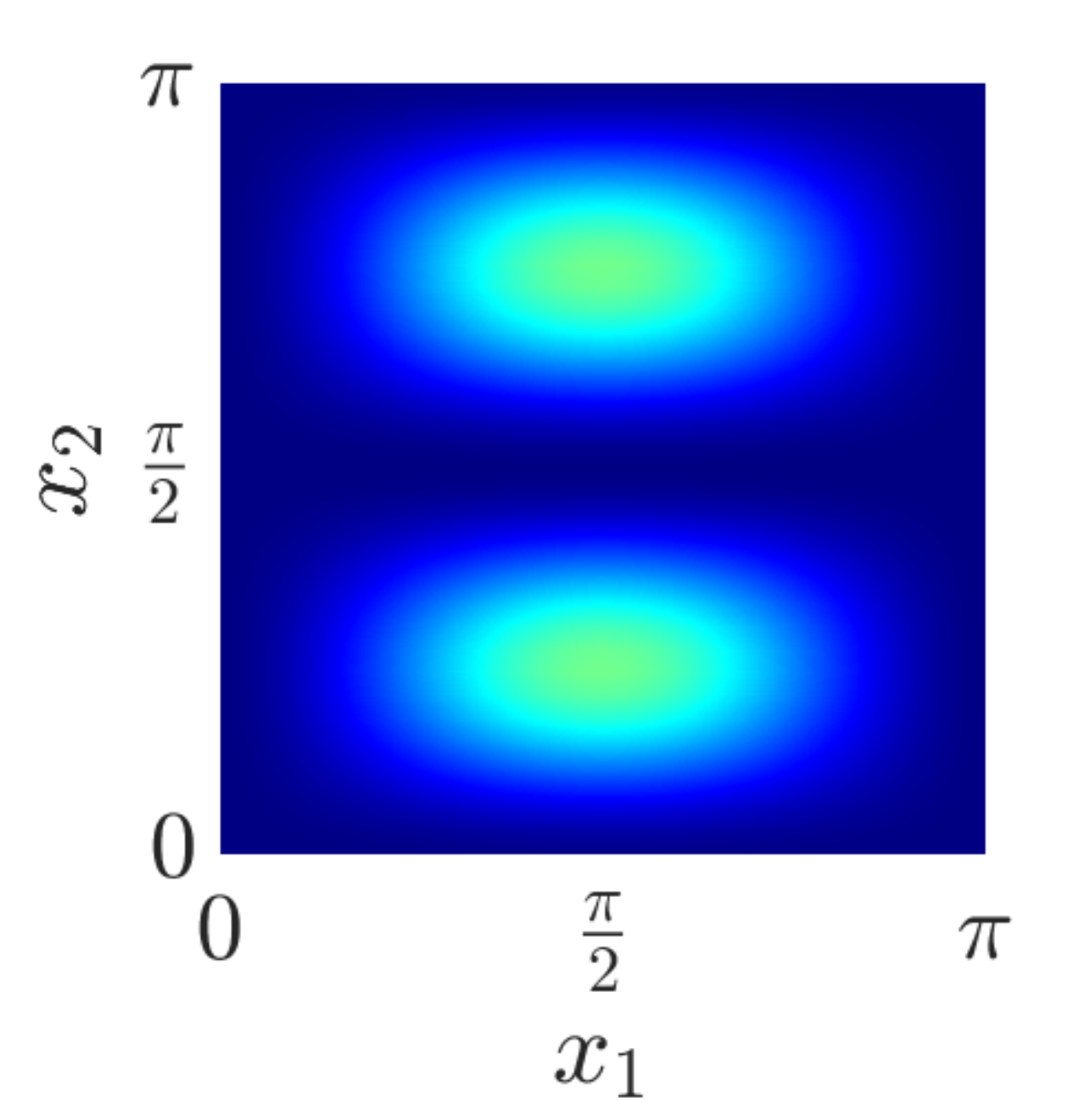}
\includegraphics[scale=0.27]{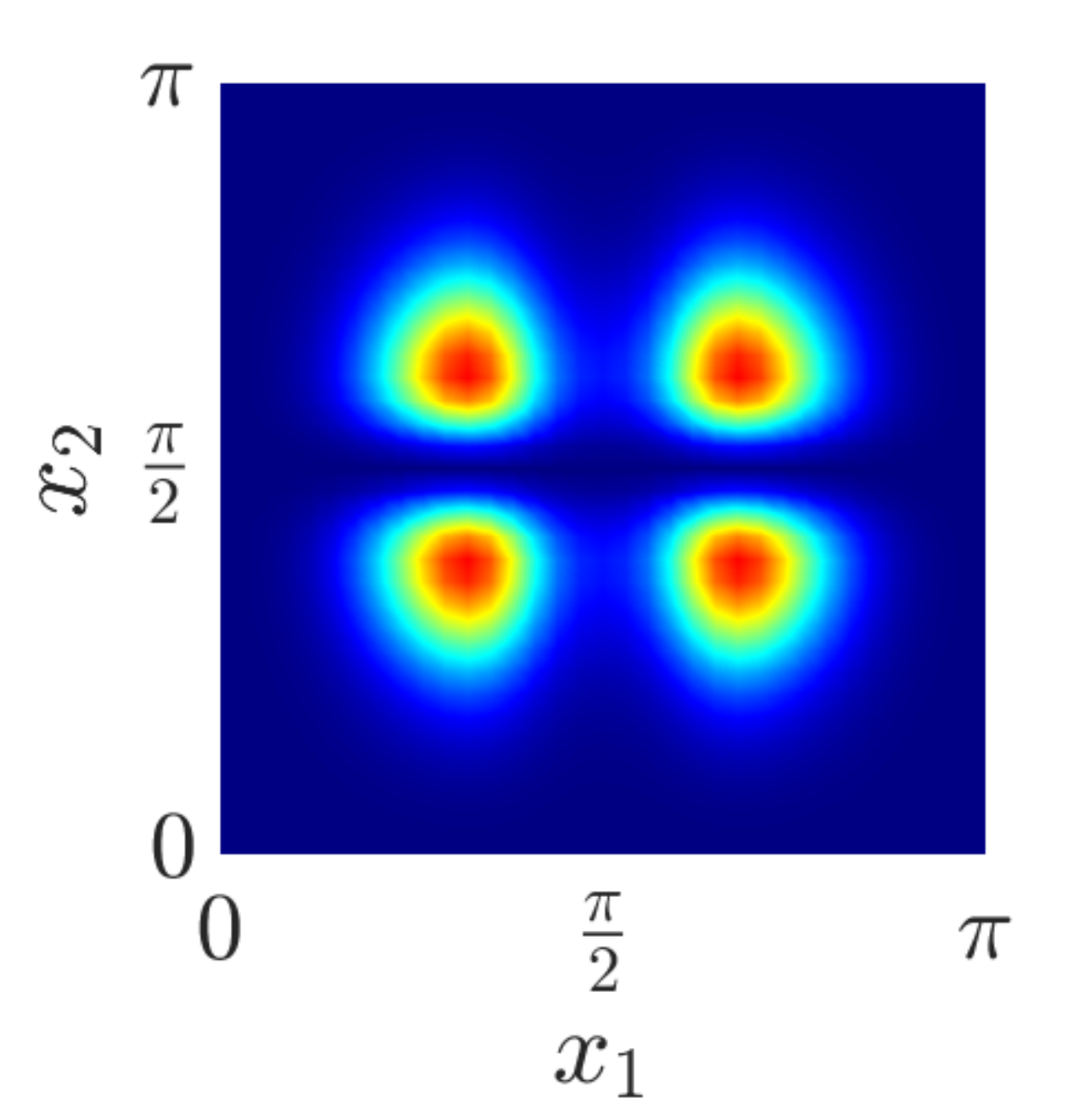}
\includegraphics[scale=0.27]{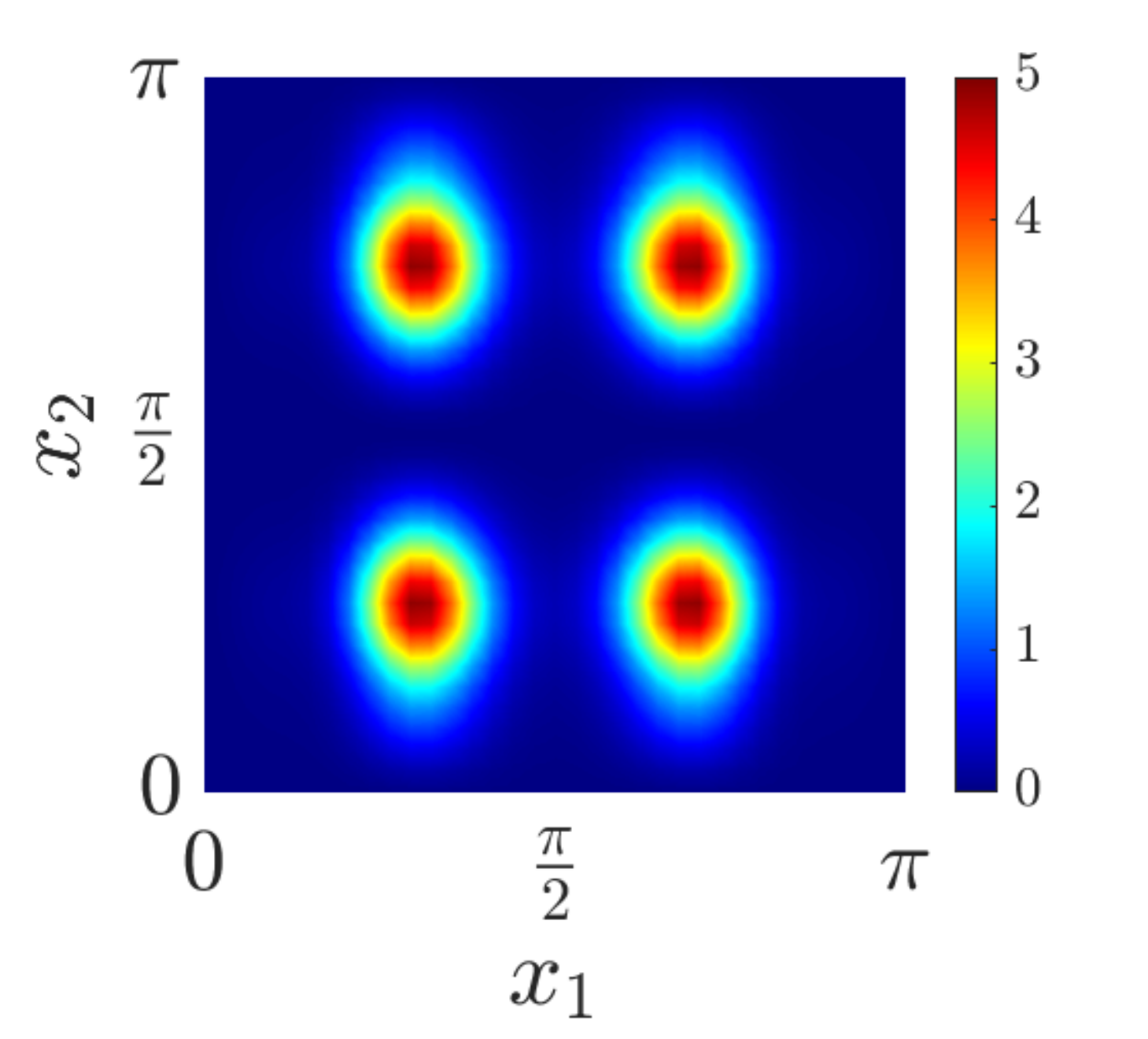}\\
{\rotatebox{90}{\hspace{2.2cm}\rotatebox{-90}{
\footnotesize$p(x_3,x_4,t)$\hspace{0.7cm}}}}
\includegraphics[scale=0.27]{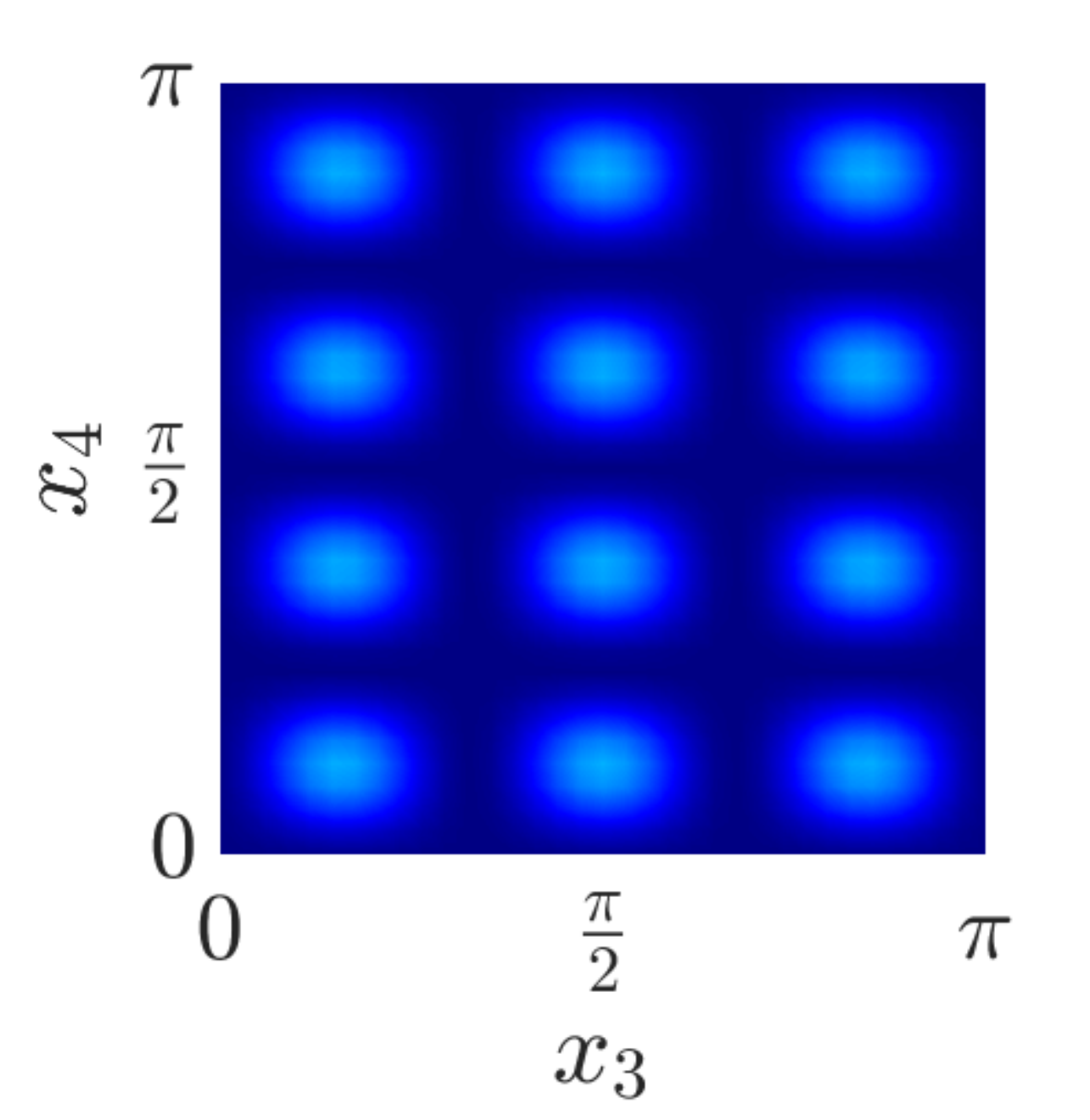}
\includegraphics[scale=0.27]{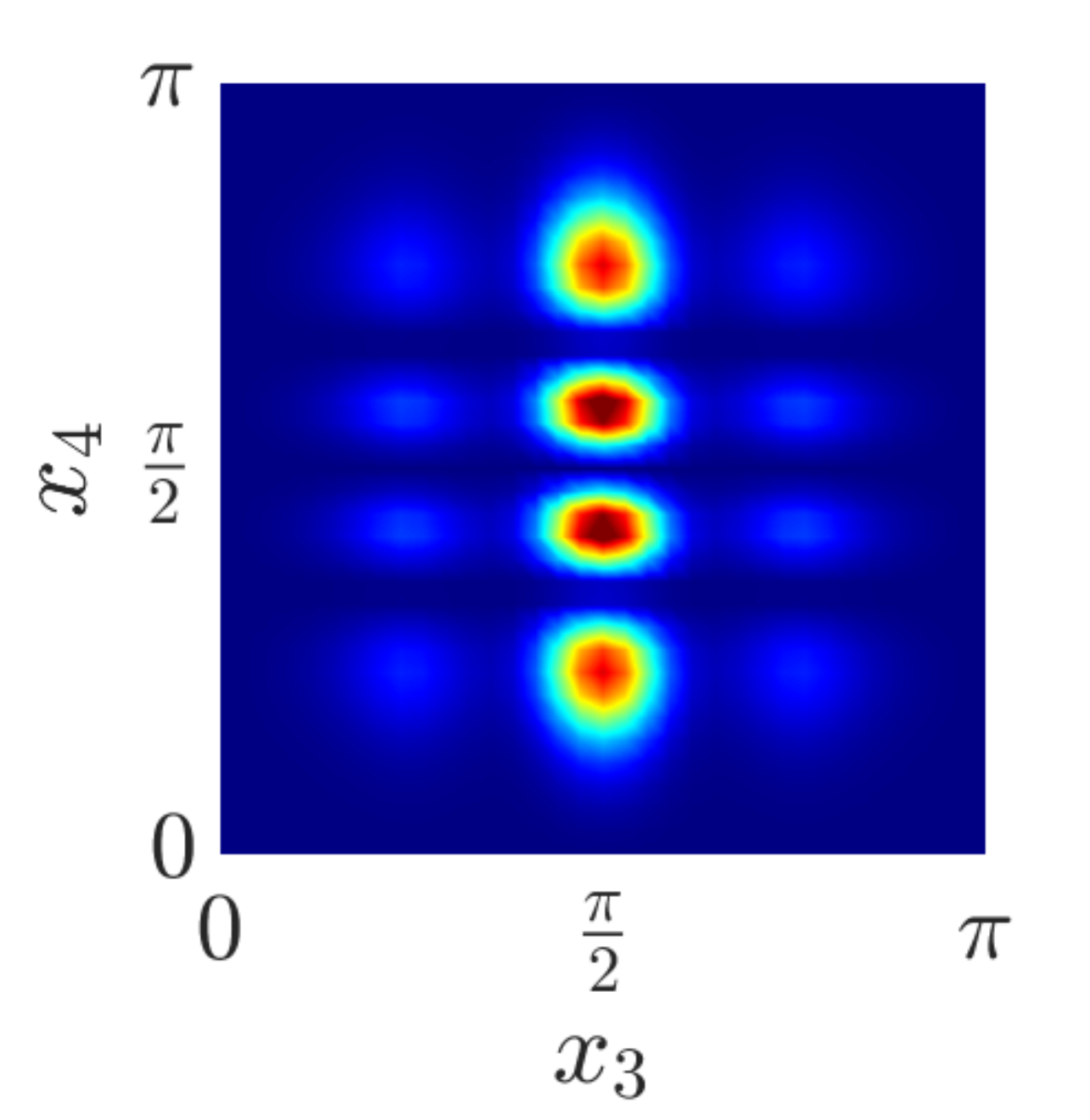}
\includegraphics[scale=0.27]{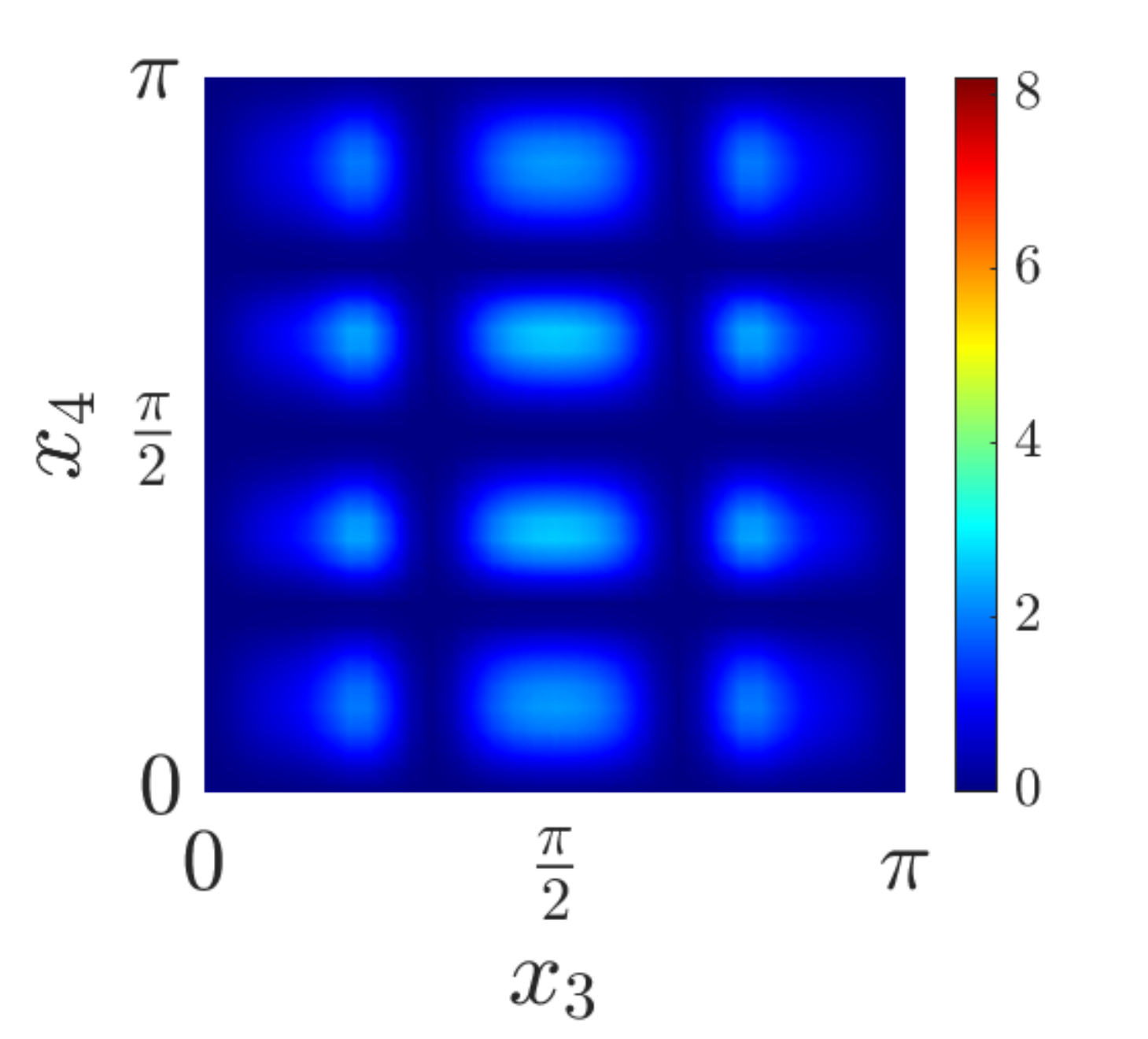}\\
{\rotatebox{90}{\hspace{2.2cm}\rotatebox{-90}{
\footnotesize$p(x_5,x_6,t)$\hspace{0.7cm}}}}
\includegraphics[scale=0.27]{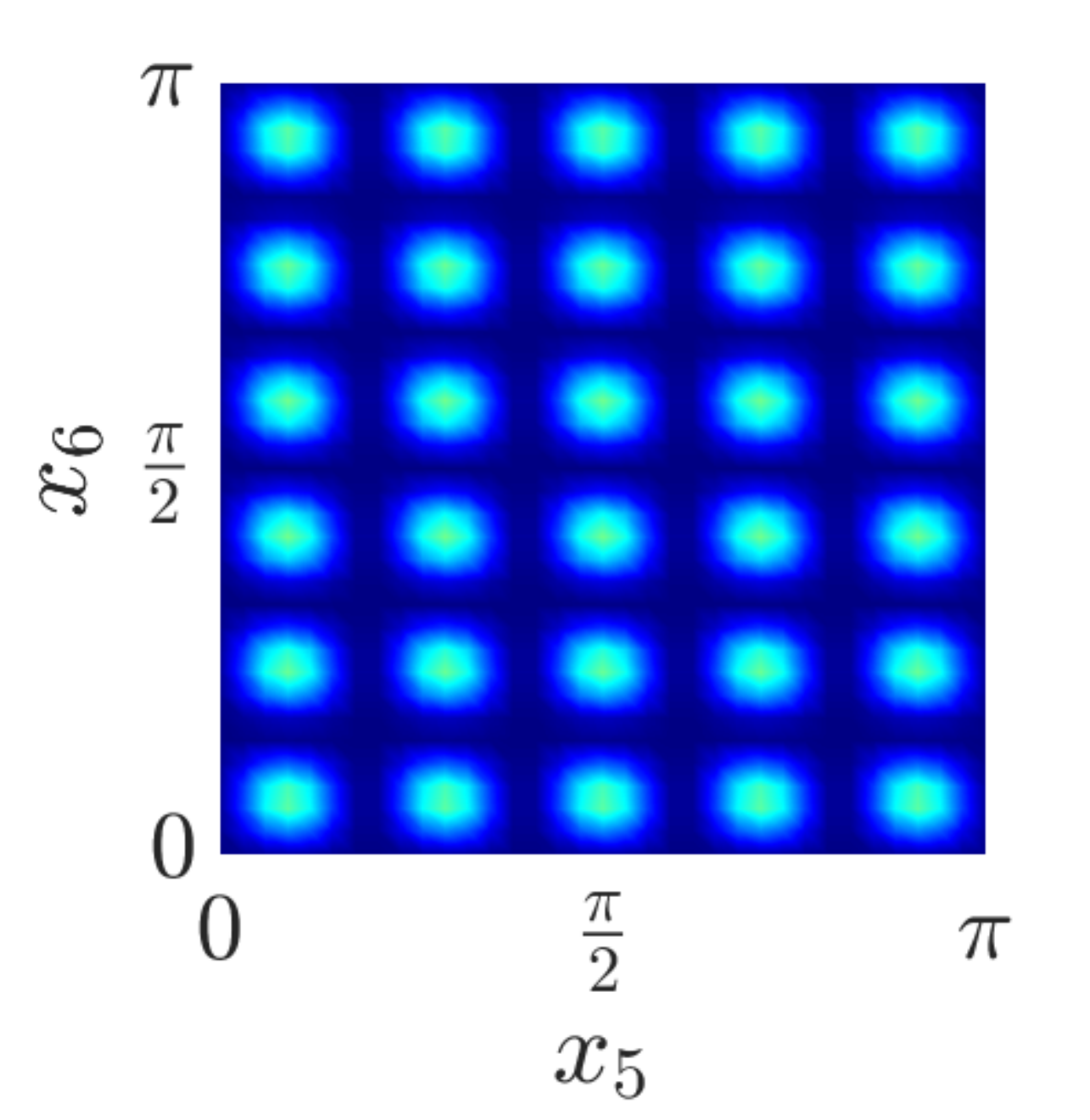}
\includegraphics[scale=0.27]{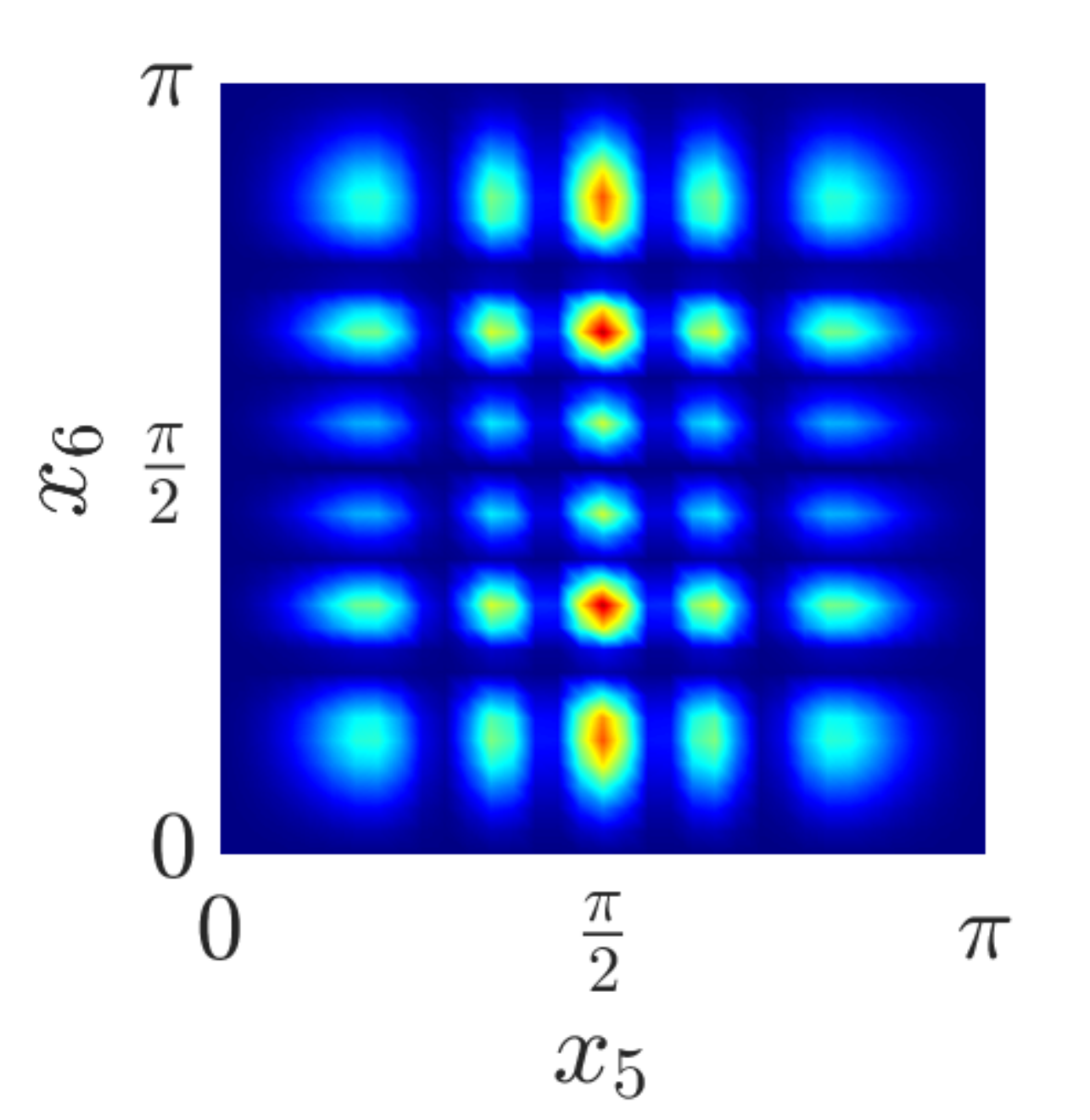}
\includegraphics[scale=0.27]{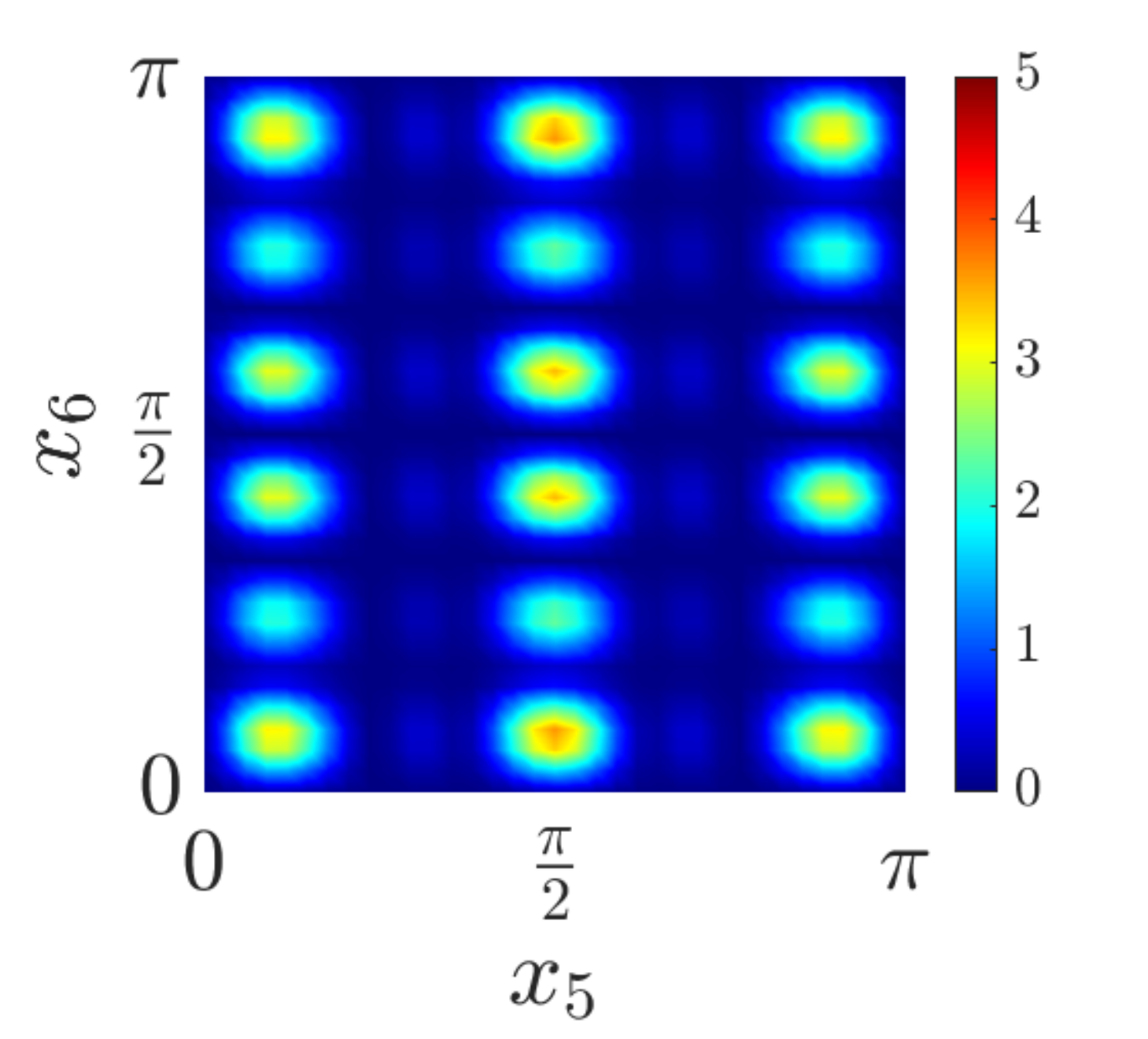}\\
\caption{Marginal probability density functions representing particle positions
generated by the nonlinear Schr\"odinger equation \eqref{eqn:nlse} with 
$\varepsilon=10^{-4}$, interaction potential \eqref{eqn:pot-well} 
and initial condition \eqref{ICS}.}
\label{fig:time-snapshot-plot-schrodinger}
\end{figure}
\begin{figure}[t]
\centerline{\footnotesize (a)\hspace{6.5cm}(b)}
\centering
\includegraphics[scale=0.5]{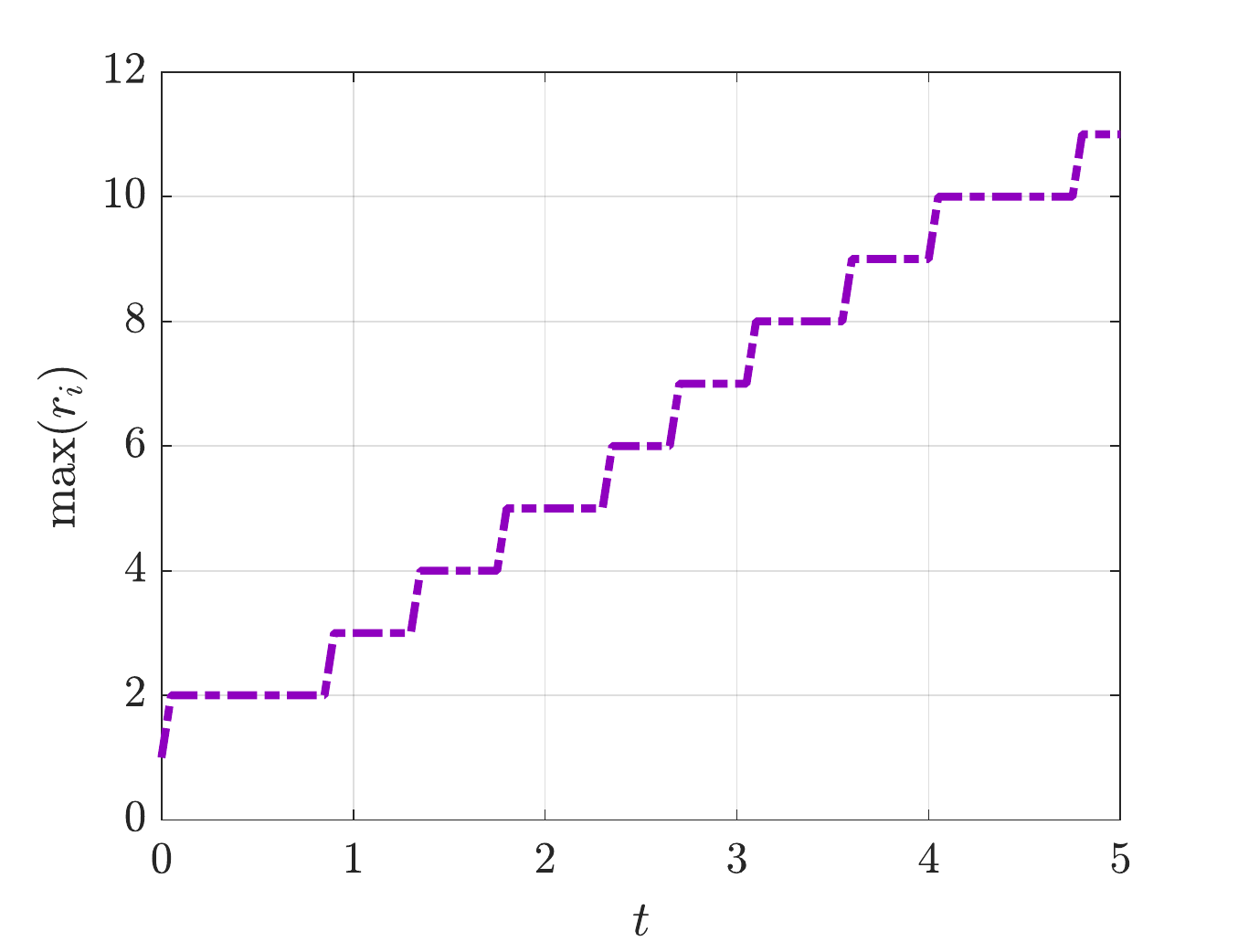}
\includegraphics[scale=0.5]{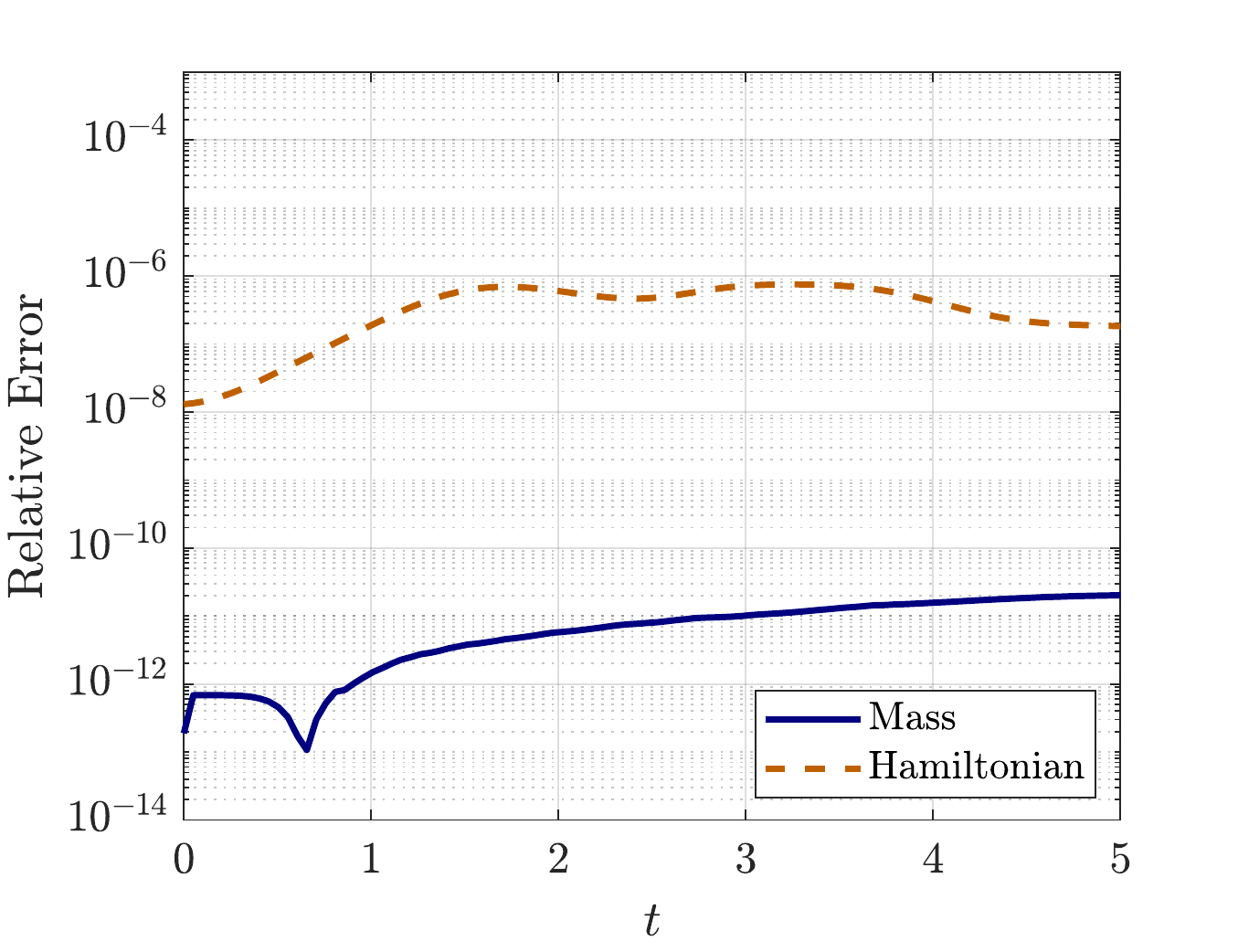}
\caption{(a) Maximum tensor rank versus time, and (b) relative error in the 
the solution mass and Hamiltonian \eqref{Ham} for 
nonlinear Schr\"odinger equation \eqref{eqn:nlse} in dimension $d=6$, with 
$\varepsilon=10^{-4}$, interaction potential \eqref{eqn:pot-well} 
and initial condition \eqref{ICS}.}
\label{fig:rank-schrodinger}
\end{figure}
\noindent
The nonlinear Schr\"odinger equation  
is complex-valued PDE whose main applications are 
wave propagation in nonlinear optical fibers, 
and Bose-Einstein condensates \cite{PhysRevLett.86.2353,pdefind}. 
The equation can be written as\footnote{As is well-known, the nonlinear 
Schr\"odinger  equation \eqref{eqn:nlse} is a Hamiltonian PDE which can 
be derived as a stationary point of the energy density (Hamilton's functional)
\begin{equation}
H({ \phi})
=
\int_{\Omega}\left(
\frac{1}{4}\|\nabla \phi\|^2
+\frac{1}{2}V({\bm x})|{\phi }|^2
+\frac{\varepsilon}{4}
|\phi|^4 \right)d{\bm x}.
\label{Ham}
\end{equation}} 
\begin{equation}
\label{eqn:nlse}
\frac{\partial \phi(\bm x,t)}{\partial t}
=
\frac{i}{2}\Delta\phi({\bm x},t) -
iV({\bm x})\phi(\bm x,t) -
	i\varepsilon|{\phi }({\bm x},t)|^2
	{\phi }({\bm x},t).
\end{equation}
where $V(\bm x)$ is the particle interaction  potential. 
In our example, we consider 6 particles
trapped on a line segment in the presence
of a double-well potential defined as
\begin{equation}
\label{eqn:pot-well}
V({\bm x})=
\sum_{k=1}^{6} W(x_k),
\qquad
W(x_k)= \left [
	1+e^{\cos(x_k)^2}+\frac{3}{4}
	\left ( 1+e^{\sin(x_k)^2} \right )
	\right ] \eta_{\theta}(x_k).
\end{equation}
Here, $W(x_{k})$ is a potential with barriers at $x_k=0$
and $x_k=\pi$ (see Figure \ref{fig:wells-schrodinger}).
The function $\eta_{\theta}(x_i)$ is a mollifier
which converges weakly to $1+\delta(x_i)+\delta(x_i-\pi)$
as $\theta\rightarrow 0$. One such mollifier is
\begin{equation}
\label{eqn:mollifier}
\eta_{\theta}(x)
=1+\frac{1}{\sqrt{2\pi}\theta}\left(
\exp\left [{-\frac{x^2}{2\theta^2}}\right]
+
\exp\left [{-\frac{(x-\pi)^2}{2\theta^2}}\right ]
\right ).
\end{equation}
As $\theta\rightarrow 0$,
the weak limit of 
$\eta_{\theta}$
translates to
zero Dirichlet boundary conditions
on the domain $\Omega = [0,\pi]^6$.
The Dirichlet conditions naturally allow us to
use a discrete sine transform to compute
the Laplacian's differentiation matrices.
We discretize the domain $\Omega$ on a
uniform grid with 35 points per dimension.
This gives us a tensor with $35^6=838265625$
entries, or 14.7 Gigabytes per temporal solution 
snapshot if we store the uncompressed
tensor in a double precision 
IEEE 754 floating point format.
We choose a product of pure states for
our initial condition, i.e., 
\begin{equation}
\label{ICS}
\phi({\bm x},0) =
\prod_{k=1}^6 \frac{6^{1/6}4k}{2k\pi-\sin(2\pi k)} \sin(kx_k),
\end{equation}
The normalizing constant $ 6^{1/6}4k/(2k\pi-\sin(2\pi k))$ guarantees that 
the wavefunction has an initial mass of 6 particles.
We now apply an operator splitting method to solve \eqref{eqn:nlse}.
The linear components are all tangential to the tensor manifold
${\cal H}_{\bm r}$, and have a physical
interpretation. Specifically,
\begin{equation}
\frac{\partial g_k}{\partial t}=
\frac{i}{2}\frac{\partial^2 g_k}{\partial x_k^2}
-iW(x_k)g_k,
\quad \quad k=1,\dots,6,
\end{equation}
is a sequence of one-dimensional linear Schr\"odinger
equations. The non-tangential part reduces to
\begin{equation}
\frac{d u}{d t}=
i\varepsilon |u|^2 u,
\end{equation}
which may be interpreted as an ODE describing
all pointwise interactions of the particles.
Here we set $\varepsilon=10^{-4}$ to model
weak iteractions. 
Clearly, the linear terms in \eqref{eqn:nlse} have purely 
imaginary eigenvalues. Therefore to integrate the semi-discrete
form of \eqref{eqn:nlse} 
in time we need a numerical scheme that has the imaginary
axis within its 
stability region.
Since implicit step-truncation Euler method introduces 
a significant damping, thereby exacerbating inaccuracy due to
discrete time stepping,  we apply the implicit step-truncation
midpoint method. For this problem, we set 
$\Delta t = 5\times 10^{-2}$ and the tensor truncation error 
to be constant in time at $100\Delta t^{3}$ to maintain 
second-order consistency.
Tolerance of the inexact Newton method was set
to $5\times 10^{-5}$ and the HT/TT-GMRES relative
error to $5\times \eta = 10^{-4}$. 
In Figure \ref{fig:time-snapshot-plot-schrodinger}
we plot the time-dependent marginal probability density 
functions for the joint position variables $(x_k,x_{k+1})$, $k=1,3,5$.
Such probability densities are defined as
\begin{equation}
p(x_1,x_{2},t)
=
\frac{1}{6}
\int_{[0,1]^4}
\phi^*(\bm x,t)\phi(\bm x,t)
{d}x_3 {d}x_4
{d}x_5 {d}x_6,
\end{equation}
and analogously for 
$p(x_3,x_{4},t)$ and 
$p(x_5,x_{6},t)$.
It is seen that the lower energy pure states (position variables
$(x_1,x_{2})$) quickly get trapped in the two wells,
oscillating at their bottoms. Interestingly,
at $t=2.5$ it appears that particle $x_3$
is most likely to be observed in between the two
wells whenever the particle $x_4$ is in a well bottom.

%

In Figure \ref{fig:rank-schrodinger}(a), we plot
the rank over time for this problem.
The rank also has physical meaning. A higher
rank HT tensor is equivalent to a wavefunction
with many entangled states, regardless of
which $L^2(\Omega)$ basis we choose. 
In the example discussed in this section, the particles 
interacting over time monotonically increase the rank. 
We emphasize that the nonlinearity in \eqref{eqn:nlse} poses a
significant challenge to tensor methods. In fact, in a 
single application of the function
$
i\varepsilon |u|^2 u,
$
we may end up tripling the rank.
This can be mitigated somewhat by using
the approximate element-wise tensor multiplication routine.
Even so, if the inexact Newton Method requires many dozens 
of iterations to halt, the rank may grow very rapidly in 
a single time step, causing a slowing due to large array 
storage. This problem is particularly apparent 
when $\varepsilon \approx 1$ or if $\Delta t$ is made 
significantly smaller, e.g. $\Delta t = 10^{-4}$.
A more effective way of evaluating nonlinear
functions on tensors decompositions would
certainly mitigate this issue.
In Figure \ref{fig:rank-schrodinger}(b),
we plot the relative error of the solution mass 
and the Hamiltonian \eqref{Ham}
over time. The relative errors hover around $10^{-11}$
and $10^{-6}$, respectively. It is remarkable that even though
the additional tensor truncation done after the inner
loop of the implicit solver in principle destroys the symplectic 
properties of the midpoint method, the mass 
and Hamiltonian are still preserved with high accuracy.

\appendix

\section{Solving algebraic equations on tensor manifolds}
\label{apndx:ht-tt-newton}
\noindent
We have seen in section \ref{sec:implicit-st} that 
a large class of implicit step-truncation methods can be 
equivalently formulated as a root-finding problem for a
nonlinear system algebraic equations of the form 
\begin{equation}
{\bm H}({\bm f}) = {\bm 0}
\label{nonlinEq}
\end{equation}
at each time step. In this appendix we develop 
numerical algorithms to compute an approximate 
solution of \eqref{nonlinEq} on a 
tensor manifold $\mathcal{H}_{\bm r}$ with 
a given rank $\bm r$. In other words, 
we are interested in finding  
${\bm f}\in \mathcal{H}_{\bm r}$ that solves 
\eqref{nonlinEq} with controlled accuracy. 
%
%
%
To this end, we combine the inexact inexact 
Newton method \cite[Theorem 2.3 and Corollary 3.5]{dembo1982inexact} 
with the TT-GMRES linear solver proposed in \cite{dolgov2013ttgmres}.


\begin{theorem}[Inexact Newton method \cite{dembo1982inexact}]
\label{thm:dembo-inexact}
Let ${\bm H}:{\mathbb R}^{N}\rightarrow{\mathbb R}^{N}$
be continuously differentiable in a neighborhood 
of a zero $\bm f^*$, and suppose that the Jacobian 
of $\bm H$, i.e., $\bm J_{\bm H}(\bm f)=\partial \bm H(\bm f)/\partial \bm f$, 
is invertible at  ${\bm f}^{*}$. Given ${\bm f}^{[0]}\in{\mathbb R}^{N}$,
consider the sequence
\begin{equation}
{\bm f}^{[j+1]} = {\bm f}^{[j]} + {\bm s}^{[j]},\qquad j=0,1,\ldots
\end{equation}
where each ${\bm s}^{[j]}$ solves the Newton iteration
up to relative error $\eta^{[j]}$, i.e., it satisfies
\begin{equation}
\left \|\bm J_{\bm H}\left({\bm f}^{[j]}\right)
{\bm s}^{[j]}
+ {\bm H}\left ({\bm f}^{[j]}\right )
\right \|
\leq 
\left \|
{\bm H}\left ({\bm f}^{[j]}\right )
\right \|\eta^{[j]}.
\label{boundsys}
\end{equation}
If $\eta^{[j]} < 1$ for all $j$, then there exists 
$\varepsilon > 0$ so that for any initial guess satisfying
$\left \| {\bm f}^{[0]} - {\bm f}^*\right \| <\varepsilon$
the sequence $\{{{\bm f}^{[j]}}\}$ converges linearly
to ${{\bm f}^{[*]}}$. If $\eta^{[j]}\rightarrow 0$ as 
$j\rightarrow \infty$,
then the convergence speed is superlinear.
\end{theorem}

\noindent
The next question is how to compute an approximate 
solution of the linear system 
\begin{equation}
\bm J_{\bm H}\left({\bm f}^{[j]}\right)
{\bm s}^{[j]}
=- {\bm H}\left ({\bm f}^{[j]}\right )
\end{equation}
satisfying the bound \eqref{boundsys}, without 
inverting the Jacobian $\bm J_{\bm H}\left({\bm f}^{[j]}\right)$ 
and assuming that ${\bm s}^{[j]}\in \mathcal{H}_{\bm r_j}$, i.e., 
that ${\bm s}^{[j]}$ is a tensor with 
rank $\bm r_j$. To this end we utilize the relaxed HT/TT-GMRES 
method discussed in \cite{dolgov2013ttgmres}
HT/TT-GMRES is an adapted tensor-structured generalized 
minimal residual (GMRES) method to solve 
linear systems in a tensor format. 
The solver employs an indirect accuracy check and a 
stagnation restart check in its halting criterion 
which we summarize in the following Lemma.

\begin{lemma}[Accuracy of HT/TT-GMRES \cite{dolgov2013ttgmres}]
\label{lemma:ht-tt-gmres}
Let ${\bm J}{\bm f} = {\bm b}$ be a linear system
where ${\bm f}$, ${\bm b}$ are tensors in HT
or TT format, and $\bm J$ is a bounded 
linear operator on $\bm f$. Let 
$\{{\bm f}^{[0]},{\bm f}^{[1]},\ldots\}$ be the 
sequence of approximate solutions generated by 
HT/TT-GMRES algorithm in \cite{dolgov2013ttgmres}, and 
$\varepsilon > 0$ be the stopping tolerance for 
the iterations. Then
\begin{equation}
\label{eqn:gmres-error-1}
\left \|
{\bm J}{\bm f^{[j]}} - {\bm b}
\right \|
\leq m
\| {\bm J}\|
\| {\bm J}^{-1}\|
\| {\bm b}\|\varepsilon,
\end{equation}
where $m$ is the number of Krylov 
iterations performed before restart. Similarly,
the distance between ${\bm f^{[j]}}$ and the 
exact solution ${\bm f}$ can be bounded as
\begin{equation}
\label{eqn:gmres-error-2}
\left \|
{\bm f^{[j]}} - {\bm f}
\right \|
\leq m
\| {\bm J}\|
\| {\bm J}^{-1}\|^{2}
\| {\bm b}\|\varepsilon.
\end{equation}
\end{lemma}

\noindent
We can now combine the HT/TT-GMRES linear solver 
with the inexact Newton method, to obtain an 
algorithm that allows us to solve nonlinear algebraic equations 
of the form \eqref{nonlinEq} on a tensor manifold. 

\begin{theorem}[HT/TT Newton method]
\label{thm:ht-tt-newton}
Let
${\bm H}:{\mathbb R}^{n_1\times n_2 \times
\cdots \times n_d}\rightarrow
{\mathbb R}^{n_1\times n_2 \times
\cdots \times n_d}$
be a continuously differentiable nonlinear map which 
operates on HT or TT tensor formats, and let 
${\bm f}^*$ be a zero of $\bm H$. 
Suppose that the Jacobian of $\bm H$, denoted as $\bm J_{\bm H}(\bm f)$,
is invertible at $\bm f^*$. Given an initial guess ${\bm f}^{[0]}$,
consider the iteration
\begin{equation}
{\bm f}^{[j+1]} = {\mathfrak{T}_{\bm r}}\left (
 {\bm {f}}^{[j]} + {\bm s}^{[j]}
\right )
\end{equation}
where ${\bm s}^{[j]}$ is the HT/TT-GMRES solution of 
$\bm J_{\bm H}\left(\bm f^{[j]}\right){\bm s}^{[j]}=
-{\bm H}\left ({\bm f}^{[j]}\right )$ satisfying
\begin{equation}
\left \|
\bm J_{\bm H}\left(\bm f^{[j]}\right){\bm s}^{[j]}
+ {\bm H}\left ({\bm f}^{[j]}\right )
\right \|
\leq 
{\frac{1}{2}}
\left \|
{\bm H}\left ({\bm f}^{[j]}\right )
\right \|\eta^{[j]},
\end{equation}
where $\eta^{[j]}$ is the relative error, which can 
be any value in the range $0\leq \eta^{[j]} < 1$.
Then the Newton iteration converges linearly so long as the
rank ${\bm r}$ of the truncation operator 
${\mathfrak{T}_{\bm r}}$ is chosen to satisfy
\begin{equation}
\left \|
{\mathfrak{T}_{\bm r}}\left (
 {\bm {f}}^{[j]} + {\bm s}^{[j]}
\right ) - {\bm {f}}^{[j]} - {\bm s}^{[j]}
\right \|
\leq 
\frac{\left \|{\bm H}\left ({\bm f}^{[j]}
\right)\right\|}{2\left\|{\bm J}^{[j]}\right\|}\eta^{[j]}.
\end{equation}
\end{theorem}	

\noindent

\begin{proof}
Let ${\bm f}^{[0]}$ be an initial guess.
Consider the sequence
\begin{align}
{\bm {\tilde f}}^{[j+1]} = {\bm {f}}^{[j]} + {\bm s}^{[j]}\qquad 
j=0,1,\ldots
\end{align}
where ${\bm s}^{[j]}$ is the HT/TT-GMRES solution of 
$\bm J_{\bm H}\left(\bm f^{[j]}\right){\bm s}^{[j]}=
-{\bm H}\left ({\bm f}^{[j]}\right )$ obtained with 
tolerance
\begin{equation}
\varepsilon_j < \frac{\eta^{[j]}}{2 m 
\left\| {\bm J}_{\bm H} \left( \bm f^{[j]} \right)\right\|
\left\| {\bm J}^{-1}_{\bm H}\left(\bm f^{[j]}\right)\right\|}, 
\end{equation}
and $0\leq\eta^{[j]}<1$. The next step is to 
truncate ${\bm {\tilde f}}^{[j+1]}$ to a tensor 
\begin{equation}
{\bm f}^{[j+1]} = {\mathfrak{T}_{\bm r}({\bm {\tilde f}}^{[j+1]})}
\end{equation}
with rank $\bm r$ chosen so that
\begin{equation}
\left \| {\bm f}^{[j+1]} - {\bm {\tilde f}}^{[j+1]}\right\|
<\frac{\left \|{\bm H}\left ({\bm f}^{[j]}
\right)\right\|}{2\left\|\bm J_{\bm H}\left ({\bm f}^{[j]}
\right)\right\|}\eta^{[j]}.
\end{equation}
Then the error in solving the Newton iteration this step is
\begin{equation}
{\bm {\tilde r}}^{[j]} =
{\bm J}_{\bm H}\left(\bm f^{[j]}\right)\left[
{\bm f}^{[j+1]} -{\bm{\tilde f}}^{[j+1]}
+ {\bm s}^{j}\right ]
+{\bm H}\left ({\bm f}^{[j]}
\right),
\end{equation}
which we constructed to satisfy the bound
\begin{equation}
\left \| {\bm {\tilde r}}^{[j]}
\right \|
\leq
\left \|
{{\bm J}_{\bm H}\left(\bm f^{[j]}\right)}{\bm s}^{[j]}
+{\bm H}\left ({\bm f}^{[j]}\right )
\right \|
+
\left \|
{\bm J}_{\bm H}\left(\bm f^{[j]}\right)\left [
{\bm f}^{[j+1]} -{\bm{\tilde f}}^{[j+1]}
\right]\right \|
<\left \|
{\bm H}\left ({\bm f}^{[j]}\right )
\right \|\eta^{[j]}.
\end{equation}
Hence, the inexact HT/TT-GMRES Newton method converges 
linearly. This completes the proof.
\begin{flushright}
\qed
\end{flushright}
\end{proof}

\subsubsection*{An Example}

\begin{figure}[t]
\centerline{ 
\includegraphics[scale=0.57]{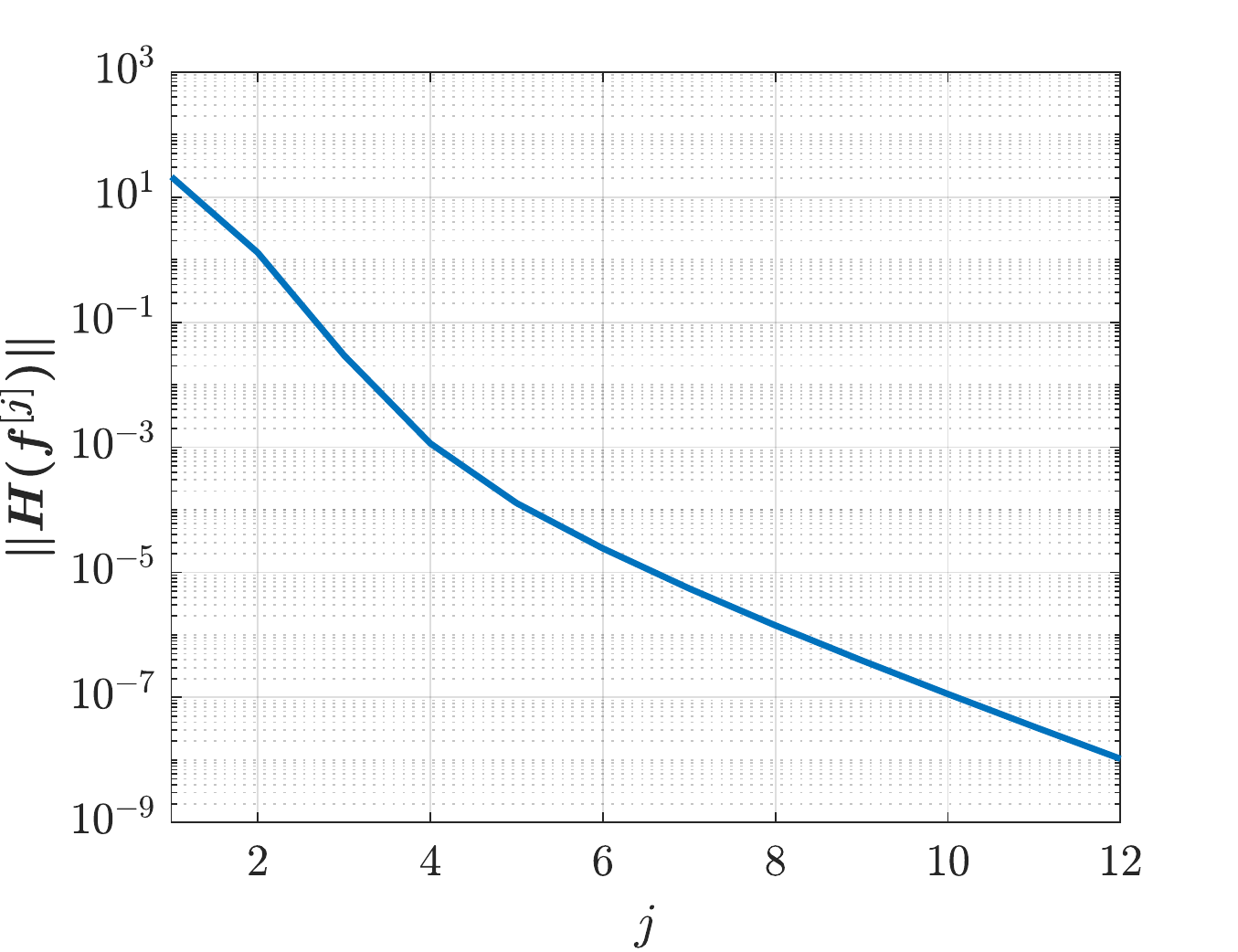}
}
\caption{
Error versus iteration count of inexact Newton's method
in the HT format.
}
\label{fig:ac-newton-example-err}
\end{figure}
\noindent
Let us provide a brief numerical demonstration of the inexact 
Newton method with HT/TT-GMRES iteration. To this
end, consider the cubic function
\begin{equation}
\bm H(\bm f)=  1.5{\bm f} + 0.5 {\bm f}_{0} + 0.125({\bm f} + {\bm f}_0)^3
\end{equation}
where  all  products are computed using the approximate
element-wise Hadamard tensor product with accuracies set 
to ${10}^{-12}$, and ${\bm f}_0$ is a given tensor which corresponds 
to the initial condition used in section \ref{sec:ac-eqn-example} 
truncated to an absolute tolerance of $10^{-4}$.
The Jacobian operator is easily obtained as 
\begin{equation}
{\bm J}_{\bm H}({\bm f}){\bm s}
= 1.5{\bm s} +0.375({\bm f} + {\bm f}_0)^2{\bm s}.
\end{equation}
We set the relative error of the matrix inverse to be $\eta = 10^{-3}$.
In Figure \ref{fig:ac-newton-example-err} we plot the results 
of the proposed inexact Newton method with HT/TT-GMRES 
iterations. We see that that the target tolerance 
of $2.2\times 10^{-8}$ is hit in just 12 iterations.

\bibliography{ht-implicit}
\bibliographystyle{plain}
\end{document}